\documentclass[]{article}

\usepackage{graphicx}
\usepackage{amsmath}
\usepackage{amssymb}
\usepackage{amsthm} 
\usepackage{bm}
\usepackage{enumerate}
\usepackage{color}
\usepackage{mathdots}
\usepackage{sectsty}
\usepackage{hyperref}
\hypersetup{hidelinks}
\usepackage{tikz}
\usepackage{caption}
\usepackage{adjustbox}
\usepackage{yhmath}
\usepackage{tikz-cd}
\usepackage{mathrsfs}
\usepackage{quiver}
\sectionfont{\scshape\centering\fontsize{12}{14}\selectfont}
\subsectionfont{\scshape\fontsize{12}{14}\selectfont}
\usepackage{fancyhdr}

\newcommand\shorttitle{Six-functor formalism on rigid-analytic varieties.}
\newcommand\authors{Arun Soor}

\newenvironment{keywords}{}{}
\newtheorem{thm}{Theorem}[section]
\newtheorem{cor}[thm]{Corollary}
\newtheorem{lem}[thm]{Lemma}
\newtheorem{defn}[thm]{Definition}
\newtheorem{rmk}[thm]{Remark}
\newtheorem{example}[thm]{Example}

\newtheorem{prop}[thm]{Proposition}
\newtheorem{notations}[thm]{Notations}
\newtheorem{scholium}[thm]{Scholium}

\newtheorem{assum}[thm]{Assumptions}

\newtheorem*{claim*}{Claim}

\newcommand*{\sheafhom}{\mathcal{H}\kern -.5pt om}

\DeclareMathOperator{\indlim}{``lim''}
\DeclareMathOperator{\prolim}{``lim''}

\makeatletter
\def\keywords{\xdef\@thefnmark{}\@footnotetext}
\makeatother

\newcommand{\Addresses}{{
  \bigskip
  \footnotesize

  Arun Soor, \textsc{Department of Mathematics, University of Michigan, 4831 East Hall, 530 Church Street, Ann Arbor, MI 48109}\par\nopagebreak
  \textit{E-mail address}, Arun Soor, \texttt{\href{mailto:soor@umich.edu}{soor@umich.edu}}}}

\title{\large \bf A six-functor formalism for quasi-coherent sheaves and stratifications on rigid-analytic varieties}
\author{Arun Soor}
\date{\today}
\begin{document}
\maketitle
 \keywords{2020 \emph{Mathematics Subject Classification.} Primary 14G22; Secondary 14F10, 32C35, 32C38, 14F06, 14F08, 18N60.}%
\abstract{We develop a theory of derived rigid spaces and quasi-coherent sheaves and analytic ``stratifications" on them. Amongst other things, we obtain a six-functor formalism for these quasi-coherent sheaves and analytic stratifications. We provide evidence that the category of analytic stratifications is related to the theory of $\wideparen{\mathcal{D}}$-modules introduced by Ardakov--Wadsley.}
\tableofcontents
\section{Introduction}
The purpose of this article is to produce a six-functor formalism for quasi-coherent sheaves and analytic stratifications or $\wideparen{\mathcal{D}}$-modules on derived rigid-analytic varieties. 

The idea of a \emph{six-functor formalism} is that we start with some ($\infty$-)category $\mathscr{C}$ of geometric objects $X$, admitting all fiber products, and associate to each $X \in \mathscr{C}$ a closed symmetric monoidal ($\infty$-)category $(Q(X), \otimes)$ in a manner which satisfies an enormous number of functorial properties. That is, to each morphism $f$ of $\mathscr{C}$ we associate a \emph{pullback} and a \emph{pushforwards} functor, and for morphisms in a certain special class $E$ we associate a \emph{compactly supported pushforwards} and an \emph{exceptional pullback}. These assignments should all be compatible with composition, and satisfy the \emph{base-change} and \emph{projection} formulas\footnote{For a more detailed exposition we direct the reader to \S\ref{subsec:sixf1} and \cite{mann_p-adic_2022,ScholzeSixFunctors}.}. The idea of the ``yoga of the six operations" was initiated by Grothendieck in his study of the functorial properties of $\ell$-adic cohomology. Besides Grothendieck, the yoga of the six operations also found early adopters in the theory of $\mathcal{D}$-modules on smooth schemes over algebraically-closed fields of characteristic $0$, and complex-analytic varieties, in the work of Bernstein \cite{BernsteinDmodules} and Mebkhout \cite{MebkhoutSixOperations}, and Kashiwara--Schapira \cite{SheavesOnManifolds}, respectively. Without much explanation, I will say that in these approaches one cuts down the category of coherent $\mathcal{D}$-modules, to obtain the category of \emph{holonomic $\mathcal{D}$-modules}, which is stable under the six operations, and which still contains all vector bundles with connection.

In recent years, due to the breakthrough works of Liu--Zheng \cite{liu_enhanced_2017}, Gaitsgory--Rozenblyum \cite{GaitsgoryStudy1}, and Mann \cite{mann_p-adic_2022}, our collective understanding of the six operations has developed from a \emph{yoga} to a rigorously defined notion.  Roughly speaking, on associates to the pair $(\mathscr{C},E)$ the \emph{category of correspondences} $\operatorname{Corr}(\mathscr{C},E)$, whose objects are the same as $\mathscr{C}$ and whose morphisms are given by spans. A six-functor formalism is then a lax-symmetric monoidal functor $Q: \operatorname{Corr}(\mathscr{C},E) \to \mathsf{Cat}_\infty$, where the latter is the $\infty$-category of $\infty$-categories\footnote{Once again, for a more detailed exposition we direct the reader to \S\ref{subsec:sixf1} and \cite{mann_p-adic_2022,ScholzeSixFunctors}.}. This amazingly succinct definition provides a viable way to manipulate six-functor formalisms and produce new ones out of old ones. As one of the myriad of applications, one can obtain streamlined proofs of Poincar\'e duality \cite{zavyalov_poincare_2023}. One can also use six-functor formalisms to understand what the ``correct" definitions of some objects or functors should be. For instance, one might understand ULA-sheaves as ``$f$-smooth objects" \cite[Lecture VI]{ScholzeSixFunctors}. One could argue that the formal properties encoded in a six-functor formalism are so good, that in some way the existence of a six-functor formalism becomes an organising principle. 

With the rigorous definition of a six-functor formalism in hand, Gaitsgory--Rozenblyum \cite{GaitsgoryStudy1} and Scholze \cite[Lecture VIII]{ScholzeSixFunctors} have developed another approach to the six-functor formalism for algebraic $\mathcal{D}$-modules, which one might say lies at the opposite extreme\footnote{That is, they adopt the ``Grothendieck philosophy" that categories, rather than objects, should be well-behaved.} to the approach based on holonomicity. In this approach one makes extensive use of higher category theory as developed by Lurie \cite{HigherToposTheory,HigherAlgebra}, and derived algebraic geometry, and obtain a six-functor formalism for all ``quasi-coherent" $\mathcal{D}$-modules. 
The \emph{raison d'\^etre} for derived geometry in this context, is to obtain \emph{base-change}. It may seem extreme to invoke derived geometry just to get base-change, but when working in this area one realises how powerful and useful the technique of base-change is. It sometimes feels like ``everything is base change". 

In our approach, we will follow in some way the path laid out by Gaitsgory--Rozenblyum, Scholze, and Camargo \cite{Camargo_deRham} to obtain a six-functor formalism on derived rigid spaces which is related to the theory of $\wideparen{\mathcal{D}}$-modules defined by Ardakov--Wadsley \cite{Dcap1}. These \emph{infinite order differential operators} are analogous to the sheaves $\mathcal{D}^\infty$ appearing in complex-analytic geometry \cite[Chapter III.4]{BjorkAnalytic}. 

For instance, a similarity between this work and that of \cite{GaitsgoryStudy1,Camargo_deRham,SimpsonHodge}, is the following. Let $X$ be a (derived) rigid-analytic space. Rather than working with ``quasi-coherent $\wideparen{\mathcal{D}}_X$-modules" directly, we prefer to reintepret these objects as quasi-coherent sheaves on a prestack $X_{\mathrm{str}}$ associated to $X$, which we call the \emph{stratifying stack}. These quasi-coherent sheaves on $X_{\mathrm{str}}$ are what we call \emph{(analytic) stratifications}. We note, however, that the precise definition of $X_{\mathrm{str}}$ in this article is slightly different from, but very closely related to, the \emph{de Rham stack} used in \cite{GaitsgoryStudy1,Camargo_deRham,SimpsonHodge}. Whilst they define their de Rham stacks via a notion of reduction on a category of algebras, we prefer the homotopy-theoretic perspective, where we view stacks as certain kinds of simplicial objects\footnote{More precisely, we use that the $\infty$-category of simplicial presheaves has the structure of a \emph{model $\infty$-category} in the sense of Mazel-Gee \cite{mazel-gee_model_2015}, in which the weak equivalences are precisely those morphisms which are sent to equivalences under the geometric realization functor. The work of \emph{loc. cit.} provides a systematic way to understand how fiber products interact with geometric realization, and is therefore very useful for understanding derived stacks from the simplicial perspective.}. That is, we define $X_{\mathrm{str}}$ directly via the infinitesimal groupoid. This groupoid object is obtained from the \emph{analytic germ} of the diagonal. In smooth geometry, a similar construction appears in the work of Borisov--Kremnizer \cite[\S 3.4]{BorisovKremnizerdeRham}.

One reason to use the stratifying stack or the de Rham stack, is that the classical formulas for the six-operations as in \cite{BernsteinDmodules} take some time to appreciate, and, in the quasi-coherent setting, it is perhaps unclear that they are the right thing from the perspective of category theory. Following \cite{GaitsgoryCrystals}, we recall that differential operators can be recovered as the monad controlling codescent along $X \to X_{\mathrm{str}}$. In fact there is also a \emph{comonad of jets} which controls descent along $X \to X_{\mathrm{str}}$. By using that $(-)_{\mathrm{str}}$ is a functor, and the more familiar functoriality of quasi-coherent sheaves on prestacks, we can then chase the Barr--Beck--Lurie equivalence to obtain formulas for the six operations, in terms of jets and/or differential operators.

\subsection{The content of this paper.}
Let $K/\mathbf{Q}_p$ be a complete field extension. In order to construct the six-functor formalism for stratifications or $\mathcal{D}$-modules, it stands to reason that we should first construct the six-functor formalism for quasi-coherent sheaves on derived rigid spaces. In order to do this, we must first develop a theory of derived rigid geometry. This is what we do in \S \ref{sec:homotopical}-\ref{sec:derivedrigidgeometry}. The material of \S\ref{sec:homotopical}-\ref{sec:derivedrigidgeometry} owes an overwhelming intellectual debt to \cite{DAnG} and also takes much inspiration from \cite[\S 2]{mann_p-adic_2022}. In \cite{DAnG}, Ben-Bassat--Kelly--Kremnizer develop a theory of derived analytic geometry following the homotopical algebraic geometry of To\"en--Vezzosi \cite{toen_homotopical_2008}. The overarching idea of \S\ref{sec:homotopical}-\ref{sec:derivedrigidgeometry} is that derived algebraic geometry can be developed relative to any ``nice enough" symmetric monoidal $\infty$-category, where ``nice enough" is captured in the definition of an \emph{$(\infty,1)$-algebra context} \cite[\S 2.1]{DAnG}. By taking this symmetric monoidal $\infty$-category to be the derived category $D(\mathsf{CBorn}_K)$ of complete bornological spaces, one obtains a theory of derived analytic geometry, which contains a theory of derived rigid geometry\footnote{This is a difference between our work and that of \cite{Camargo_deRham,LecturesOnAnalyticGeometry}: we prefer the use of bornological spaces, instead of condensed mathematics.}. The properties of the $\infty$-category $D(\mathsf{CBorn}_K)$ are thoroughly investigated in \S\ref{sec:homotopical}.  

\emph{In \S\ref{subsec:derivedaffinoid}} we define the category $\mathsf{dAfndAlg}$ as a certain full subcategory of the monoids in $D_{\geqslant 0}(\mathsf{CBorn}_K)$. For any $A \in \mathsf{dAfndAlg}$, its truncation $\pi_0A$ is an affinoid algebra in the classical sense. We define $\mathsf{dAfnd}$ to be the opposite $\infty$-category to $\mathsf{dAfndAlg}$. We denote the object of $\mathsf{dAfnd}$ corresponding to $A \in \mathsf{dAfndAlg}$ by the formal expression $\operatorname{dSp}(A)$. We define the \emph{weak Grothendieck topology} on $\mathsf{dAfnd}$ whose covers are essentially given by finite jointly-surjective collections of \emph{derived rational subspaces}. We prove that this topology is subcanonical and that the prestack sending 
\begin{equation}
    \operatorname{dSp}(A) \mapsto \operatorname{QCoh}( \operatorname{dSp}(A)) := \operatorname{Mod}_A(D(\mathsf{CBorn}_K))
\end{equation}
is a sheaf in the weak topology.

\emph{In \S\ref{subsec:gluing}} we define the category $\mathsf{dRig}$ of derived rigid spaces as a certain full subcategory of $\operatorname{Shv}_{\mathrm{weak}}(\mathsf{dAfnd}, \infty \mathsf{Grpd})$. The full subcategory $\mathsf{dRig}$ is closed under all coproducts and fiber products, and $\mathsf{dRig}$ is equipped with the \emph{strong Grothendieck topology} whose covers are given by jointly-surjective families of \emph{analytic subspaces}. By right Kan extension along $\mathsf{dAfnd}^\mathsf{op} \to \mathsf{dRig}^\mathsf{op}$, the functor $\operatorname{QCoh}$ becomes a sheaf in the analytic topology.

\emph{In \S\ref{subsec:topological}} we explain how Hoffman-Lawson duality \cite[VII \S 4]{JohnstoneStone} can be used to associate a locally spectral topological space  $|X|$ to any $X \in \mathsf{dRig}$, whose locale of open subsets is canonically isomorphic to the locale of analytic subspaces of $X$. We also define a functor $X \mapsto X_0$ sending $X$ to its \emph{classical truncation}, which extends $\operatorname{dSp}(A) \mapsto \operatorname{dSp}(\pi_0A)$, and prove the topological invariance property $|X| \cong |X_0|$. Thus, derived rigid geometry obeys one of the principles of derived geometry, that ``all the geometry happens on $X_0$".  

\emph{In \S\ref{sec:derivedrigidsix}} we develop the six-functor formalism for quasi-coherent sheaves on derived rigid spaces. For any morphism $f: X \to Y$ in $\mathsf{dRig}$, we write $f^*$ for the symmetric-monoidal pullback functor from $\operatorname{QCoh}(Y)$ to $\operatorname{QCoh}(X)$ and $f_*$ for its right adjoint. In \S\ref{subsubsec:qcqs}, we prove that, for any qcqs morphism $f: X \to Y$ in $\mathsf{dRig}$, the functor $f_*$ satisfies base-change and the projection formula. This allows for the construction of a basic six-functor formalism in which the $!$-able morphisms are the qcqs ones. However, in analytic geometry, many interesting morphisms are not quasi-compact, and so we should enlarge the class of $!$-able morphisms. In \S\ref{subsec:sixf2} we develop an \emph{extension formalism} for abstract six-functor formalisms, which in \S\ref{sec:derivedrigidsix} allows for the class of $!$-able morphisms to be extended to a larger class $E$. The content of \S\ref{subsec:sixf2} is mostly a re-hashing of \cite[Theorem 4.20]{ScholzeSixFunctors} and we include it to convince the reader that the result of \emph{loc. cit.} is true in greater generality\footnote{Please see also Remark \ref{rmk:contentofsixf2}.}. The class $E$ has \emph{good stability properties}, c.f. Definition \ref{defn:Estabilityproperties}. 

\emph{In \S\ref{sec:univ!descent}} we establish that various interesting (non quasi-compact) morphisms in rigid geometry, really do belong to the class $E$ of $!$-able morphisms. Using that the class $E$ is $!$-local on the source, this boils down to showing that certain infinite covers are of universal $!$-descent (Corollary \ref{cor:countableshriekdescent}), so that for instance the morphism $\mathbf{A}^1_K \to *$ is $!$-able (Example \ref{ex:Affineline!able}). A notable feature of our approach is that we do not use compactifications, only the notion of ``compact supports" provided to us by \S\ref{subsec:sixf2}.

\emph{In \S\ref{sec:localcoh}} we develop a theory of local (co)homology in derived rigid geometry. Let $X \in \mathsf{dRig}$ and $S \subseteq |X|$ be a closed subset. Under the hypothesis that the complementary open $j: U \hookrightarrow X$ satisfies $j^! \simeq j^*$, we obtain various recollement sequences (Proposition \ref{prop:localcohformalproperties}). In Proposition \ref{prop:localcohincreasing} we use the results of \S\ref{sec:univ!descent} to provide a criterion for when $j^! \simeq j^*$. As a by-product we also obtain formulas for the local (co)homology functors in terms of sequential limits or colimits. 

\emph{In \S\ref{subsec:Zariskiclosed} and \S\ref{subsec:Zariskiopen}} we define Zariski-closed and Zariski-open immersions as the complements of Zariski-closed immersions. We show in Proposition \ref{prop:Zariskiopen} that these fit in to the formalism of \S\ref{sec:localcoh}: in particular, if $j: U \to X$ is a Zariski-open immersion, there is an equivalence $j^! \simeq j^*$.

\emph{In \S\ref{subsec:germs}} we introduce germs along Zariski-closed immersions\footnote{I was originally inspired to define differential operators via local cohomology, after reading the paper of Jiang \cite{jiang_derived_2023}. It seems quite plausible that Jiang's techniques generalize to the analytic setting. In his paper, Jiang makes extensive use of the
\emph{categorical K\"unneth formula}, that is, that $\operatorname{QCoh}(X \times_Y Z) \simeq \operatorname{QCoh}(X) \otimes_{\operatorname{QCoh}(Y)} \operatorname{QCoh}(Z)$, where the latter is Lurie's tensor product on presentable $\infty$-categories \cite{HigherToposTheory}. In the author's PhD thesis \cite[\S2.3.3]{SoorThesis} we obtain such a formula in analytic geometry, allowing for Jiang's work to be adapted.}. The definition is quite simple. Let $i: Z = \operatorname{dSp}(B) \to \operatorname{dSp}(A) = X$ be a Zariski-closed immersion of derived affinoids, induced by a morphism $A \to B$ surjective on $\pi_0$. Then the \emph{germ along $Z$} is 
\begin{equation}
    A_Z^\dagger := \underset{U \supseteq |Z|}{\operatorname{colim}}A_U 
\end{equation}
where the colimit is taken in $\mathsf{dAlg}:= \mathsf{CAlg}(D_{\geqslant 0}(\mathsf{CBorn}_K))^\mathsf{op}$ and runs over all affinoid subdomains $U \supseteq |Z|$. We denote the object of the opposite category $\mathsf{dAff} := \mathsf{dAlg}^\mathsf{op}$ corresponding to $A_Z^\dagger$ by the formal expression $(Z \subseteq X)^\dagger$.  In Proposition \ref{prop:germequivalence} we show that there is a natural equivalence of $\infty$-categories 
\begin{equation}
    \Gamma_Z \operatorname{QCoh}(X) \simeq \operatorname{QCoh}((Z \subseteq X)^\dagger) = \operatorname{Mod}_{A^\dagger_Z}D(\mathsf{CBorn}_K), 
\end{equation}
in algebra, we are familiar with the identification between ``sheaves with support" and sheaves on the formal completion, and this is nothing but an analytic counterpart to that. Indeed, a recurring theme of this article is the following: everywhere where one would see a ``formal neighbourhood" in algebraic geometry, in our analytic geometry we instead replace it by an \emph{analytic germ}, and see what we get. 

\emph{In the brief \S\ref{sec:sixfPStk}} we investigate a six-functor formalism which incorporates the above ``germs". In this we prefer a na\"ive approach. That is, we take the ``trivial" six-functor formalism on $\mathsf{dAff} = \mathsf{CAlg}(D_{\geqslant 0}(\mathsf{CBorn}_K))$ which sends $X = \operatorname{dSp}(A)$ to $\operatorname{QCoh}(\operatorname{dSp}(A)) := \operatorname{Mod}_AD(\mathsf{CBorn}_K)$ and in which every morphism $f$ satisfies $f_! = f_*$. The utility is that our ``germs" naturally belong to $\mathsf{dAff}$. Then we apply the formalism of \S\ref{subsec:sixf2} to obtain a six-functor formalism on $\mathsf{PStk} := \operatorname{Psh}(\mathsf{dAff}, \infty\mathsf{Grpd})$, with a class of $!$-able edges which has the good stability properties.

The theory developed in \S\ref{subsec:germs} and \S\ref{sec:sixfPStk} allows us in \S\ref{sec:stratification} to contemplate the following. For any morphism $f: X \to Y$ in $\mathsf{dAfnd}$ and any $n \geqslant 0$ we can consider the germ $(X \subseteq X^{n+1/Y})^\dagger$ along the diagonal. Letting $n$ vary, these can be arranged into a simplicial object in $\mathsf{PStk}$, which is in fact an \emph{internal groupoid object}, called the \emph{infinitesimal groupoid} and denoted $\operatorname{Inf}(X/Y)$. We define the \emph{stratifying stack} of $f$ as the geometric realization
\begin{equation}
    (X/Y)_{\mathrm{str}} := \underset{[n] \in \Delta^{\mathsf{op}}}{\operatorname{colim}} \operatorname{Inf}(X/Y)_n,
\end{equation}
where the colimit is taken in $\mathsf{PStk}$. When $Y = \operatorname{dSp}(K) = *$ is the terminal object, we just write $\operatorname{Inf}(X)$ and $X_{\mathrm{str}}$. As in \cite[Lecture VIII]{ScholzeSixFunctors}, the idea is that the six-functor formalism for stratifications already lives in the six-functor formalism on $\mathsf{PStk}$. More precisely, we identify a class of \emph{good} morphisms, c.f. Definition \ref{defn:goodmorphism}, which is stable under base-change and composition, and is such that  $(-)_{\mathrm{str}}$ induces a functor 
\begin{equation}
    (-)_{\mathrm{str}}: \operatorname{Corr}(\mathsf{dAfnd}, \mathrm{good}) \to \operatorname{Corr}(\mathsf{PStk},\widetilde{E}), 
\end{equation}
where we have denoted the class of $!$-able edges in the six-functor formalism on $\mathsf{PStk}$ by $\widetilde{E}$. By post-composition with the six-functor formalism $\operatorname{QCoh}$ on $(\mathsf{PStk}, \widetilde{E})$ gives a basic six-functor formalism 
\begin{equation}
    \operatorname{Strat} := \operatorname{QCoh} \circ (-)_{\mathrm{str}}
\end{equation}
for \emph{(analytic) stratifications}. By Kan extension, we can lift $\operatorname{Strat}$ to a six-functor formalism on all of $\mathsf{dRig}$, in which the class $E_{\mathrm{str}}$ of $!$-able morphisms contains all those which are representable in the class \emph{good}. Unfortunately, it does not seem that the extension formalism of \S\ref{subsec:sixf2} applies here because the class of \emph{good} morphisms does not have the right-cancellative property, so that the extension principles of \cite{mann_p-adic_2022} have to be applied in a more ad-hoc way. 

\emph{In \S\ref{subsec:DescentKashiwara}}, we prove that $\operatorname{Strat}$, when viewed as a prestack via the upper-star functors, is a sheaf in the analytic topology (Lemma \ref{lem:Crysdescent}). We prove a version of Kashiwara's equivalence for a class of Zariski-closed immersions $i: Z \to X$ in $\mathsf{dRig}$ which we call \emph{stratifying}: this means that $i$ locally admits a retraction.

\emph{In \S\ref{subsec:DmodformulasnStuff}}, we investigate the relation between $\operatorname{Strat}(X)$ and ``$\mathcal{D}$-modules". Let $X \in \mathsf{dAfnd}$. The key to understanding the relation is looking at the canonical morphism $p: X \to X_{\mathrm{str}}$ and the induced adjunctions $p^* \dashv p_*$ and $p_! \dashv p^!$ on quasi-coherent sheaves. We define the \emph{comonad of (analytic) jets} to be $\mathcal{J}^\infty_X:= p^*p_*$ and the \emph{monad of (infinite-order) differential operators} to be $\mathcal{D}^\infty_X := p^!p_!$. It is always the case that $p^* \dashv p_*$ is comonadic: that is, there is an equivalence between $\operatorname{Strat}(X)$ and comodules in $\operatorname{QCoh}(X)$ over the comonad $\mathcal{J}^\infty_X$. If $X \to X_{\mathrm{str}}$ is \emph{of $!$-descent} then the adjunction $p_! \dashv p^!$ is monadic so that there is an equivalence between $\operatorname{Strat}(X)$ and the category of modules over the monad $\mathcal{D}^\infty_X$. It turns out that there is a canonical equivalence $p_! \simeq p_*$ and that this can be used to pass between these descriptions in terms of $\mathcal{J}^\infty_X$-comodules and $\mathcal{D}^\infty_X$-modules. By base-change, the underlying endofunctors of $\mathcal{J}^\infty_X$ and $\mathcal{D}^\infty_X$ are given by the simple formulas
\begin{equation*}
\begin{aligned}
  \mathcal{J}_X^\infty \simeq \tilde{\pi}_{1,*}\tilde{\pi}_2^*  \simeq (A \widehat\otimes_K A)^\dagger_\Delta \widehat{\otimes}_A(-)&& \text{ and } && \mathcal{D}^\infty_X \simeq \tilde{\pi}_{2,*}\tilde{\pi}_1^! \simeq R \underline{\operatorname{Hom}}_A((A \widehat\otimes_K A)^\dagger_\Delta, -). 
\end{aligned}
\end{equation*}
Here $\tilde{\pi}_1, \tilde{\pi}_2: (X \subseteq X \times X)^\dagger \to X$ are the two projections. Moreover, under suitable hypotheses, we can chase the explicit equivalence of categories implicit in the Barr--Beck--Lurie theorem to give formulas for the six operations in Theorem \ref{thm:formulasforsix}. For some reason, it turns out to be convenient to use \emph{both} the descriptions to give these formulas, in terms of jets \emph{and} differential operators.

\emph{In Theorem \ref{thm:StratMonadicity}}, we give a partial answer to the question of when the morphism $p: X \to X_{\mathrm{str}}$ is of (universal) $!$-descent. From the previous discussion, this is clearly important for knowing when $\operatorname{Strat}(X)$ is related to $\mathcal{D}^\infty_X$-modules. We prove that $p$ is of universal $!$-descent whenever $X$ is a classical affinoid equipped with an \'etale morphism $X \to \mathbf{D}^r_K$.

\emph{In \S\ref{subsec:AWrelation}}, we compute the action of the monad $\mathcal{D}^\infty_X$ on the unit object $1_X \in \operatorname{QCoh}(X)$ when $X$ is a classical affinoid equipped with an \'etale morphism $X \to \mathbf{D}^r_K$. We show that it agrees with the infinite-order differential operators introduced by Ardakov-Wadsley \cite{Dcap1}. In our follow-up work \cite{soor_relation_2025} we show that the category of $\mathcal{C}$-complexes introduced by Bode in \cite[\S 6]{bode_six_2021}, embeds fully-faithfully in $\operatorname{Strat}(X)$, for any smooth classical rigid variety $X$. 

\paragraph{Acknowledgements.} Firstly, I'd like to thank my supervisor Konstantin Ardakov for his continued interest and support for this project. I would like to thank  Andreas Bode, Lukas Brantner, Jack Kelly, Ken Lee, James Taylor, and Finn Wiersig for their interest and many helpful discussions, meetings and email exchanges and conversations related to this work. I would also like to thank Guillermo Cortiñas, Kobi Kremnizer, and Devarshi Mukherjee for their interest in this work.

\section{Homotopical algebra in quasi-abelian categories}\label{sec:homotopical}
We make extensive use of the theory of homotopical algebra in quasi-abelian categories as developed in \cite{kelly_homotopy_2021} and \cite{schneiders_quasi-abelian_1999}. We assume familiarity with the basics of higher algebra and model categories, as these topics are too vast to give a proper summary; we will give an indication of any non-standard or particular notions.

Let $\mathscr{A}$ be a quasi-abelian category. Recall \cite[\S 1]{schneiders_quasi-abelian_1999} that this means that $\mathscr{A}$ is an additive category which has all kernels and cokernels, and strict\footnote{Recall that a morphism $f$ is called strict if the natural morphism $\operatorname{coker} \operatorname{ker}f \xrightarrow[]{ }\operatorname{ker}\operatorname{coker}f$ is an isomorphism.} epimorphisms (resp. monomorphisms) are stable under pullbacks (resp. pushouts). We recall the following properties.
\begin{defn}
\begin{enumerate}[(i)]
    \item A functor $F: \mathscr{A} \to \mathscr{B}$ between quasi-abelian categories is called \emph{left exact} (resp. \emph{strongly left exact}) if it preserves the kernels of strict morphisms (resp. all morphisms).
    \item A functor $F: \mathscr{A} \to \mathscr{B}$ between quasi-abelian categories is called \emph{right exact} (resp. \emph{strongly right exact}) if it preserves the cokernels of strict morphisms (resp. all morphisms).
    \item A functor $F: \mathscr{A} \to \mathscr{B}$ between quasi-abelian categories is called \emph{exact} (resp. \emph{strongly exact}) if it is left and right exact (resp. strongly left and right exact). 
\end{enumerate}
  
\end{defn}
\begin{defn}\cite[Definition 2.47]{kelly_homotopy_2021}
\begin{enumerate}[(i)]
    \item An object $P \in \mathscr{A}$ is called \emph{projective} if the functor $\operatorname{Hom}(P,-): \mathscr{A} \to \mathsf{Ab}$ (valued in abelian groups), takes strict epimorphisms to surjections.
    \item We say that $\mathscr{A}$ has \emph{enough projectives} if for each object $M \in \mathscr{A}$ there exists a projective object $P$ together with a strict epimorphism $P \twoheadrightarrow M$. 
\end{enumerate}
\end{defn}
\begin{defn}\cite[Definition 2.92]{kelly_homotopy_2021}
Assume that $\mathscr{A}$ admits small coproducts. A small subcategory $\mathscr{P}$ of objects in $\mathscr{A}$ is called \emph{generating} if for each object $M \in \mathscr{A}$ there exists a small collection $\{P_i\}_{i \in \mathcal{I}}$ of objects of $\mathscr{P}$ together with a strict epimorphism $\bigoplus_{i \in \mathcal{I}} P_i \twoheadrightarrow M$.
\end{defn}
\begin{defn}\cite[Definition 2.1.10]{schneiders_quasi-abelian_1999}. Let $\mathscr{A}$ be a quasi-abelian category.
\begin{enumerate}[(i)]
    \item Assume that $\mathscr{A}$ admits (small) coproducts. An object $C \in \mathscr{A}$ is called \emph{small} if $\operatorname{Hom}_{\mathscr{A}}(C,-): \mathscr{A} \to \mathsf{Ab}$ commutes with (small) coproducts.
    \item The category $\mathscr{A}$ is called \emph{quasi-elementary} if it is cocomplete and has a small generating subcategory $\mathscr{P} \subseteq \mathscr{A}$ of small projective objects. 
\end{enumerate}
\end{defn}
\begin{defn}\cite[Definition 2.97]{kelly_homotopy_2021} Let $\mathcal{S}$ be a collection of morphisms in a cocomplete quasi-abelian category $\mathscr{A}$ stable under composition.
\begin{enumerate}[(i)]
        \item Let $\mathcal{I}$ be a filtered category. An object $C \in \mathscr{A}$ is called \emph{$(\mathcal{I},\mathcal{S})$-tiny} if the functor $\operatorname{Hom}(C,-): \mathscr{A} \to \mathsf{Ab}$ commutes with colimits of diagrams in $\operatorname{Fun}_\mathcal{S}(\mathcal{I},\mathscr{A})$. Here $\operatorname{Fun}_\mathcal{S}(\mathcal{I},\mathscr{A}) \subseteq \operatorname{Fun}(\mathcal{I},\mathscr{A})$ denotes the sub-class of those functors which take morphisms in $\mathcal{I}$ into $\mathcal{S}$.
        \item The category $\mathscr{A}$ is called \emph{$(\mathcal{I},\mathcal{S})$-elementary} if $\mathscr{A}$ is generated by a subcategory $\mathscr{P} \subseteq \mathscr{A}$ consisting of $(\mathcal{I},\mathcal{S})$-tiny projective objects. 
        \item An object $C \in \mathscr{A}$ is called \emph{$\mathcal{S}$-tiny} if the functor $\operatorname{Hom}(C,-): \mathscr{A} \to \mathsf{Ab}$ commutes with colimits of diagrams in $\operatorname{Fun}_\mathcal{S}(\mathcal{I},\mathscr{A})$, for any filtered category $\mathcal{I}$. 
    \item The category $\mathscr{A}$ is called \emph{$\mathcal{S}$-elementary} if $\mathscr{A}$ is generated by a subcategory $\mathscr{P} \subseteq \mathscr{A}$ consisting of $\mathcal{S}$-tiny projective objects. 
\end{enumerate}
\end{defn}
In what follows we will often take $(\mathcal{I}, \mathcal{S}):= (\mathbf{N}, \mathrm{SplitMon})$, where $\mathrm{SplitMon}$ is the class of split monomorphisms in $\mathscr{A}$, or $\mathcal{S} := \mathrm{AdMon}$ to be the class of strict monomorphisms in $\mathscr{A}$, or $\mathcal{S}:= \mathrm{all}$. In each of these cases we say that $\mathscr{A}$ is \emph{$(\mathbf{N}, \mathrm{SplitMon})$-elementary}, \emph{$\mathrm{AdMon}$-elementary}, and \emph{elementary}, respectively. Of course, we have the following chain of implications:
\begin{equation}
\begin{tikzcd}[row sep=small]
	{\text{elementary}} \\
	{\text{AdMon-elementary}} \\
	{\text{quasi-elementary}} \\
	{(\mathbf{N},\mathrm{SplitMon})\text{-elementary}} \\
	{\text{enough projectives}.}
	\arrow[Rightarrow, from=1-1, to=2-1]
	\arrow[Rightarrow, from=2-1, to=3-1]
	\arrow[Rightarrow, from=3-1, to=4-1]
	\arrow[Rightarrow, from=4-1, to=5-1]
\end{tikzcd}
\end{equation}
\begin{notations}
Let $\mathscr{A}$ be an additive category and let $\operatorname{Ch}(\mathscr{A})$ denote the category of cochain complexes. In this article we always use \emph{superscripts} to denote \emph{cohomological} indexing convention and \emph{subscripts} for \emph{homological} indexing. These conventions are related by $M^{i} = M_{-i}$ for $i \in \mathbf{Z}$. 
\end{notations}
\begin{thm}\cite[Theorem 4.59, Theorem 4.65]{kelly_homotopy_2021}\label{thm:kelly}
\begin{enumerate}[(i)]
    \item Let $\mathscr{A}$ be a quasi-abelian category with enough projectives. Then the \emph{projective model structure} on $\operatorname{Ch}^{\leqslant 0}(\mathscr{A})$ exists. The weak equivalences, fibrations and cofibrations may be described as follows:
    \begin{enumerate}
    \item[(W)] A morphism is a weak equivalence if it is a strict quasi-isomorphism, i.e., its cone is strictly exact. 
    \item[(F)] A morphism is a fibration if the its components are strict epimorphisms in positive degrees. 
    \item[(C)] A morphism is a cofibration if it is a degreewise strict monomorphism with degreewise projective cokernel.
    \end{enumerate}
    Further, this is a simplicial model structure. 
    \item Assume that $\mathscr{A}$ is a $(\mathbf{N},\mathrm{SplitMon})$-elementary quasi-abelian category. Then the \emph{projective model structure} on $\operatorname{Ch}(\mathscr{A})$ exists. The weak equivalences and fibrations may be described as follows:
    \begin{enumerate}
    \item[(W)] A morphism is a weak equivalence if it is a strict quasi-isomorphism, i.e., its cone is strictly exact. 
    \item[(F)] A morphism is a fibration if it is a degreewise strict epimorphism.
    \end{enumerate}
    Further, this is a stable and simplicial model structure. 
\end{enumerate}
\end{thm}
This permits us to make the following definition.
\begin{defn}
Let $\mathscr{A}$ be a $(\mathbf{N},\mathrm{SplitMon})$-elementary quasi-abelian category. The \emph{derived $\infty$-category} of $\mathscr{A}$ is defined to be the underlying $\infty$-category of the simplicial model category $\operatorname{Ch}(\mathscr{A})$. That is, it is the $\infty$-categorical localization
\begin{equation}
    D(\mathscr{A}) := N(\operatorname{Ch}(\mathscr{A}))[W^{-1}].
\end{equation}
This is a stable $\infty$-category.
\end{defn}
We recall \cite[\S 1.2.2]{schneiders_quasi-abelian_1999} that $D(\mathscr{A})$ is equipped with two canonical $t$-structures. Of these, it is conventional to prefer the \emph{left $t$-structure} which may be described as follows.\footnote{We recall that a $t$-structure on a stable $\infty$-category, is \emph{by definition} given by a $t$-structure on its homotopy category.} 
\begin{prop}\cite[\S 1.2.2]{schneiders_quasi-abelian_1999}
Let $D^{\leqslant 0}(\mathscr{A})$ (resp. $D^{\geqslant 0}(\mathscr{A})$) denote the full sub $\infty$-category of complexes which are strictly exact in positive (resp. negative) degrees. Then the pair
\begin{equation}
    (D^{\leqslant 0}(\mathscr{A}), D^{\geqslant 0}(\mathscr{A})) 
\end{equation}
determines a $t$-structure on $D(\mathscr{A})$. 
\end{prop}
The heart of this $t$-structure is called the \emph{left heart} of $\mathscr{A}$ \cite[\S 1.2.3]{schneiders_quasi-abelian_1999} and denoted by $LH(\mathscr{A})$. Consequently, we get cohomology functors 
\begin{equation}
    H^i: D(\mathscr{A}) \to LH(\mathscr{A})
\end{equation}
for each $i \in \mathbf{Z}$. The category $LH(\mathscr{A})$ admits the following very explicit description. Let $K(\mathscr{A})$ be the category with $\operatorname{Ob}(K(\mathscr{A})) = \operatorname{Ob}(\operatorname{Ch}(\mathscr{A}))$ and whose morphisms are chain-homotopy classes of morphisms in $\operatorname{Ch}(\mathscr{A})$. 
\begin{prop}\cite[Corollary 1.2.21]{schneiders_quasi-abelian_1999}
The category $LH(\mathscr{A})$ is equivalent to the full subcategory of $K(\mathscr{A})$ on two-term complexes
\begin{equation}
 0\to    M^{-1} \xrightarrow[]{d} M^0 \to 0
\end{equation}
with $d$ a monomorphism, \emph{localized} at the class of morphisms $f: [M^{-1} \to M^0] \to [N^{-1} \to N^0]$ such that
\begin{equation}
\begin{tikzcd}
	{M^{-1}} & {M^0} \\
	{N^{-1}} & {N^0}
	\arrow[from=1-1, to=1-2]
	\arrow["{f^{(-1)}}"', from=1-1, to=2-1]
	\arrow["{f^0}", from=1-2, to=2-2]
	\arrow[from=2-1, to=2-2]
\end{tikzcd}
\end{equation}
is a Cartesian and coCartesian square in $\mathscr{A}$. 
\end{prop}
With respect to this description, we obtain the following:
\begin{prop}\cite[\S 1.2]{schneiders_quasi-abelian_1999}, see also \cite[\S3.1]{bode_six_2021}. 
\begin{enumerate}[(i)]
    \item The canonical functor $I: \mathscr{A} \to LH(\mathscr{A})$ is induced by the functor given by 
    \begin{equation}
        M \mapsto [0 \to M].
    \end{equation}
    This is fully-faithful and admits a left adjoint $C: LH(\mathscr{A}) \to \mathscr{A}$ which is induced by the functor
    \begin{equation}
     [M^{-1} \xrightarrow[]{d} M^0] \mapsto \operatorname{coker}d.
    \end{equation}
    on two-term complexes, so that $\mathscr{A}$ is a reflective subcategory of $LH(\mathscr{A})$. 
    \item The cohomology functor $H^i : D(\mathscr{A}) \to LH(\mathscr{A})$ is given on objects by 
    \begin{equation}
       H^i: M^\bullet \mapsto [\operatorname{coker} \operatorname{ker} d^{i-1} \to \operatorname{ker} d^i].
    \end{equation}
    In particular, a complex $M^\bullet \in D(\mathscr{A})$ is strict (meaning that its differentials are strict morphisms) if and only if the cohomology objects $H^iM^\bullet \in LH(\mathscr{A})$ factor through (the essential image of) $\mathscr{A}$. A complex $M^\bullet$ is strictly exact if and only if $H^iM^\bullet = 0$ for all $i \in \mathbf{Z}$. 
\end{enumerate}
\end{prop}
\subsection{Banach rings and complete bornological modules}\label{sec:Banach}
\begin{rmk}[Important remark]
In this article we only consider \emph{non-Archimedean} Banach algebras and \emph{non-Archimedean} Banach modules. However, it seems quite likely, or perhaps even obvious, that much of the content of this article carries over to the Archimedean setting. 
\end{rmk}
Our conventions on Banach rings and modules follows Berkovich \cite[Chapter 1]{BerkovichSpectral}. 
\begin{defn}
\begin{enumerate}[(i)]
    \item Let $V$ be an abelian group. A (non-Archimedean) \emph{seminorm} on $V$ is a function $\|\cdot\|: V \to \mathbf{R}^{\geqslant 0}$ such that $\|0\|= 0$ and $\|v-w\| \leqslant \operatorname{max}\{ \|v\|,\|w\|\}$ for all $v,w \in V$. It is called a \emph{norm} if $\|v\| =0$ implies $v =0 $. $V$ is called \emph{complete} if it is complete as a metric space. We write $\widehat{V}$ for the completion of a seminormed abelian group $V$. Seminorms $\|\cdot\|$ and $\|\cdot\|^\prime$ are called \emph{equivalent} if there exists $C, C^\prime \in \mathbf{R}^{>0}$ such that $C \|\cdot \| \leqslant  \|\cdot \|^\prime \leqslant C^\prime \|\cdot \|$. 
    \item Let $R$ be a ring. A \emph{seminorm} on $R$ is a seminorm on the abelian group $(R,+)$ such that $\|1\| = 1$ and $\|fg \| \leqslant \|f\|\cdot \|g\|$ for all $f, g \in R$. It is called \emph{multiplicative} if $\|fg\| = \|f\|\cdot \|g\|$ is satisfied for all $f,g \in R$. A \emph{Banach ring} is a normed ring which is complete as a metric space.  
    \item Let $R$ be a seminormed ring. A \emph{seminormed $R$-module} is a $R$-module $V$ equipped with a seminorm $\|\cdot\|_V$ such that there exists $C \in \mathbf{R}^{>0}$ such that $\|fv\|_V \leqslant C\|f\|_R \|v\|_V$ for all $f \in R$ and all $v \in V$. If $R$ is a Banach ring, then such $V$ is called a \emph{Banach $R$-module} if it is complete as a metric space.
    \item A \emph{non-Archimedean field} is a field which is complete with respect to a multiplicative seminorm. 
    \item Let $R$ be a Banach ring and let $V, W$ be Banach $R$-modules. The \emph{completed tensor product} $V \widehat{\otimes}_R W$ is defined to be the completion of $V \otimes_RW $ with respect to the norm
    \begin{equation}
        \|x\| := \operatorname{inf}\left\{ \max_i \|v_i\|\|w_i\|: x = \sum v_i \otimes w_i\right\}.
    \end{equation}
    The \emph{internal Hom}, written  $\underline{\operatorname{Hom}}_R(V,W)$ is the $R$-module of bounded $R$-linear maps $\operatorname{Hom}_R(V,W)$ equipped with the \emph{operator norm}. 
\end{enumerate}
\end{defn}

\begin{defn}
Let $K$ be a non-trivially valued non-Archimedean field with unit ball $o \subseteq K$. Let $V$ be a $K$-vector space. A \emph{bornology} on $V$ is a collection of $\mathscr{B}$ of \emph{bounded subsets} of $V$ satifying the following properties:
\begin{enumerate}[(i)]
    \item If $B \in \mathscr{B}$ and $B^\prime \subseteq B$ then $B^\prime \in \mathscr{B}$;
    \item If $v \in V$ then $\{v\} \in \mathscr{B}$;
    \item $\mathscr{B}$ is closed under finite unions;
    \item If $B \in \mathscr{B}$ and $r \in K$ then $rB \in \mathscr{B}$;
    \item If $B \in \mathscr{B}$ then the $o$-submodule $o \cdot B \in \mathscr{B}$.
\end{enumerate}
The pair $V = (V,\mathscr{B})$ is called a \emph{(convex) bornological $K$-vector space}. A morphism $T: V \to W$ of $K$-vector spaces is called \emph{bounded} if $T(B) \subseteq W$ is bounded for every bounded subset $B \subseteq V$. In this way we obtain the category $\mathsf{Born}_K$ of bornological $K$-vector spaces.
\end{defn}
\begin{prop}\cite{HogbeNlendThesis,ProsmansHomological, BambozziDaggerBanach}.
Let $K$ be a non-trivially valued non-Archimedean field with unit ball $o \subseteq K$. 
\begin{enumerate}[(i)]
    \item The category $\mathsf{Born}_K$ is closed symmetric monoidal. The \emph{bornological tensor product} is defined to be $V \otimes _KW$ endowed with the bornology generated by the collection of $B \otimes_o B^\prime$ for $B, B^\prime$ bounded $o$-submodules of $V, W$. The \emph{internal Hom} $\underline{\operatorname{Hom}}_K(V,W)$ is the $K$-vector space $\operatorname{Hom}_K(V,W)$ of bounded linear maps equipped with the bornology generated by equibounded subsets.
    \item $\mathsf{Born}_K$ is a complete and cocomplete $\mathrm{AdMon}$-elementary quasi-abelian category.
    \item A generating family of $\mathrm{AdMon}$-tiny projective objects is given by $\{ \coprod^{\leqslant 1}_S K\}_S$ for $S$ ranging over (small) sets. Here $\coprod^{\leqslant1}_S K$ is the normed $K$-vector space with underlying $K$-vector space $\coprod_S K$ and norm $\|(r_s)_s\| := \operatorname{sup}_{s} \| r_s\|$.
    \item A morphism $\varphi: V \to W$ is strict if and only if the subspace bornology on $\operatorname{im}\varphi$ coincides with the quotient bornology on $V/\operatorname{ker}\varphi$. 
    \end{enumerate}
\end{prop}
Every seminormed $K$-vector space acquires a bornology in an obvious way. Therefore we may make the following definition. The point is that the norm should not be the part of the data of a $K$-Banach space, only the bornology. 
\begin{defn}
Let $K$ be a non-trivially valued non-Archimedean field. The category of \emph{seminormed $K$-vector spaces}, (resp. \emph{normed $K$-vector spaces}, resp. \emph{$K$-Banach spaces}), is defined to be the full subcategory of $\mathsf{Born}_K$ on objects $V$ whose bornology is induced by a seminorm (resp. a norm, resp. a norm making $V$ into a Banach space). We denote these categories by $\mathsf{SNrm}_K, \mathsf{Nrm}_K$, and $\mathsf{Ban}_K$, respectively.
\end{defn}
\begin{defn}
Let $K$ be a non-trivially valued non-Archimedean field and let $V \in \mathsf{Born}_K$. Given a bounded $o$-submodule $B \subseteq V$ we define $V_B := \operatorname{span}_K B \subseteq V$ equipped with the bornology defined by the gauge seminorm:
    \begin{equation}
        \|x\|_B := \operatorname{inf}\{|\lambda|: x \in \lambda B\}.  
    \end{equation}
\end{defn}
\begin{prop}\cite{HogbeNlendThesis, ProsmansHomological}.
Let $K$ be a non-trivially valued non-Archimedean field. There is an adjunction 
\begin{equation}
    \operatorname{diss}: \mathsf{Born}_K \leftrightarrows \mathsf{Ind}(\mathsf{SNrm}_K) : \operatorname{colim}
\end{equation}
in which the right adjoint $\operatorname{diss}: V \mapsto \indlim V_B$ is fully faithful. The essential image is given by the \emph{essentially monomorphic $\mathrm{Ind}$-objects}, i.e., those $\mathrm{Ind}$-objects which are equivalent to $\mathrm{Ind}$-systems of monomorphisms. Consequently there is an equivalence of categories 
\begin{equation}
\operatorname{colim}:\mathsf{Ind}^m(\mathsf{SNrm}_K) \xrightarrow[]{\sim} \mathsf{Born}_K.
\end{equation}
\end{prop}
\begin{defn}\cite{HogbeNlendThesis, ProsmansHomological}.
Let $K$ be a non-trivially valued non-Archimedean field and let $V \in \mathsf{Born}_K$. 
\begin{enumerate}[(i)]
    \item $V$ is called \emph{separated} if, for every $B \in \mathscr{B}$ there exists a bounded $o$-submodule $B^\prime  \supseteq B$ such that the gauge seminorm on $V_{B^\prime}$ is a norm. We let $\mathsf{SBorn}_K \subseteq \mathsf{Born}_K$ denote the full subcategory on separated bornological $K$-vector spaces. 
    \item $V$ is called \emph{complete} if, for every $B \in \mathscr{B}$ there exists a bounded $o$-submodule $B^\prime  \supseteq B$ such that $V_{B^\prime}$ is a $K$-Banach space. We let $\mathsf{CBorn}_K \subseteq \mathsf{Born}_K$ denote the full subcategory on complete bornological $K$-vector spaces.
\end{enumerate}
\end{defn}
\begin{prop}\cite{HogbeNlendThesis, ProsmansHomological}.
Let $K$ be a non-trivially valued non-Archimedean field.
\begin{enumerate}[(i)]
\item The inclusion $\mathsf{SBorn}_K \subseteq \mathsf{Born}_K$ admits a left adjoint $\operatorname{sep}: \mathsf{Born}_K \to \mathsf{SBorn}_K$.
\item The category $\mathsf{SBorn}_K$ is closed symmetric monoidal. The tensor product is given by $\operatorname{sep}(V \otimes_K W)$ and the internal Hom is the same as in $\mathsf{Born}_K$. 
\item $\mathsf{SBorn}_K$ is a complete and cocomplete $\operatorname{AdMon}$-elementary quasi-abelian category. 
\item A generating family of $\mathrm{AdMon}$-tiny projective objects is given by $\{ \coprod^{\leqslant1}_S K\}_S$ for $S$ ranging over (small) sets. 
\item A morphism $\varphi: V \to W$ is strict if and only if $\operatorname{im}\varphi$ is bornologically closed and the bornology on $\operatorname{im}\varphi$ coincides with the quotient bornology on $V/\operatorname{ker} \varphi$. 
\item There is an adjunction
\begin{equation}
    \operatorname{diss}: \mathsf{SBorn}_K \leftrightarrows \mathsf{Ind}(\mathsf{Nrm}_K) : \operatorname{colim}
\end{equation}
in which the right adjoint $\operatorname{diss}: V \mapsto \indlim \operatorname{sep}(V_B)$ is fully-faithful and whose essential image is given by the the essentially monomorphic $\operatorname{Ind}$-objects, so that there is an equivalence of categories
\begin{equation}
    \operatorname{colim} : \mathsf{Ind}^m(\mathsf{Nrm}_K) \xrightarrow[]{\sim} \mathsf{SBorn}_K . 
\end{equation}
\end{enumerate}
\end{prop}
\begin{defn}
Let $K$ be a non-trivially valued non-Archimdean field and let $V$ be a $K$-Banach space. Let $S$ be a (small) set. We define the \emph{space of $V$-valued zero sequences} to be 
    \begin{equation}
        c_0(S,V):= \{ \phi: S \to V : \forall \varepsilon>0,  \exists \text{ at most finitely many }s\in S: \|\phi(s)\| > \varepsilon\},
    \end{equation}
with the norm $\|\phi\| := \sup_{s \in S}\|\phi(s)\|$. When $V=K$ we will just write $c_0(S) := c_0(S,K)$. 
\end{defn}
\begin{lem} Let $K$ be a non-trivially valued non-Archimedean field and let $V \in \mathsf{CBorn}_K$. 
\begin{enumerate}[(i)]
    \item There is a natural isomorphism
    \begin{equation}
        \operatorname{Hom}_K(c_0(S),V) \cong \left\{ \text{functions }f:S \to V: f(S) \subseteq V \text{ is bounded} \right\}. 
    \end{equation}
    \item For every $S, S^\prime$, there are natural isomorphisms
    \begin{equation}
       c_0(S) \widehat{\otimes}_K c_0(S^\prime) \cong   c_0(S \times S^\prime) \cong c_0(S, c_0(S^\prime)), 
    \end{equation}
    of $K$-Banach spaces. 
\end{enumerate}
\end{lem}
\begin{proof}
We omit the proof of (i). We only mention that (ii) can be proved using (i) together with currying and Yoneda's lemma. 
\end{proof}
\begin{prop}\cite{HogbeNlendThesis, ProsmansHomological} \label{prop:CBornK}
Let $K$ be a non-trivially valued non-Archimedean field. 
\begin{enumerate}[(i)]
\item The inclusion $\mathsf{CBorn}_K \subseteq \mathsf{Born}_K$ admits a left adjoint $\widehat{(\cdot)}: \mathsf{Born}_K \to \mathsf{CBorn}_K$.
\item The category $\mathsf{CBorn}_K$ is closed symmetric monoidal. The tensor product is given by the \emph{completed tensor product} 
\begin{equation}
    V \widehat{\otimes}_K W := \widehat{V \otimes_K W}, 
\end{equation}
and the internal Hom is the same as in $\mathsf{Born}_K$. 
\item $\mathsf{CBorn}_K$ is a complete and cocomplete $\operatorname{AdMon}$-elementary quasi-abelian category. A generating family of $\mathrm{AdMon}$-tiny projective objects is given by $\{c_0(S)\}_S$ for $S$ ranging over (small) sets. A morphism $\varphi: V \to W$ is strict if and only if $\operatorname{im}\varphi$ is bornologically closed and the bornology on $\operatorname{im}\varphi$ coincides with the quotient bornology on $V/\operatorname{ker} \varphi$. 
\item There is an adjunction
\begin{equation}
    \operatorname{diss}: \mathsf{CBorn}_K \leftrightarrows \mathsf{Ind}(\mathsf{Ban}_K) : \operatorname{colim}
\end{equation}
in which the right adjoint $\operatorname{diss}: V \mapsto \indlim (\widehat{V}_B)$ is fully-faithful and whose essential image is given by the the essentially monomorphic $\operatorname{Ind}$-objects, so that there is an equivalence of categories
\begin{equation}
    \operatorname{colim} : \mathsf{Ind}^m(\mathsf{Ban}_K) \xrightarrow[]{\sim} \mathsf{CBorn}_K . 
\end{equation}
\end{enumerate}
\end{prop}
\begin{cor}\label{cor:essentiallyMono}
Let $K$ be a non-trivially valued non-Archimedean field. In any of the categories $\mathsf{Born}_K$, $\mathsf{SBorn}_K$ and $\mathsf{CBorn}_K$, colimits of (essentially) monomorphic filtered systems are strongly exact.
\end{cor}

\begin{rmk}\label{rmk:CBornRgeneral}
It is natural to ask how the definition of $\mathsf{CBorn}_K$ generalizes when we replace $K$ by a general Banach ring $R$. It seems like the good definition is $\mathsf{CBorn}_R := \mathsf{Ind}^m\mathsf{Ban}_R$. By Proposition \ref{prop:CBornK}(iv) this is consistent with our previous definition in the case when $R = K$ is a non-trivially valued non-Archimedean field. 
\end{rmk}

\subsection{Universal property of the derived category}\label{sec:universalpropertyofderived}
In this section we assume that 
\begin{itemize}
    \item[$\star$] $\mathscr{A}$ is a quasi-elementary quasi-abelian category.
\end{itemize}
We fix a small generating set of small projective objects $\mathscr{P} \subseteq \mathscr{A}$, and we assume that $\mathscr{P}$ is closed under finite products in $\mathscr{A}$ (there is no harm in this). We always view $\mathscr{P}$ as a full subcategory of $\mathscr{A}$. 
There is a functor
\begin{equation}
\mathscr{A} \to   \operatorname{Fun}^\Pi(\mathscr{P}^\mathsf{op}, \mathsf{Set})
\end{equation}
which sends $A \mapsto h_A := \operatorname{Hom}(-,A)$. (Here and elsewhere, $\operatorname{Fun}^\Pi$ denotes the \emph{finite} product-preserving functors). Passing to simplicial objects, we obtain
\begin{equation}\label{eq:passtosimplicial}
s\mathscr{A} \to \operatorname{Fun}^\Pi(\mathscr{P}^\mathsf{op}, s\mathsf{Set})
\end{equation}
sending $A_\bullet \mapsto h_{A_\bullet} := \operatorname{Hom}(-,A_\bullet)$.
\begin{prop}\cite[\S II.4]{quillen_homotopical_1967}
With notations as above. We may define a model structure (called the \emph{standard model structure}) on $s\mathscr{A}$ as follows: A morphism $A_\bullet \to B_\bullet$ is a weak equivalence (resp. fibration) if $\operatorname{Hom}(P,A_\bullet) \to \operatorname{Hom}(P, B_\bullet)$ is a weak homotopy equivalence (resp. Kan fibration) in $s\mathsf{Set}$, for all $P \in \mathscr{P}$.
\end{prop}
\begin{prop}[Quillen, Bergner, Lurie]\label{prop:simplicialmodelstructure}
\begin{enumerate}[(i)]
    \item We may define a model structure on $\operatorname{Fun}^\Pi(\mathscr{P}^\mathsf{op},s \mathsf{Set})$ as follows: A morphism $\alpha : F \to F^\prime $ is a weak equivalence (resp. fibration) if $\alpha_P: F(P) \to F^\prime(P)$ is a weak homotopy equivalence (resp. Kan fibration) in $s\mathsf{Set}$, for all $P \in \mathscr{P}$.
    \item This model structure presents the $\infty$-category 
    \begin{equation}
    \operatorname{Fun}^\Pi(N(\mathscr{P})^\mathsf{op},\infty\mathsf{Grpd}) = \operatorname{sInd}(N(\mathscr{P})). 
    \end{equation}
    \item By transport of structure via \eqref{eq:passtosimplicial}, the model structure of (i) gives the standard model structure on $s\mathscr{A}$. 
\end{enumerate} 
\end{prop}
\begin{proof}
(i) This is \cite[Proposition 5.5.9.1]{HigherToposTheory}. (ii): This is \cite[Corollary 5.5.9.3]{HigherToposTheory}. (iii): Obvious.
\end{proof}
\begin{prop}\label{prop:DoldKan}
Under the equivalence of categories 
\begin{equation}
N :s\mathscr{A} \simeq  \operatorname{Ch}^{\leqslant 0}(\mathscr{A})  : \Gamma
\end{equation}
furnished by the Dold-Kan correspondence, the model structure on $s\mathscr{A}$ induces the model structure on $\operatorname{Ch}^{\leqslant 0}(\mathscr{A})$ and vice-versa. 
\end{prop}
\begin{proof}
The functors $N$ and $\Gamma$ furnish an equivalence of categories by, for instance, \cite[Corollary 4.73]{kelly_homotopy_2021}. For each $P \in \mathscr{P}$ we have a commutative diagram 
\begin{equation}
\begin{tikzcd}
	{s\mathscr{A}} & {\operatorname{Ch}^{\leqslant 0}(\mathscr{A})} \\
	{s\mathsf{Ab}} & {\operatorname{Ch}^{\leqslant 0}(\mathsf{Ab})}
	\arrow["N", from=1-1, to=1-2]
	\arrow["{\operatorname{Hom}(P,-)}"', from=1-1, to=2-1]
	\arrow["{\operatorname{Hom}(P,-)}", from=1-2, to=2-2]
	\arrow["N", from=2-1, to=2-2]
\end{tikzcd}
\end{equation}
By definition, $f: X_\bullet \to Y_\bullet$ is a fibration if and only if $\operatorname{Hom}(P,X_\bullet) \to \operatorname{Hom}(P,Y_\bullet)$ is a fibration for all $P \in \mathscr{P}$. By (the generalization to abelian groups of) \cite[Proposition 2.3.1]{quillen_homotopical_1967}, together with commutativity of the above square, this holds if and only if $\operatorname{Hom}(P, N(X_\bullet)) \to \operatorname{Hom}(P, N(Y_\bullet))$ is an epimorphism (of chain complexes) in positive degrees. By \cite[Proposition 2.4.2]{quillen_homotopical_1967} this implies that $N(X_\bullet) \to N(Y_\bullet)$ is a degreewise strict epimorphism (we recall that in additive categories, the notion of strict and effective epimorphism coincide). Hence $f$ is a fibration if and only if $N(f)$ is. 

Similarly, a morphism $f: X_\bullet \to Y_\bullet$ is a weak equivalence if and only if $
\operatorname{Hom}(P, N(X_\bullet)) \to \operatorname{Hom}(P, N(Y_\bullet))$ is a quasi-isomorphism of chain complexes, for each $P \in \mathscr{P}$. Looking at the cone of $N(X_\bullet) \to N(Y_\bullet)$ and applying \cite[Corollary 2.95]{kelly_homotopy_2021} this holds if and only if $N(X_\bullet) \to N(Y_\bullet)$ is a strict quasi-isomorphism. Hence $f$ is a weak equivalence if and only if $N(f)$ is.
\end{proof}
\begin{cor}\label{cor:simplicialmodelstructure}
The model structure on $s\mathscr{A}$ may be equivalently described as follows: A morphism $A_\bullet \to B_\bullet$ is a weak equivalence (resp. fibration, resp. cofibration), if and only if it is a strict weak homotopy equivalence (resp. strict epimorphism in positive degrees, resp. degreewise strict monomorphism with degreewise projective cokernel). 
\end{cor}
\begin{proof}
It should be possible to argue directly, looking at the proof of \cite[\S II.4]{quillen_homotopical_1967}. In any case, we can use Proposition \ref{prop:DoldKan}, whence the claims about weak equivalences and fibrations are clear. The statement about cofibrations can be deduced from the corresponding statement for $\operatorname{Ch}^{\leqslant 0}(\mathscr{A})$, which is contained in \cite[Theorem 4.65]{kelly_homotopy_2021}.
\end{proof}
\begin{thm}
With notation as above. The functor \eqref{eq:passtosimplicial} induces an equivalence of $\infty$-categories
\begin{equation}
 N(s\mathscr{A})[W^{-1}] \xrightarrow[]{\sim}   \operatorname{sInd}(N(\mathscr{P})).
\end{equation}
\end{thm}
\begin{proof}
The following argument is quite similar to the proof of \cite[Proposition 2.1.14]{schneiders_quasi-abelian_1999}. Let us write $\mathscr{A}^\prime := \operatorname{Fun}^\Pi(\mathscr{P}^\mathsf{op}, \mathsf{Set})$. This is an abelian category (in fact it is equal to the left heart of $\mathscr{A}$). The image of $\mathscr{P}$ under the Yoneda embedding gives a generating family of small (strongly) projective objects for $\mathscr{A}^\prime$. Hence we may consider the standard model structure on $s\mathscr{A}^\prime$. It is clear that there is an equivalence of categories
\begin{equation}
    s \mathscr{A}^\prime = \operatorname{Fun}^\Pi(\mathscr{P}^\mathsf{op},s \mathsf{Set})
\end{equation}
which identifies this model structure with the one from Proposition \ref{prop:simplicialmodelstructure}(i). Let $\mathscr{L}$ be the category obtained from $\mathscr{P}$ by freely adjoining (small) direct sums. By the smallness assumption on objects of $\mathscr{P}$, both functors $\mathscr{L} \to \mathscr{A}$ and $\mathscr{L} \to \mathscr{A}^\prime$ are fully-faithful, factor through the projective objects, and induce fully-faithful functors $s\mathscr{L} \to s\mathscr{A}$ and $s\mathscr{L} \to s\mathscr{A}^\prime$. By Corollary \ref{cor:simplicialmodelstructure}, the image of $s\mathscr{L}$ in $s\mathscr{A}$ or $s\mathscr{A}^\prime$ consists of fibrant-cofibrant objects. In this way we obtain a diagram
\begin{equation}
\begin{tikzcd}
	{N(s\mathscr{A})[W^{-1}] } && {N(s\mathscr{A}^\prime)[W^{-1}]  } \\
	& {N(s\mathscr{L})[H^{-1}]}
	\arrow[from=1-1, to=1-3]
	\arrow["{\simeq }", from=2-2, to=1-1]
	\arrow["\simeq"', from=2-2, to=1-3]
\end{tikzcd}
\end{equation}
in which $[H^{-1}]$ denotes the localization at simplicial homotopy equivalences. This completes the proof.
\end{proof}
\begin{cor}[Universal property of $D^{\leqslant 0}(\mathscr{A)}$]\label{cor:upDleq}
\begin{enumerate}[(i)]
    \item There is an equivalence
\begin{equation}
    D^{\leqslant 0}(\mathscr{A}) \simeq \operatorname{Fun}^\Pi(\mathscr{P}^\mathsf{op}, \infty\mathsf{Grpd}) = \operatorname{sInd}(N(\mathscr{P})). 
\end{equation}
    \item Let $\mathscr{D}$ be any cocomplete $\infty$-category. Composition with $N(\mathscr{P}) \to D^{\leqslant 0}(\mathscr{A})$ induces an equivalence of $\infty$-categories
\begin{equation}
    \operatorname{Fun}^L(D^{\leqslant 0}(\mathscr{A}), \mathscr{D}) \xrightarrow[]{\sim} \operatorname{Fun}^\Pi(N(\mathscr{P}), \mathscr{D}),
\end{equation}
where $\operatorname{Fun}^L$ denotes the full subcategory spanned by colimit-preserving functors.
\end{enumerate}
\end{cor}
\begin{proof}
(i): Combine Proposition \ref{prop:simplicialmodelstructure} and Proposition \ref{prop:DoldKan}. (ii): Combine (i) and \cite[Corollary 5.5.8.15(c)]{HigherToposTheory}.
\end{proof}
\begin{defn}\cite[Definition 5.5.7.1]{HigherAlgebra}\begin{enumerate}[(i)]
    \item Let $\mathscr{C}$ be an $\infty$-category admitting filtered colimits. An object $C \in \mathscr{C}$ is called \emph{compact} if $\operatorname{Hom}(C,-)$ commutes with filtered colimits. We let $\mathscr{C}^\omega \subseteq \mathscr{C}$ be the full subcategory spanned by compact objects.  
    \item An $\infty$-category $\mathscr{C}$ is called \emph{compactly generated} if there exists a small category $\mathscr{C}_0$ admitting finite colimits, and an equivalence of $\infty$-categories $\operatorname{Ind}(\mathscr{C}_0) \simeq \mathscr{C}$.
\end{enumerate}
\end{defn}
\begin{prop}\cite[\S 5.5.7]{HigherToposTheory}
Let $\mathscr{C}$ be a compactly generated $\infty$-category. Then the full subcategory $\mathscr{C}^\omega \subseteq \mathscr{C}$ is essentially small, admits finite colimits, and the inclusion induces an equivalence of $\infty$-categories $\operatorname{Ind}(\mathscr{C}^\omega) \xrightarrow[]{\sim} \mathscr{C}$. 
\end{prop}
\begin{lem}\label{lem:kappafiltered}
Let $\mathscr{C}$ be a compactly generated $\infty$-category. Then for any regular cardinal $\kappa$, one has that $\kappa$-small limits commute with $\kappa$-filtered colimits.
\end{lem}
\begin{proof}
This is well-known, but we include the proof for completeness. Let $p: I \times J\to \mathscr{C}$ be a diagram where $I$ is $\kappa$-filtered and $J$ is $\kappa$-small. We need to check that the canonical morphism
\begin{equation}
    \underset{I}{\operatorname{colim}} \, \underset{J}{\operatorname{lim}} p \to \underset{J}{\operatorname{lim}} \, \underset{I}{\operatorname{colim}}p
\end{equation}
is an equivalence. This can be checked after applying $\operatorname{Map}(C,-)$ for each $C \in \mathscr{C}^\omega$, reducing the proof of the Lemma to the case when $\mathscr{C}= \infty\mathsf{Grpd}$, which is \cite[Proposition 5.3.3.3]{HigherToposTheory}.
\end{proof}
\begin{lem}\label{lem:compactrightleft}
Let $F:\mathscr{C} \leftrightarrows \mathscr{D}:G$ be an adjunction between compactly generated $\infty$-categories. Then the right adjoint $G$ preserves filtered colimits if and only if the left adjoint $F$ preserves compact objects.
\end{lem}
\begin{proof}
This is again well-known, but we include the proof for completeness. The \emph{only if} direction is clear by adjunction. For the \emph{if} direction, suppose $F$ preserves compact objects. Let $p: I \to \mathscr{D}$ be a filtered diagram. We need to check that the canonical morphism
\begin{equation}
   \varphi: \underset{I}{\operatorname{colim}} G\circ p \to G \underset{I}{\operatorname{colim}} p
\end{equation}
is an equivalence in $\mathscr{C}$. This can be checked after applying $\operatorname{Map}(C,-)$ for each $C \in \mathscr{C}^\omega$. For such $C$ then $\operatorname{Map}(C,\varphi)$ is seen to be an equivalence, by adjunction and the fact that $F$ preserves compact objects. 
\end{proof}
\begin{prop}[Compact generation of $D^{\leqslant 0}(\mathscr{A)}$]\label{prop:compactgeneration}
\begin{enumerate}[(i)]
    \item The category $D^{\leqslant 0}(\mathscr{A)}$ is compactly generated.
    \item Let $j : N(\mathscr{P}) \to D^{\leqslant 0}(\mathscr{A}) $ be the inclusion. An object $C \in D^{\leqslant 0}(\mathscr{A})$ is compact if and only if the following holds: There exists a finite diagram $p: K \to \mathscr{P}$ such that $C$ is a retract of $\operatorname{colim} j\circ p $. 
\end{enumerate} 
\end{prop}
\begin{proof}
By \cite[Proposition 5.3.5.12]{HigherToposTheory}, the $\infty$-category $\operatorname{PSh}(N(\mathscr{P))}$ is compactly generated and the compact objects admit precisely the description as in (ii), c.f. \cite[Proposition 5.3.4.17]{HigherToposTheory}.
Now, $\operatorname{sInd}(N(\mathscr{P}))$ is a localization of $\operatorname{PSh}(N(\mathscr{P}))$, and the inclusion $\operatorname{sInd}(N(\mathscr{P})) \hookrightarrow \operatorname{PSh}(N(\mathscr{P}))$ preserves sifted colimits\footnote{This is obvious since $\operatorname{sInd}$ is formed by freely adjoining sifted colimits, whereas $\operatorname{PSh}$ is formed by freely adjoining all colimits.}, in particular filtered colimits. This implies that the left adjoint (the localization) preserves compact objects, giving both (i) and (ii). 
\end{proof}
\begin{rmk}\label{rmk:connectivecpct}
Because compact objects are stable under finite colimits and retracts, we arrive at the following alternative description of $D^{\leqslant 0}(\mathscr{A})^{\omega}$: it is the full subcategory generated under cones, suspensions and retracts by $N(\mathscr{P}) \subseteq D^{\leqslant 0}(\mathscr{A})$. One might call these ``connective $\mathscr{P}$-perfect complexes". 
\end{rmk}

\begin{example}
Let $K$ be a non-trivially valued non-Archimdedean field, let $\mathscr{A} =\mathsf{CBorn}_K$ and $\mathscr{P} = \{c_0(S)\}_S$ for $S$ ranging over (small) sets. We obtain
\begin{equation}
    D^{\leqslant 0}(\mathsf{CBorn}_K) \simeq \operatorname{sInd}(N(\{c_0(S)\}_S)) . 
\end{equation}
\end{example} 
\begin{defn} Let $\mathscr{C}$ be a pointed $\infty$-category with finite limits. We let $\operatorname{Sp}(\mathscr{C})$ denote the \emph{category of spectrum objects} of $\mathscr{C}$. It is the limit
\begin{equation}
    \operatorname{Sp}(\mathscr{C}) := \operatorname{lim}\left(\mathscr{C} \xleftarrow[]{\Omega} \mathscr{C} \xleftarrow[]{\Omega} \mathscr{C} \leftarrow \cdots \right)
\end{equation}
computed in $\mathsf{Cat}_\infty$. The category $\mathsf{Sp}$ of \emph{spectra} is defined as $ \mathsf{Sp} := \operatorname{Sp}(\infty\mathsf{Grpd}_*)$ where $\infty\mathsf{Grpd}_* := \infty\mathsf{Grpd}_{*/}$ denotes \emph{pointed} $\infty$-groupoids. 
\end{defn}
\begin{lem}\label{lem:DAtensorSp}
\begin{enumerate}[(i)]
\item The $t$-structure on $D(\mathscr{A})$ is both left and right complete;
\item There are $t$-exact equivalences of $\infty$-categories 
\begin{equation}
    D(\mathscr{A}) \simeq \operatorname{Sp}(D^{\leqslant0}(\mathscr{A})) \simeq \operatorname{Sp}(\operatorname{sInd}(N(\mathscr{P}))).
\end{equation}
\item There is a $t$-exact equivalence of $\infty$-categories 
\begin{equation}
    \operatorname{Sp}(\operatorname{sInd}(N(\mathscr{P}))) \simeq \operatorname{Fun}^\Pi(N(\mathscr{P})^\mathsf{op}, \mathsf{Sp}). 
\end{equation}
\item There is an equivalence of $\infty$-categories
\begin{equation}
    D^{\leqslant0}(\mathscr{A}) \otimes \mathsf{Sp} \simeq D(\mathscr{A}),
\end{equation}
where the tensor product on the left is the Lurie tensor product on $\mathsf{Pr}^L$. 
\end{enumerate}
\end{lem}
\begin{proof}
(i): Our working assumption that $\mathscr{A}$ is quasi-elementary implies that products in $\mathscr{A}$ are exact and coproducts in $\mathscr{A}$ are (strongly) exact, c.f \cite[Proposition 2.1.15]{schneiders_quasi-abelian_1999}. Hence both products and coproducts are $t$-exact in $D(\mathscr{A})$. From this it follows by \cite[Proposition 1.2.1.19]{HigherAlgebra} and its dual that $D(\mathscr{A})$ is left and right $t$-complete. 

(ii): The equivalence $D(\mathscr{A}) \simeq \operatorname{Sp}(D^{\leqslant0}(\mathscr{A}))$ is an immediate consequence of right-completeness of the $t$-structure, established in (i). 

(iii): This is essentially \cite[Remark 1.2]{LurieSpectralSchemes}, but let us reproduce the proof here for convenience. Because the endofunctor $\Omega$ of $\infty\mathsf{Grpd}_*$ commutes with (finite) products, we can regard $\mathsf{Sp}$ as the limit of the tower 
\begin{equation}
    \infty\mathsf{Grpd}_* \xleftarrow{\Omega} \infty\mathsf{Grpd}_* \xleftarrow{\Omega} \infty\mathsf{Grpd}_* \xleftarrow{\Omega}  \dots 
\end{equation}
computed in the $\infty$-category $\mathsf{Cat}_\infty^\Pi$ of $\infty$-categories admitting finite products, with finite-product preserving functors. Consequently we  obtain an equivalence 
\begin{equation}
\operatorname{Fun}^\Pi(N(\mathscr{P})^\mathsf{op}, \mathsf{Sp}) \simeq  \underset{\Omega}{\operatorname{lim}}\operatorname{Fun}^\Pi(N(\mathscr{P})^\mathsf{op}, \infty\mathsf{Grpd}_*)
\end{equation}
and we note that there is a canonical equivalence 
\begin{equation}
\operatorname{Fun}^\Pi(N(\mathscr{P})^\mathsf{op}, \infty\mathsf{Grpd}_*) \simeq \operatorname{Fun}^\Pi(N(\mathscr{P})^\mathsf{op}, \infty\mathsf{Grpd})_{*/}
\end{equation}
Hence the conclusion will follow if we can show that $\operatorname{sInd}(N(\mathscr{P}))$ was already pointed. This follows from the fact that the Yoneda embedding $j: N(\mathscr{P}) \hookrightarrow \operatorname{sInd}(N(\mathscr{P}))$ preserves finite coproducts and all limits \cite[Proposition 5.5.8.10]{HigherToposTheory}, and $N(\mathscr{P})$ has a zero object. 

(iv): Follows from (i), c.f. \cite[Example 4.8.1.23]{HigherAlgebra} (note that $D^{\leqslant 0}(\mathscr{A})$ is already pointed). 
\end{proof}
\begin{prop}[Compact generation of $D(\mathscr{A})$]
\label{prop:compactDA}\begin{enumerate}[(i)]
    \item The $\infty$-category $D(\mathscr{A})$ is compactly generated. 
    \item Let $j: N(\mathscr{P}) \hookrightarrow D(\mathscr{A})$ be the inclusion of $N(\mathscr{P})$ in degree $0$. An object $C \in D(\mathscr{A})$ is compact if and only if the following holds: There exists $n \geqslant 0$ and a finite diagram $p :K \to \mathscr{P}$ such that $C$ is a retract of $\operatorname{colim}\Omega^n  j \circ p$.  
\end{enumerate}
\end{prop}
\begin{proof}
Let us temporarily set $\mathscr{C} := D^{\leqslant 0}(\mathscr{A})$. By Proposition \ref{prop:compactgeneration}, $\mathscr{C}$ is compactly generated, so by Lemma \ref{lem:kappafiltered} finite limits commute with filtered colimits in $\mathscr{C}$. In particular the loops functor $\Omega: \mathscr{C} \to \mathscr{C}$ commutes with filtered colimits. Therefore we can view 
\begin{equation}
    \mathscr{C} \xleftarrow[]{\Omega} \mathscr{C} \xleftarrow[]{\Omega} \mathscr{C} \xleftarrow[]{\Omega} \cdots 
\end{equation}
as a diagram in the $\infty$-category $\mathsf{Pr}^R_\omega$ of compactly generated $\infty$-categories, with filtered-colimit preserving, right adjoint functors \cite[Definition 5.5.7.5]{HigherToposTheory}. 
Now \cite[Proposition 5.5.7.6]{HigherToposTheory} says that $\mathsf{Pr}^R_\omega$ is closed under limits in $\mathsf{Cat}_\infty$, giving (i).

(ii): We examine the proof of \emph{loc. cit.}. Let $\Omega^{\infty-n}: \operatorname{Sp}(\mathscr{C}) \to \mathscr{C}$ be the projection to the $n$\textsuperscript{th} component, which by (i) is a morphism in $\mathsf{Pr}^R_\omega$. Hence the left adjoint $\Sigma^{\infty-n}: \mathscr{C} \to \operatorname{Sp}(\mathscr{C})$ preserves compact objects. Together with Proposition \ref{prop:compactgeneration} this implies that all objects as in (ii) are indeed compact. Moreover, \cite[Lemma 6.3.3.6]{HigherToposTheory} implies that the identity functor on $\operatorname{Sp}(\mathscr{C})$ can be written as 
\begin{equation}
  \operatorname{id} \simeq \underset{n}{\operatorname{colim}} \Sigma^{\infty-n} \Omega^{\infty-n},
\end{equation}
so $\operatorname{Sp}(\mathscr{C})$ is generated under filtered colimits by the essential images of the $\Sigma^{\infty-n}$. Since $\mathscr{C}$ is generated under filtered colimits by objects as in Proposition \ref{prop:compactgeneration}(ii), we see that objects as in (ii) generate $\operatorname{Sp}(\mathscr{C})$ under filtered colimits. A retract argument then shows that every compact object has the form as in (ii).  
\end{proof}
\begin{prop}[Universal property of $D(\mathscr{A})$]\label{prop:DAprop} \begin{enumerate}[(i)]
    \item Let $\mathscr{D}$ be any stable presentable $\infty$-category. There is an equivalence of $\infty$-categories 
\begin{equation}
    \operatorname{Fun}^L(D(\mathscr{A}), \mathscr{D}) \simeq \operatorname{Fun}^\amalg(N(\mathscr{P}), \mathscr{D}). 
\end{equation}
where $\operatorname{Fun}^L$ denotes the colimit-preserving functors and $\operatorname{Fun}^\amalg$ denotes the finite-coproduct preserving functors.
\item Let $\mathscr{D}$ be any stable presentable $\infty$-category with $t$-structure $(\mathscr{D}^{\leqslant 0}, \mathscr{D}^{\geqslant 0})$. There is an equivalence of $\infty$-categories 
\begin{equation}
    \operatorname{Fun}^\prime(D(\mathscr{A}), \mathscr{D}) \simeq \operatorname{Fun}^\amalg(N(\mathscr{P}), \mathscr{D}^{\leqslant 0}). 
\end{equation}
where $\operatorname{Fun}^\prime$ denotes the colimit-preserving and right $t$-exact functors.
\end{enumerate}
\end{prop}
\begin{proof}
(i): By Proposition \ref{prop:DAprop} we have $D^{\leqslant 0}(\mathscr{A}) \otimes \mathsf{Sp} \simeq D(\mathscr{A})$ computed in the symmetric monoidal $\infty$-category $\mathsf{Pr}^L$. Now $\operatorname{Fun}^L(-,-)$ is the internal Hom in $\mathsf{Pr}^L$. Hence 
\begin{equation}
    \operatorname{Fun}^L(D^{\leqslant 0}(\mathscr{A}) \otimes \mathsf{Sp},\mathscr{D}) \simeq \operatorname{Fun}^L(D^{\leqslant 0}(\mathscr{A}) , \operatorname{Fun}^L(\mathsf{Sp},\mathscr{D})),
\end{equation}
but since $\mathscr{D}$ is stable and presentable we have $\operatorname{Fun}^L(\mathsf{Sp},\mathscr{D}) \simeq \mathscr{D}$, proving (i). (ii): Clear from (i). 
\end{proof}
\begin{rmk}
As in Remark \ref{rmk:connectivecpct}, we arrive at the following alternative description of $D(\mathscr{A})^{\omega}$: it is the full subcategory generated under cones, shifts and retracts by $N(\mathscr{P}) \subseteq D(\mathscr{A})$. One might call these ``$\mathscr{P}$-perfect complexes". 
\end{rmk}
\subsection{Monoidal structure on $D(\mathscr{A})$}\label{sec:monoidalstr}
In this section we continue with the same assumptions as in \S\ref{sec:universalpropertyofderived} but further assume that:
\begin{itemize}
    \item[$\star$] $\mathscr{A}$ is endowed with a closed symmetric monoidal structure $(\mathscr{A}, \otimes, \underline{\operatorname{Hom}})$;
    \item[$\star$] the monoidal structure on $\mathscr{A}$ restricts to $\mathscr{P}$;
    \item[$\star$] every object of $\mathscr{P}$ is flat. 
\end{itemize}
\begin{example}
Let $K$ be a non-trivially valued non-Archimedean field, let 
\begin{equation}
    \mathscr{A} = (\mathsf{CBorn}_K, \widehat{\otimes}_K, \underline{\operatorname{Hom}}_K),
\end{equation} 
and $\mathscr{P} = \{c_0(S)\}_S$ for $S$ ranging over (small) sets. There are canonical isomorphisms
\begin{equation}
    c_0(S) \widehat{\otimes}_K c_0(S^\prime) \xrightarrow[]{\sim} c_0(S \times S^\prime), 
\end{equation}
so that $\mathscr{P} \subseteq \mathscr{A}$ satisfies all the above assumptions. 
\end{example}
Kelly has proved the following (the same result also holds for unbounded complexes, but for this section we only need the result in the connective case):
\begin{thm}\cite[Theorem 4.69]{kelly_homotopy_2021}
Under the above assumptions. The projective model structure on $\operatorname{Ch}^{\leqslant 0}(\mathscr{A})$ is monoidal, i.e., $(\operatorname{Ch}^{\leqslant 0}(\mathscr{A}), \otimes)$ is a monoidal model category. 
\end{thm}
As a consequence of the dictionary between model categories and $\infty$-categories we obtain the following:  
\begin{cor}
The $\infty$-category $D^{\leqslant 0}(\mathscr{A})$ is presentably symmetric monoidal, when endowed with the derived tensor product. 
\end{cor}
\begin{thm}
There exists a symmetric monoidal structure  $\otimes_{\operatorname{Day}}$ (called  \emph{Day convolution}) on $\operatorname{sInd}(N(\mathscr{P}))$ which is characterised up to equivalence by the following properties: 
\begin{enumerate}[(i)]
    \item The Yoneda embedding $j: N(\mathscr{P}) \hookrightarrow \operatorname{sInd}(N(\mathscr{P}))$ extends to a symmetric monoidal functor; 
    \item $\otimes_{\mathrm{Day}}$ commutes with colimits separately in each variable.
\end{enumerate}
\end{thm}
\begin{proof}
The Theorem with \emph{sifted colimits} in place of \emph{all colimits} in (ii) is \cite[Proposition 4.8.1.10]{HigherAlgebra}, taking $\mathcal{K} = \emptyset$ and $\mathcal{K}^\prime$ to be the collection of sifted simplicial sets, in the notations of \emph{loc. cit.}. To get the full statement of (ii), one argues \emph{mutandis mutatis} as in \cite[Proposition 4.8.1.14]{HigherAlgebra}, replacing the use of \cite[Proposition 5.5.1.9]{HigherToposTheory} with \cite[Remark 5.5.8.16(3)]{HigherToposTheory}.
\end{proof}
\begin{cor}\label{cor:monoidalUP}
\begin{enumerate}[(i)]
    \item The equivalence of Corollary \ref{cor:upDleq}(i) upgrades to an equivalence of presentably symmetric monoidal $\infty$-categories:
\begin{equation}
    (\operatorname{sInd}(N(\mathscr{P})), \otimes_{\mathrm{Day}}) \xrightarrow[]{\sim} (D^{\leqslant 0}(\mathscr{A}), \otimes^\mathbf{L}).
\end{equation}
    \item Let $\mathscr{D}$ be any symmetric monoidal $\infty$-category such that $\mathscr{D}$ is cocomplete and the tensor product $\mathscr{D} \times \mathscr{D} \to \mathscr{D}$ preserves colimits separately in each variable. Then there is an equivalence of $\infty$-categories 
    \begin{equation}
        \operatorname{Fun}^{L,\otimes}(D^{\leqslant 0}(\mathscr{A}), \mathscr{D}) \simeq  \operatorname{Fun}^{\amalg,\otimes}(N(\mathscr{P}), \mathscr{D}),
    \end{equation}
    where $\operatorname{Fun}^{L,\otimes}$ (resp. $\operatorname{Fun}^{\Pi,\otimes}$) denotes colimit-preserving (resp. finite product preserving) symmetric monoidal functors. 
\end{enumerate}
\end{cor}
\begin{proof}
(i): Both tensor products preserve colimits (separately in each variable) and restrict to the symmetric monoidal structure on $N(\mathscr{P})$. (ii): Follows from (i) using \cite[Remark 5.5.8.16(3)]{HigherToposTheory}. 
\end{proof}
\begin{prop}\label{prop:derivedtensor}
There is a unique (up to equivalence) symmetric monoidal structure $\otimes^\mathbf{L}$ on $D(\mathscr{A})$ with the following properties: 
\begin{enumerate}[(i)]
    \item The inclusion $D^{\leqslant 0}(\mathscr{A}) \hookrightarrow D(\mathscr{A})$ extends to a symmetric monoidal functor;
    \item $\otimes^\mathbf{L}$ commutes with colimits separately in each variable. In particular $D(\mathscr{A})$ is \emph{presentably symmetric monoidal}.
\end{enumerate}
\end{prop}
\begin{proof}
This follows immediately by interpreting the formula 
\begin{equation}
    D^{\leqslant 0}(\mathscr{A}) \otimes \mathsf{Sp} \simeq D(\mathscr{A}) 
\end{equation}
coming from Proposition \ref{prop:DAprop}(iv), as a tensor product of \emph{commutative algebra objects} in $\mathsf{Pr}^L$.  
\end{proof}
Because $D(\mathscr{A})$ is \emph{presentably symmetric monoidal}, it is in particular closed symmetric monoidal. We write 
\begin{equation}
R\underline{\operatorname{Hom}}_\mathscr{A}(-,-): D(\mathscr{A})^\mathsf{op} \times D(\mathscr{A}) \to D(\mathscr{A})
\end{equation}
for the internal $\operatorname{Hom}$-bifunctor on $D(\mathscr{A})$.
\begin{lem}
\begin{enumerate}[(i)]
    \item If $C^\bullet, C^{\prime, \bullet}$ are compact objects of $D(\mathscr{A})$, then so is $C^\bullet \otimes^\mathbf{L} C^{\prime,\bullet}$. 
    \item If $C^\bullet \in D(\mathscr{A})$ is compact then the functor $R\underline{\operatorname{Hom}}_\mathscr{A}(C^\bullet,-)$ commutes with filtered colimits. 
\end{enumerate}
\end{lem}
\begin{proof}
(i): By Proposition \ref{prop:derivedtensor}, the derived tensor product restricts to the ordinary tensor product on $N(\mathscr{P})$, and commutes with colimits (in particular suspensions and hence loops since $\mathscr{D}$ is stable) separately in each variable. Therefore the claim follows from the explicit description of Proposition \ref{prop:compactDA}(ii). (ii): By Lemma \ref{lem:compactrightleft}, this is immediate from (i). 
\end{proof}
\begin{rmk}
Using Proposition \ref{prop:derivedtensor} together the model stucture on $\operatorname{Ch}(\mathscr{A})$ and the formalism of bounded resolving classes (c.f. \cite[\S 2.3.2]{kelly_homotopy_2021}, \cite{SpaltensteinUnbounded}) gives a way to calculate derived tensor products:
\begin{itemize}
    \item[$\star$] If $M^\bullet, N^\bullet\in \operatorname{Ch}^-(\mathscr{A})$, then one can find a degreewise projective complex $P^\bullet$ with an acyclic fibration $P^\bullet \to M^\bullet$. Then $ P^\bullet \otimes N^\bullet \xrightarrow[]{\sim} M^\bullet \otimes^\mathbf{L} N^\bullet$ in $D(\mathscr{A})$.  
    \item[$\star$] If $M^\bullet, N^\bullet \in D(\mathscr{A})$, then there exists a direct system $\{P^\bullet_n\}_{n \geqslant -1}$ of bounded-above degreewise projective complexes, with a system of acyclic fibrations $P_n^\bullet \to \tau^{\leqslant n}M^\bullet$. Put $P^\bullet := \operatorname{colim}_{n \geqslant -1} P_n^\bullet$. Then \begin{equation}
    M^\bullet \otimes^\mathbf{L} N^\bullet \simeq  \underset{{n, m \geqslant -1}}{\operatorname{colim}} (P_n^\bullet \otimes \tau^{\leqslant m} N^\bullet) \simeq P^\bullet \otimes N^\bullet. 
    \end{equation}
\end{itemize}
\end{rmk}

\section{Theory of abstract six-functor formalisms}\label{subsec:sixf1}
\begin{defn}\cite[Definition A.5.1]{mann_p-adic_2022}
A \emph{geometric setup} is a pair $(\mathscr{C},E)$ where $\mathscr{C}$ is an $\infty$-category and $E$ is a collection of morphisms of $\mathscr{C}$ such that:
\begin{itemize}
    \item[$\star$] $E$ contains all equivalences, is stable under composition, and pullbacks of morphisms in $E$ by arbitrary morphisms of $\mathscr{C}$ exist and remain in $E$.
\end{itemize}
\end{defn}
\begin{rmk}[Important remark]
Our definition of a geometric setup does not require $E$ to satisfy the \emph{right-cancellation property} (which says that $E$ is closed under the formation of diagonals). In particular, our convention follows \cite[\S A.5]{mann_p-adic_2022} which is different to \cite{heyer_6-functor_2024} in this way. 
\end{rmk}
Given a geometric setup one can define an $\infty$-category $\operatorname{Corr}(\mathscr{C},E)$, which can be described informally as follows \cite[\S A.5]{mann_p-adic_2022}. 

\begin{itemize}
    \item[$\star$] The objects of $\operatorname{Corr}(\mathscr{C},E)$ are the same as those of $\mathscr{C}$. 
    \item[$\star$] Morphisms $X \dashrightarrow Y$ of $\operatorname{Corr}(\mathscr{C},E)$  are given by spans $X \xleftarrow{g} U \xrightarrow{f} Y$ with $f \in E$. The composite of $X \leftarrow U \rightarrow Y$ and $Y \leftarrow V \rightarrow Z$ is given by the composed span $X \leftarrow U \leftarrow U \times_Y V \rightarrow V \rightarrow Z$. 
    \item[$\star$] $\operatorname{Corr}(\mathscr{C},E)$ has a monoidal structure built from the coCartesian monoidal structure on $\mathscr{C}^{\mathsf{op}}$. 
\end{itemize}
A lax-monoidal functor $Q: \operatorname{Corr}(\mathscr{C},E) \to \mathsf{Cat}_\infty$, (where the latter is endowed with the Cartesian monoidal structure), determines functors
\begin{equation}
\begin{aligned}
g^* := Q(X \xleftarrow[]{g} Y = Y): Q(X) \to Q(Y) && \text{ and } \\
f_! := Q(X = X \xrightarrow[]{f} Y): Q(X) \to Q(Y) && \text{ and }\\
\otimes_X : Q(X) \times Q(X) \to Q(X).
\end{aligned}
\end{equation}
\begin{defn}\cite{liu_enhanced_2017,GaitsgoryStudy1,mann_p-adic_2022}. \label{defn:sixfunctors}
\begin{itemize}
    \item[$\star$] A \emph{six-functor formalism} is a lax-monoidal functor\footnote{Where $\mathsf{Cat}_\infty$ is endowed with the Cartesian monoidal structure.} $Q: \operatorname{Corr}(\mathscr{C},E) \to \mathsf{Cat}_\infty$ such that $g^*$ and $f_!$ admit right adjoints for every morphism $g$ in $\mathscr{C}$ and every $f \in E$, and $M \otimes_X -$ admits a right adjoint for every $M \in  Q(X)$.  
    \item[$\star$] The right adjoints are denoted $g_*, f^!$, and $\underline{\operatorname{Hom}}_X(M,-)$, respectively.
\end{itemize}
\end{defn}
\begin{rmk}\label{rmk:sixfunctorsbasic}
With notations as in Definition \ref{defn:sixfunctors}. The following basic identities are valid in any six-functor formalism. 
\begin{enumerate}[(i)]
    \item(Projection formula). Let $[f:X \to Y] \in E$. There is a canonical equivalence
    \begin{equation}\label{eq:projectionformula}
        f_! \otimes_Y \operatorname{id} \simeq f_!(\operatorname{id} \otimes_X f^*)
    \end{equation}
    of functors $Q(X) \times Q(Y) \to Q(Y)$.
    \item Let $M \in Q(X)$ and $[f: X \to Y] \in E$. There is a canonical equivalence 
    \begin{equation}\label{eq:adjointtoprojectionformula}
        f_*\underline{\operatorname{Hom}}_X(M,f^!(-)) \simeq \underline{\operatorname{Hom}}_Y(f_!M,-) 
    \end{equation}
    of functors $Q(Y) \to Q(Y)$. This follows by passing to adjoints in the projection formula \eqref{eq:projectionformula}. 
    \item(Base-change). Suppose that we are given a Cartesian square
    \begin{equation}
\begin{tikzcd}
	{X^\prime} & {Y^\prime} \\
	X & Y
	\arrow["{f^\prime}", from=1-1, to=1-2]
	\arrow["{g^\prime}", from=1-1, to=2-1]
	\arrow["\lrcorner"{anchor=center, pos=0.125}, draw=none, from=1-1, to=2-2]
	\arrow["g", from=1-2, to=2-2]
	\arrow["f", from=2-1, to=2-2]
\end{tikzcd}
    \end{equation}
with $f \in E$ (hence also $f^\prime \in E$). Then:
\begin{enumerate}
    \item There is a canonical equivalence $
    g^*f_! \simeq f^\prime_! g^{\prime,*}$ of functors $Q(X) \to Q(Y^\prime)$. 
\item There is a canonical equivalence $
    f^!g_* \simeq g^\prime_* f^{\prime,!}$ of functors $Q(Y^\prime) \to Q(X)$. This follows by passing to right adjoints in the previous. 
\end{enumerate}
\end{enumerate}
\end{rmk}

We may prefer to equivalently view a six-functor formalism as a map of operads $\operatorname{Corr}(\mathscr{C},E) \to \mathsf{Cat}_\infty^\otimes$. Let $(\mathscr{C},E)$ be a geometric setup and let $Q: \operatorname{Corr}(\mathscr{C},E) \to \mathsf{Cat}_\infty^\otimes$ be a six-functor formalism. Let $\mathscr{C}_E$ be the subcategory of $\mathscr{C}$ where we only allow morphisms from $E$. There are functors $\mathscr{C}^\mathsf{op}  \to  \operatorname{Corr}(\mathscr{C},E) $ and $\mathscr{C}_E \to \operatorname{Corr}(\mathscr{C},E)$. On objects these are both induced by the identity; on morphisms the former sends $[g:X \to Y]$ to $[Y \xleftarrow[]{g} X = X]$ and the latter sends $[f: X \to Y] \in E$ to $[X = X \xrightarrow[]{f} Y]$. Via these functors we can restrict $Q$ and obtain functors
\begin{equation}
\begin{aligned}
    Q^*: \mathscr{C}^\mathsf{op} \to \mathsf{Cat}_\infty && \text{ and } && Q_!: \mathscr{C}_E \to \mathsf{Cat}_\infty.
\end{aligned}
\end{equation}
By passing to right adjoints we obtain
\begin{equation}
\begin{aligned}
     Q_*: \mathscr{C} \to \mathsf{Cat}_\infty && \text{ and } && Q^!: \mathscr{C}_E^\mathsf{op} \to \mathsf{Cat}_\infty.
\end{aligned}
\end{equation}
In this context, we can make the following definition. 
\begin{defn}\cite[Definition 4.14]{ScholzeSixFunctors}.\label{defn:universal!descent}
    \begin{enumerate}[(i)]
    \item We say that a morphism $f: X \to Y$ in $\mathscr{C}$ is of \emph{$*$-descent} if the canonical morphism 
    \begin{equation}
        Q^*(Y) \to \underset{[m] \in \Delta}{\operatorname{lim}} Q^*(X^{m+1/Y}) 
    \end{equation}
    is an equivalence. We say that a morphism $f: X \to Y$ is of \emph{universal $*$-descent} if, for every $Z \in \mathscr{C}$ with a morphism $Z \to Y$, the base-change $X \times_Y Z \to Z$ is of $*$-descent.
    \item We say that a morphism $[f: X \to Y] \in E$ is of \emph{$!$-descent} if the canonical morphism 
    \begin{equation}
        Q^!(Y) \to \underset{[m] \in \Delta}{\operatorname{lim}} Q^!(X^{m+1/Y}) 
    \end{equation}
    is an equivalence. We say that a morphism $f: X \to Y$ is of \emph{universal $!$-descent} if, for every $Z \in \mathscr{C}$ with a morphism $Z \to Y$, the base-change $X \times_Y Z \to Z$ is of $!$-descent.
\end{enumerate}
\end{defn}
Lemma \ref{lem:descentmonadicity} below can be viewed as the ``easy direction" in a higher-categorical version of the Bénabou--Roubaud theorem. In the other direction, one has Lurie's Beck--Chevalley condition \cite[Corollary 4.7.5.3]{HigherAlgebra}.
\begin{lem}\label{lem:descentmonadicity}
Fix a geometric setup $(\mathscr{C},E)$ and a six-functor formalism $Q: \operatorname{Corr}(\mathscr{C},E) \to \mathsf{Cat}^\otimes_\infty$. 
\begin{enumerate}[(i)]
    \item Let $f: X \to Y$ be a morphism of $\mathscr{C}$ which is of $*$-descent, c.f. Definition \ref{defn:universal!descent}(i). Then $f^*$ induces an equivalence of categories $Q(Y) \simeq \operatorname{Comod}_{f^*f_*}Q(X)$, where the latter is the category of comodules over the comonad $f^*f_*$. 
    \item Let $f: X \to Y$ be a morphism of $E$ which is of $!$-descent, c.f. Definition \ref{defn:universal!descent}(ii). Then $f^!$ induces an equivalence of categories $Q(Y) \simeq \operatorname{Mod}_{f^!f_!}Q(X)$, where the latter is the category of modules over the monad $f^!f_!$. 
\end{enumerate}
\end{lem}
\begin{proof}
(i): This is \cite[Proposition 3.1.27]{Camargo_deRham}.

(ii):  We adapt the argument of \cite[Proposition 3.1.27]{Camargo_deRham}. The proof is an application of the Barr--Beck--Lurie theorem \cite[Theorem 4.7.3.5]{HigherAlgebra}. The assumption that $f$ satisfies $!$-descent implies that $f^!$ is conservative. It remains to show that $f^!$ preserves geometric realizations of $f^!$-split simplicial objects.

    Let $(M_m)_{[m] \in \Delta_{+}^\mathsf{op}}$ be an augmented simplicial object of $Q(Y)$ which becomes split (i.e., acquires extra degeneracies), after applying $f^!$. For each $n \geqslant 0$ let $f_{n+1} : Y^{n+1/X} \to X$ be the projection. 
    
    Then, for all $n \geqslant 1$, the augmented simplicial object $(f_{n+1}^!M_m)_{m \in \Delta_{+}^\mathsf{op}}$ of $Q(Y^{n+1/X})$ is split. Let us set $N_{n+1} = \operatorname{colim}_{[m] \in \Delta^\mathsf{op}}f_{n+1}^!M_m$. Note that the existence of the splitting implies that the collection $(N_{n+1})_{[n] \in \Delta^{\mathsf{op}}}$ is a Cartesian section of $Q(Y^{\bullet +1 /X})$.  Now, we compute 
    \begin{equation}
    \begin{aligned}
        \underset{{[m] \in \Delta^{\mathsf{op}}}}{\operatorname{colim}} M_m &\simeq  \underset{{[m] \in \Delta^{\mathsf{op}}}}{\operatorname{colim}}  \underset{{[n] \in \Delta^{\mathsf{op}}}}{\operatorname{colim}} f_{n+1,!}f_{n+1}^!M_m  \\
        &\simeq \underset{{[n] \in \Delta^{\mathsf{op}}}}{\operatorname{colim}}  f_{n+1,!} \underset{{[m] \in \Delta^{\mathsf{op}}}}{\operatorname{colim}} f_{n+1}^!M_m \\
        &\simeq \underset{{[n] \in \Delta^{\mathsf{op}}}}{\operatorname{colim}} f_{n+1,!}N_{n+1}.
    \end{aligned}
    \end{equation}
    Now, since $(N_{n+1})_{[n] \in \Delta^{\mathsf{op}}}$ is a Cartesian section, $!$-descent for $f$ says that 
    \begin{equation}
        f^! \operatorname{colim}_{[m] \in \Delta^{\mathsf{op}}} M_m = N_1 = \operatorname{colim}_{[m] \in \Delta^\mathsf{op}}f^!M_m.
    \end{equation}
    Therefore $f^!$ preserves geometric realizations of $f^!$-split simplicial objects and we are done. 
\end{proof}
\subsection{An extension formalism for abstract six-functor formalisms}\label{subsec:sixf2}
\begin{rmk}\label{rmk:contentofsixf2}
After the first version of this article was uploaded to arXiv in September 2024, essentially the same result as in this subsection appeared as \cite[Theorem 3.4.11]{heyer_6-functor_2024}. We claim no originality for the results of this subsection: the main result (Theorem \ref{thm:extensionformalism}) is just a re-hashing of \cite[Theorem 4.20]{ScholzeSixFunctors}, and all of the extension principles were already developed in Mann's thesis \cite[\S A.5]{mann_p-adic_2022}. We still include it because our formulation is perhaps slightly closer to \cite[Theorem 4.20]{ScholzeSixFunctors}: in particular we chose to keep the part about being ``stable under disjoint unions" because we find it useful. This property appears to have something to do with idempotent completeness of the ``representable objects", see the proof of Theorem \ref{thm:relativealgebraic}. 
\end{rmk}
In this subsection we will fix two geometric setups $(\mathscr{C},E_0)$ and $(\mathscr{C}_0,E_{00})$, and we will also fix a six-functor formalism
\begin{equation}
    Q: \operatorname{Corr}(\mathscr{C},E_0) \to \mathsf{Cat}_\infty^\otimes,
\end{equation}
subject to the following list of Assumptions. 
\begin{assum}\label{assum:extends}
    \begin{enumerate}[(i)]
    \item $\mathscr{C}$ admits all fiber products and all small coproducts. We will denote the initial object of $\mathscr{C}$ by $\emptyset$ and the final object by $*$.
    \item $\mathscr{C}_{0} \subseteq \mathscr{C}$ is a full subcategory stable under fiber products in $\mathscr{C}$.
    \item For all $X \in \mathscr{C}$, $Q(X)$ is presentable.
    \item A morphism $f : X \to Y$ of $\mathscr{C}$ belongs to $E_0$ if and only if for every $Z \in \mathscr{C}_0$ with a morphism $Z \to Y$, the base-change $[X \times_Y Z \to Z] \in E_{00}$. 
    \item For all $X \in \mathscr{C}$, the canonical morphism 
    \begin{equation}
        Q^*(X) \to \underset{Y \in {\mathscr{C}_0}_{/X}^\mathsf{op}}{\operatorname{lim}} Q^*(Y)
    \end{equation}
    is an equivalence. That is, $Q^*: \mathscr{C}^\mathsf{op} \to \mathsf{Cat}_\infty$ is the right Kan extension of  $\left.Q^*\right|_{\mathscr{C}_0^\mathsf{op}} \to \mathsf{Cat}_\infty$ along $\mathscr{C}_0^\mathsf{op} \to \mathscr{C}^\mathsf{op}$.
    \item Coproducts in $\mathscr{C}$ are \emph{disjoint} and \emph{universal}. This means that:
    \begin{enumerate}
        \item (Disjoint). For all $X, Y \in \mathscr{C}$, the morphism $ \emptyset \to X \times_{X \coprod Y} Y$ is an equivalence. 
        \item (Universal). For all small families $\{X_i \to Y\}_{i \in \mathcal{I}}$ of morphisms in $\mathscr{C}$ and any morphism $Z \to Y$ in $\mathscr{C}$, the canonical morphism 
        \begin{equation}
            \coprod_{i \in \mathcal{I}} (X_i \times _Y Z) \to \big(\coprod_{i \in \mathcal{I}}X_i\big) \times_Y Z
        \end{equation} 
        is an equivalence. 
    \end{enumerate}
    \item For all $X, Y \in \mathscr{C}$ the morphism $X \to X \coprod Y \in E_0$.
    \item $Q^!: \mathscr{C}_{E_0}^\mathsf{op} \to \mathsf{Cat}_\infty$ preserves small products. That is, for all small families $\{X_i\}_{i \in \mathcal{I}}$ of objects of $\mathscr{C}$, the natural morphism $Q^!\big(\coprod_{i \in \mathcal{I}} X_i\big) \to \prod_{i \in \mathcal{I}} Q^!(X_i)$ is an equivalence. This condition makes sense by (vii).  
    \item Let $\delta E_0$ be the class of morphisms in $\mathscr{C}$ whose diagonal belongs to $E_0$. Then we require that $E_0 \subseteq \delta E_0$\footnote{Note that this differs slightly to \cite[Definition 4.18]{ScholzeSixFunctors}(3). I think that this change is actually necessary in order to run the argument of \cite[Theorem 4.20]{ScholzeSixFunctors} since, without it, the hypothesis (d) in \cite[Proposition A.5.14]{mann_p-adic_2022} might not be satisfied.}.
\end{enumerate}
\end{assum}

\begin{defn}\cite[Definition 4.18]{ScholzeSixFunctors}\label{defn:Estabilityproperties}
Let $(\mathscr{C},E)$ be another geometric setup such that $E_0 \subseteq E$.
\begin{enumerate}[(i)]
\item We say that the class $E$ is \emph{stable under disjoint unions} if for all small families $\{X_i \to Y\}_{i \in \mathcal{I}}$ of morphisms of $E$, the morphism $\coprod_{i \in \mathcal{I}} X_i \to Y$ belongs to $E$. 
\item We say that the class $E$ is \emph{$*$-local on the target} if whenever $f:X \to Y$ is a morphism of $\mathscr{C}$ such that, for all $Z \in \mathscr{C}_0$ with a map $Z \to Y$, the base change $X \times_Y Z \to Z$ belongs to $E$, then $f \in E$. 
\item We say that $Q$ \emph{extends uniquely to} $(\mathscr{C},E)$ if: 
\begin{enumerate}
        \item There exists a six-functor formalism $Q^\prime: \operatorname{Corr}(\mathscr{C},E) \to \mathsf{Cat}_\infty^\otimes$ equipped with an equivalence $\left.Q^\prime\right|_{\operatorname{Corr}(\mathscr{C},E_0)} \simeq Q$.
        \item Whenever $Q^{\prime \prime}: \operatorname{Corr}(\mathscr{C},E) \to \mathsf{Cat}_\infty^\otimes$ is another six-functor formalism whose restriction to $\operatorname{Corr}(\mathscr{C},E_0)$ is equipped with an equivalence
        \begin{equation}\label{eq:restrictedQequivalence}
            \left.Q^\prime\right|_{\operatorname{Corr}(\mathscr{C},E_0)} \simeq  \left.Q^{\prime \prime} \right|_{\operatorname{Corr}(\mathscr{C},E_0)}
        \end{equation} 
        then there exists a unique (up to contractible choice) equivalence $Q^\prime \simeq Q^{\prime \prime}$ lifting \eqref{eq:restrictedQequivalence}. 
\end{enumerate}
\item Assume that $Q$ extends uniquely to $(\mathscr{C}, E)$. We say that $E$ is \emph{$!$-local on the source} if whenever $f: X \to Y$ is a morphism in $\mathscr{C}$ such that there exists $[g: X^\prime \to X] \in E$ of universal $!$-descent such that $fg \in E$, then $f \in E$. 
\item Assume that $Q$ extends uniquely to $(\mathscr{C},E)$. We say that $E$ is \emph{tame} if whenever $Y \in \mathscr{C}_0$ and $[f:X \to Y] \in E$ then there exists a small family $\{X_i \to Y\}_{i \in \mathcal{I}}$ of morphisms in $E_{00}$ and a morphism $\big[\coprod_{i \in \mathcal{I}} X_i \to X\big] \in E$ over $Y$ which is of universal $!$-descent. 
\end{enumerate}
\end{defn}
\begin{thm}\cite[Theorem 4.20]{ScholzeSixFunctors}\label{thm:extensionformalism}
With notations as introduced in this subsection. Under the Assumptions \ref{assum:extends}, there is a (minimal) collection of morphisms $E \supseteq E_0$ of $\mathscr{C}$ such that $Q$ extends uniquely to $(\mathscr{C},E)$ and $E$ is stable under disjoint unions, $*$-local on the target, is $!$-local on the source, is tame, and satisfies $E \subseteq \delta E$.  
\end{thm}
\begin{proof}This is the same as \cite[Theorem 4.20]{ScholzeSixFunctors}. We reproduce the proof here for convenience, and also to convince the reader that the result is true in our slightly more general context.  

Let $A$ be the class of all classes $E$ of morphisms in $\mathscr{C}$ such that $(\mathscr{C},E)$ is a geometric setup, $Q$ extends uniquely to $(\mathscr{C},E)$, and $E$ is tame and satisfies $E \subseteq \delta E$. By Assumptions \ref{assum:extends}(iv) and \ref{assum:extends}(ix), we have $E_0 \in A$, and this is minimal.

In the proof of \emph{loc. cit.} it was observed that filtered unions can be taken in the class $A$, and that this is a convenient way to organise the fact that we will have to iterate the extension principles of \cite[\S A.5]{mann_p-adic_2022} transfinitely many times. We will proceed in steps. The steps will be indexed by ordinals starting at $2$.

\textit{Step 2}: Let $E \in A$. We let $E^\prime$ be the collection of morphisms which can be written as $\coprod_i X_i \to Y$ such that each $[X_i \to Y] \in E$. By Assumptions \ref{assum:extends}(iii), \ref{assum:extends}(vi), \ref{assum:extends}(vii), and \ref{assum:extends}(viii), we may apply \cite[Proposition A.5.12]{mann_p-adic_2022} to deduce that $(\mathscr{C}, E^\prime)$ is a geometric setup and $Q$ extends uniquely to $(\mathscr{C},E^\prime)$. Due to Assumptions \ref{assum:extends}(vi)(b) and \ref{assum:extends}(viii), the class $E^\prime$ is again tame. Using Assumptions \ref{assum:extends}(vi) and (vii) we see that $E^\prime \subseteq \delta E^\prime$.
The class $E^\prime$ is clearly stable under disjoint unions. 
We conclude that any $E \in A$ can be minimally enlarged to $E^\prime \in A$ which is stable under disjoint unions.

\textit{Step 3}: Again, let $E \in A$. Let $E^\prime$ be the class of morphisms $f: Y \to X$ in $\mathscr{C}$ such that there exists $[g: Z \to Y] \in E$ of universal $!$-descent such that $fg \in E$. By Lemma \ref{lem:Step3lemma} below, $(\mathscr{C},E^\prime)$ is again a geometric setup, and $E^\prime$ is tame and satisfies $E^\prime \subseteq \delta E^\prime$. Now we can apply \cite[Proposition A.5.14]{mann_p-adic_2022} to extend $Q$ uniquely to $(\mathscr{C},E^\prime)$, taking the class $S$ of covers in \emph{loc. cit.} to be those of universal $!$-descent. We note in particular that the assumption (a) of \emph{loc. cit.} is satisfied due to Assumption \ref{assum:extends}(iii) and the assumption (d) of \emph{loc. cit.} is satisfied since $E \subseteq \delta E$. We conclude that any $E \in A$ can be minimally enlarged to $E^\prime \in A$ which is $!$-local on the source.  

\textit{Step $\omega$}: By setting $E_1 := E_0$ and alternately applying Steps 2 and 3, we obtain a chain
\begin{equation}
    E_1 \subseteq E_2 \subseteq E_3 \subseteq E_4 \subseteq \dots
\end{equation}
such that for even $n \geqslant 2$, $E_n$ is stable under disjoint unions and for odd $n \geqslant3$, $E_n$ is $!$-local on the source. One then sets 
\begin{equation}\label{eq:Eomega}
    E_\omega := \bigcup_{n \geqslant1} E_{n} \in A.
\end{equation}
Then $E_\omega$ is stable under disjoint unions and $!$-local on the source. 

\textit{Step $\omega+1$}: Now assume that we are given $E \in A$ which is stable under disjoint unions and $!$-local on the source. We let $E^\prime$ be the collection of morphisms $[f: X \to Y]$ of $\mathscr{C}$ such that for all $Z \in \mathscr{C}_0$ with a map $Z \to Y$, the pullback $[X \times_Y Z \to Z] \in E$. By Lemma \ref{lem:stepomega+1} below, $(\mathscr{C},E^\prime)$ is again a geometric setup, $E^\prime$ is tame and $E^\prime \subseteq \delta E^\prime$. Now one defines $\mathscr{C}^\prime \subseteq \mathscr{C}$ to be the full subcategory of $\mathscr{C}$ on objects $X$ which admit an $E$-morphism to an object of $\mathscr{C}_0$ and one sets $E^{\prime \prime}$ to be the restriction of $E$ to $\mathscr{C}^\prime$. It is not hard to see that the inclusion $\mathscr{C}^\prime \subseteq \mathscr{C}$ preserves pullbacks of edges in $E^{\prime \prime}$ and that $E^\prime$ consists precisely of those morphisms whose pullback to $\mathscr{C}^\prime$ belongs to $E^{\prime \prime}$. Therefore one may restrict $Q$ to $(\mathscr{C}^\prime, E^\prime)$ and then use \cite[Proposition A.5.16]{mann_p-adic_2022} to extend $Q$ back to $(\mathscr{C},E^\prime)$. In order to satisfy the uniqueness assumption in \emph{loc. cit.} and hence conclude that this is the unique extension of $Q$ from $(\mathscr{C}, E)$ to $(\mathscr{C},E^\prime)$, we need to see that for each $X \in \mathscr{C}$, the canonical morphism
\begin{equation}
    Q^*(X) \to \underset{Y^\prime \in (\mathscr{C}^{\prime}_{/X})^\mathsf{op}}{\operatorname{lim}} Q^*(Y^\prime)
\end{equation}
is an equivalence, i.e., that $Q^*$ is right Kan extended from $\mathscr{C}^\prime$. This is a consequence of Assumption \ref{assum:extends}(v) and the transitivity of Kan extensions \cite[\href{https://kerodon.net/tag/0314}{Tag 0314}]{kerodon}.

\textit{Step $\omega \cdot 2$}: The class $E_{\omega + 1}$ may no longer be $!$-local on the source or stable under disjoint unions. Therefore one iterates Steps 2 and 3 again to obtain a chain $E_{\omega + 1} \subseteq E_{\omega + 2} \subseteq E_{\omega +3} \subseteq \dots$ such that $E_{\omega + 2n}$ is stable under disjoint unions and $E_{\omega + 2n+ 1}$ is $!$-local on the source, for all $n \geqslant1$. One then sets $E_{\omega \cdot 2} := \bigcup_{n \geqslant1} E_{\omega + n} \in A$, which is stable under disjoint unions and $!$-local on the source. 

\textit{Step $\omega^2$}: Continuing in this way, we obtain an increasing sequence $\{E_{\alpha}\}_{\alpha < \omega ^2}$ of classes in $A$. For each $m \geqslant1$, $E_{\omega \cdot m + 1}$ is $*$-local on the target and $E_{\omega \cdot m}$ is $!$-local on the source and stable under disjoint unions. Therefore $E_{\omega^2}:= \bigcup_{\alpha < \omega^2} E_\alpha \in A$ is stable under disjoint unions, $*$-local on the target, and $!$-local on the source.   
\end{proof}
The following auxiliary Lemmas were used in the proof of Theorem \ref{thm:extensionformalism}.
\begin{lem}\label{lem:Step3lemma}
With notations as introduced in this subsection. Let $(\mathscr{C},E)$ be a geometric setup with $E_0 \subseteq E$ and assume that $Q$ extends uniquely to $(\mathscr{C},E)$. Let $E^\prime$ be the class of morphisms $f:Y \to X$ in $\mathscr{C}$ such that there exists $[g: Z \to Y] \in E$ of universal $!$-descent such that $fg \in E$. Then:
\begin{enumerate}[(i)]
    \item $E^\prime$ is stable under base-change. 
    \item $E^\prime$ is stable under composition.
    \item If $E$ satisfies the right cancellation property (that is, $E \subseteq \delta E$),  then so does $E^\prime$.
    \item If $E$ is tame then so is $E^\prime$. 
\end{enumerate}
\end{lem}
\begin{proof}
(i): This is clear.

(ii): Let $[f: Y \to X]$ and $[g: Z \to Y] \in E^\prime$, so that there exists $[h:S \to Y]$ and $[k:T \to Z] \in E$ of universal $!$-descent such that $gk$ and $fh \in E$. Consider the diagram 
    \begin{equation}
\begin{tikzcd}[cramped]
	{T \times_YS} \\
	T & S \\
	{Z } & Y & X
	\arrow["{h^\prime}", from=1-1, to=2-1]
	\arrow["{l^\prime}", from=1-1, to=2-2]
	\arrow["\ulcorner"{anchor=center, pos=0.125}, draw=none, from=1-1, to=3-2]
	\arrow["k", from=2-1, to=3-1]
	\arrow["l", from=2-1, to=3-2]
	\arrow["h", from=2-2, to=3-2]
	\arrow[from=2-2, to=3-3]
	\arrow["g", from=3-1, to=3-2]
	\arrow["f", from=3-2, to=3-3]
\end{tikzcd}
    \end{equation}
to be sure, in this diagram the parallelogram is Cartesian and we have written $l := gk$. By base change $l^\prime \in E$ and so $gkh^\prime = hl^\prime \in E$. Also $kh^\prime \in E$ is of universal $!$-descent because this class is stable under composition and base-change.

(iii): Let $[f: Y \to X]$ and $[g: Z \to Y]$ be morphisms of $\mathscr{C}$ such that $fg$ and $f \in E^\prime$, so that there exists $[h:S \to Y]$ and $[k: T \to Z]$ of universal $!$-descent such that $fgk$ and $fh \in E$. Consider again the diagram 
\begin{equation}
\begin{tikzcd}[cramped]
	{T \times_YS} \\
	T & S \\
	{Z } & Y & X
	\arrow["{h^\prime}", from=1-1, to=2-1]
	\arrow["{l^\prime}", from=1-1, to=2-2]
	\arrow["\ulcorner"{anchor=center, pos=0.125}, draw=none, from=1-1, to=3-2]
	\arrow["k", from=2-1, to=3-1]
	\arrow["l", from=2-1, to=3-2]
	\arrow["h", from=2-2, to=3-2]
	\arrow[from=2-2, to=3-3]
	\arrow["g", from=3-1, to=3-2]
	\arrow["f", from=3-2, to=3-3]
\end{tikzcd}
\end{equation}
one has $fgkh^\prime = fhl^\prime \in E$. Therefore, by the right cancellation property for $E$, $l^\prime \in E$. Therefore $gkh^\prime = hl^\prime \in E$. Also $kh^\prime \in E$ is of universal $!$-descent. Therefore $g \in E^\prime$. 

(iv): This follows by unravelling the definitions. Let $[f: Y \to X] \in E^\prime$ be such that $Y \in \mathscr{C}_0$, so that there exists $[g:Z \to Y] \in E$ of universal $!$-descent such that $fg \in E$. By tameness of $E$ there exists a small family $\{X_i \to X\}_{i \in \mathcal{I}}$ of morphisms of $E_{00}$ and a morphism $\coprod_{i \in \mathcal{I}} X_i \to Z$ over $X$ which is of universal $!$-descent. The composite $\coprod_{i \in \mathcal{I}} X_i \to Z \to X$ is then a morphism of universal $!$-descent over $X$. Therefore $E^\prime$ is tame. 
\end{proof}
\begin{lem}\label{lem:stepomega+1}
With notations as introduced in this subsection. Let $(\mathscr{C},E)$ be a geometric setup with $E_0 \subseteq E$ and assume that $Q$ extends uniquely to $(\mathscr{C},E)$. Let $E^\prime$ be the collection of morphisms $[f: X \to Y]$ of $\mathscr{C}$ such that for all $Z \in \mathscr{C}_0$ with a map $Z \to Y$, the pullback $[X \times_Y Z \to Z] \in E$. Then:
\begin{enumerate}[(i)]
    \item $E^\prime$ is stable under base-change. 
    \item If $E$ is tame then $E^\prime$ is tame. 
    \item If $E$ satisfies the right cancellation property (that is, $E \subseteq \delta E$), then so does $E^\prime$. 
    \item If $E$ is stable under disjoint unions, $!$-local on the source, and tame, then $E^\prime$ is stable under composition.
\end{enumerate}
\end{lem}
\begin{proof}
(i): This is clear.

(ii): This is clear since, by the definition, if $[f:X \to Y] \in E^\prime$ has $Y \in \mathscr{C}_0$ then $f \in E$. 

(iii): This is clear. 

(iv): Let $[f: Y \to X]$ and $[g: Z \to Y] \in E^\prime$, let $[S \to X]$ be a morphism from an object $S \in \mathscr{C}_0$ and set $T:= Y \times_X S$, $R:= Z \times_X S$ and let $f^\prime$ and $g^\prime$ be the pullbacks of $f$ and $g$. By tameness there exists a small family $\{S_i \to S\}_{i \in \mathcal{I}}$ of morphisms of $E_{00}$ and a morphism $\coprod_{i \in \mathcal{I}} S_i \to T$ over $S$ of universal $!$-descent. We can summarise this information in the diagram 
\begin{equation}
\begin{tikzcd}
	{\coprod_{i \in \mathcal{I}}(S_i\times_TR)} & {\coprod_{i \in \mathcal{I}}S_i} \\
	{R } & T & S \\
	Z & Y & X
	\arrow[from=1-1, to=1-2]
	\arrow[from=1-1, to=2-1]
	\arrow[from=1-2, to=2-2]
	\arrow[from=1-2, to=2-3]
	\arrow["{g^\prime}", from=2-1, to=2-2]
	\arrow[from=2-1, to=3-1]
	\arrow["{f^\prime}", from=2-2, to=2-3]
	\arrow[from=2-2, to=3-2]
	\arrow[from=2-3, to=3-3]
	\arrow["g", from=3-1, to=3-2]
	\arrow["f", from=3-2, to=3-3]
\end{tikzcd}
\end{equation}
where we used Assumption \ref{assum:extends}(vi)(b). Since the class $E$ is $!$-local on the source and stable under disjoint unions, to check that $g^\prime \in E$ it suffices to check that each $[S_i \times_T R \to S_i] \in E$. But this is true by the assumption that $g \in E^\prime$. Therefore $E^\prime$ is stable under composition.
\end{proof}
\subsection{The six-functor formalism of relative algebraic geometry}\label{subsec:sixf3}
In this section I will assume that
\begin{itemize}
    \item[$\star$] $\mathscr{V}$ is a stable presentably symmetric monoidal $\infty$-category. We write $\otimes$ for the monoidal structure on $\mathscr{V}$.
\end{itemize}
We define $\mathscr{E} := \mathsf{CAlg}(\mathscr{V})^\mathsf{op}$. We use the formal expression $\operatorname{Spec}(A)$ to denote the object of $\mathscr{E}$ corresponding to $A \in \mathsf{CAlg}(\mathscr{V})$. By \cite[Theorem 4.5.3.1]{HigherAlgebra}, (see also \cite[Remark 4.5.3.2]{HigherAlgebra}), there is a functor 
\begin{equation}
    \operatorname{QCoh}: \mathscr{E}^\mathsf{op} \to \mathsf{CAlg}(\mathsf{Pr}^L_{\mathsf{st}}),
\end{equation}
which is given on objects by
\begin{equation}
\operatorname{QCoh}(\operatorname{Spec}(A)) := \operatorname{Mod}_A\mathscr{V},
\end{equation}
and sends $f: \operatorname{Spec}(A) \to \operatorname{Spec}(B)$ to $f^* = A \otimes_B -$. The functor $f^*$ admits a right adjoint $f_*$ which is nothing but the forgetful functor at the level of modules.
\begin{prop}\label{prop:QcohBasic}
The functor $ \operatorname{QCoh}: \mathscr{E}^\mathsf{op} \to \mathsf{CAlg}(\mathsf{Pr}^L_{\mathsf{st}})$ extends to a six-functor formalism 
\begin{equation}\label{eq:QCohbasic}
    \operatorname{QCoh}: \operatorname{Corr}(\mathscr{E}, \operatorname{all})^\otimes \to \mathsf{Pr}^{L, \otimes}_{\mathsf{st}}
\end{equation}
such that for every morphism $f$ in $\mathscr{E}$ one has $f_*= f_!$. 
\end{prop}
\begin{proof}
This is a straightforward application of \cite[Proposition A.5.10]{mann_p-adic_2022} applied to the decomposition $(I,P) = (\text{equivalences}, \text{all})$. Base change and the projection formula both follow from standard associativity properties of $\otimes$. 
\end{proof} 
\begin{notations}
\begin{itemize}
    \item[$\star$] Let $\mathsf{Aff} \subseteq \mathscr{E}$ be a full subcategory stable under fiber products and retracts. 
    \item[$\star$] Let $\tau$ be a Grothendieck topology on $\mathsf{Aff}$ such that $\operatorname{QCoh}^*$ is a sheaf in this topology. (By evaluation on the unit object, this in particular implies that $\tau$ is subcanonical).
    \item[$\star$] Let $\mathsf{Stk} := \operatorname{Shv}_\tau(\mathsf{Aff}, \infty\mathsf{Grpd})$ equipped with its natural topology as a topos, that is, the topology of effective epimorphisms.  
    \item[$\star$] Let $\mathrm{rep}$ be the collection of morphisms in $\mathsf{Stk}$ which are representable in $\mathsf{Aff}$.
\end{itemize}
\end{notations}
The Yoneda embedding induces a morphism of geometric setups $(\mathsf{Aff}, \mathrm{all}) \to (\mathsf{Stk},\mathrm{rep})$. By \cite[Proposition A.5.16]{mann_p-adic_2022} $\operatorname{QCoh}$ extends to a six-functor formalism 
\begin{equation}\label{eq:sixfrep}
    \operatorname{QCoh}: \operatorname{Corr}(\mathsf{Stk}, \mathrm{rep})^\otimes \to \mathsf{Pr}^{L,\otimes}_\mathsf{st}.
\end{equation}
such that $\operatorname{QCoh}^*$ is a sheaf in the effective-epimorphism topology.
\begin{lem}\label{lem:representablebasechange}
Let $g : Y \to X \in \operatorname{rep}$. Then $g_*$ satisfies base-change and the projection formula, and further $g_*$ is conservative. 
\end{lem} 
\begin{proof}
Using that $\tau$ is subcanonical, all statements follow from descent. 
\end{proof}
\begin{cor}\label{cor:lowershriekcor}
In the six-functor formalism \eqref{eq:sixfrep}, for every morphism $[g: Y \to X] \in \mathrm{rep}$, there is a natural equivalence $g_! \xrightarrow[]{\sim} g_*$. 
\end{cor}
\begin{proof}
By a straightforward application of \cite[Proposition A.5.10]{mann_p-adic_2022}, using Lemma \ref{lem:representablebasechange} above, one constructs a second six-functor formalism on $(\mathsf{Stk}, \mathrm{rep})$, such that for every $[g:Y \to X] \in \mathrm{rep}$ one has $g_! = g_*$. Now the unicity assertion in \cite[Proposition A.5.16]{mann_p-adic_2022} implies that this six-functor formalism is equivalent to the one constructed in \eqref{eq:sixfrep} above, which gives the Corollary. 
\end{proof}
Now we apply the \emph{extension formalism} of \S\ref{subsec:sixf2}. 
\begin{thm}[The six-functor formalism of relative algebraic geometry]\label{thm:relativealgebraic}
There exists a (minimal) class of edges $E \supseteq \mathrm{rep}$ of $\mathsf{Stk}$ such that $\operatorname{QCoh}$ extends to a six-functor formalism on $(\mathsf{Stk},E)$ and $E$ is stable under disjoint unions, $*$-local on the target, $!$-local on the source, is tame, and satisfies $E \subseteq \delta E$. 
\end{thm}
\begin{proof}
We will apply Theorem \ref{thm:extensionformalism} with 
\begin{equation}
    \begin{aligned}
        (\mathscr{C}_0, E_{00}) := (\mathsf{Aff}, \mathrm{all}),  && \text{ and } && (\mathscr{C}, E_0) := (\mathsf{Stk}, \mathrm{rep}).
    \end{aligned}
\end{equation}
We need to check that the Assumptions \ref{assum:extends}(i)-(ix) are satisfied. The Assumptions \ref{assum:extends}(i), (vi) are satisfied in any $\infty$-topos. The Assumption \ref{assum:extends}(v) follows from the fact that $\operatorname{QCoh}^*$ is left Kan extended along the Yoneda embedding, c.f. \cite[Proposition A.5.16]{mann_p-adic_2022}, and then \ref{assum:extends}(iii) follows because $\operatorname{QCoh}(Y)$ is presentable for every $Y \in \mathsf{Aff}$. The Assumption \ref{assum:extends}(ii) follows since the topology $\tau$ is subcanonical. Assumption \ref{assum:extends}(iv) is clear from the definition of a representable morphism and \ref{assum:extends}(ix) is also easily verified. So the only things to really check are Assumptions \ref{assum:extends}(vii) and \ref{assum:extends}(viii). 

For \ref{assum:extends}(vii), by base-change it suffices to show that the morphism $i: * \to * \coprod *$ into the first factor is representable. We note that this has a retraction $r$ (the fold map) and in fact one can\footnote{I would like to thank Sam Moore for a helpful discussion, and especially for explaining ``$\mathrm{Idem}$" to me.} write $ * = \underset{\operatorname{Idem}}{\operatorname{lim}}  (* \coprod *)   $ where the $\infty$-category  $\operatorname{Idem}$ is from \cite[\S 4.4.5]{HigherToposTheory}. Now fiber products commute with limits so given $\operatorname{Spec}(A) \to * \coprod *$  we can write $* \times_{* \coprod *} \operatorname{Spec}(A) = \underset{\operatorname{Idem}}{\operatorname{lim}} \operatorname{Spec(A)}$. Using that the Yoneda embedding is fully-faithful (since $\tau$ is subcanonical) and preserves all limits which exist in $\mathsf{Aff}$, we deduce that this is representable, since $\mathsf{Aff}$ is idempotent complete (by assumption, it is stable under retracts in $\mathscr{E}$). 

For Assumption \ref{assum:extends}(viii) let us take objects $X, Y \in \mathsf{Stk}$ and consider the morphisms $s:  X \to  X \coprod Y$ and $t: Y \to X \coprod Y$. By the previous these both belong to the class $\mathrm{rep}$, so that in particular one has base-change for $s_*$ and $t_*$ against the upper-shriek functors. Because $\operatorname{QCoh}^*$ is a sheaf, one has $\operatorname{id} \simeq s_* s^* \oplus t_* t^*$. Now applying $s^!$ and using base-change we deduce that $s^! \simeq s^*$.   

Now suppose we are given a small collection $\{X_i\}_{i \in \mathcal{I}}$ of objects of $\mathsf{Stk}$ and consider the morphisms $s_i: X_i \to \coprod_iX_i$ and $t_i: \coprod_{j \neq i}X_j \to \coprod_iX_i$. By the above argument we have $s_i^! \simeq s_i^*$. Since $\operatorname{QCoh}^*$ is a sheaf, one has an equivalence $\prod_is_i^* : \prod_i\operatorname{QCoh}(X_i) \xrightarrow[]{\sim} \operatorname{QCoh}(\coprod_i X_i)$. Hence also $\prod_is_i^! : \prod_i \operatorname{QCoh}(X_i) \xrightarrow[]{\sim} \operatorname{QCoh}(\coprod_i X_i)$ is an equivalence, as required. 
\end{proof}
\section{Derived rigid geometry}\label{sec:derivedrigidgeometry}
Now let $K/ \mathbf{Q}_p$ be a complete field extension. In this section we attempt to summarize the work of \cite{DAnG}, and take some shortcuts, to obtain a theory of derived rigid spaces which is good enough for our purposes. 

\subsection{Derived affinoid spaces}\label{subsec:derivedaffinoid}
\begin{defn}
We define $\mathsf{dAfndAlg}$ to be the full subcategory of monoids $A$ in $D_{\geqslant 0}(\mathsf{CBorn}_K)$ with the following properties.
\begin{enumerate}[(i)]
        \item $\pi_0(A)$ is a $K$-affinoid algebra. That is, it is the quotient of a (classical) Tate algebra in finitely many variables.
        \item For every $m \geqslant 0$, $\pi_m(A)$ is finitely-generated as a $\pi_0(A)$-module.
\end{enumerate}
\end{defn}
\begin{lem}\label{lem:dAfndAlgbasic1}
\begin{enumerate}[(i)]
    \item $\mathsf{dAfndAlg}$ is stable under pushouts in $\mathsf{Alg}(D_{\geqslant 0}(\mathsf{CBorn}_K))$.
    \item $\mathsf{dAfndAlg}$ is stable under finite products in $\mathsf{CAlg}(D_{\geqslant 0}(\mathsf{CBorn}_K))$.
    \item $\mathsf{dAfndAlg}$ is stable under retracts in $\mathsf{CAlg}(D_{\geqslant 0}(\mathsf{CBorn}_K))$. 
\end{enumerate}
\end{lem}
\begin{proof}
(i): Let $B \to A$ and $B \to C$ be morphisms in $\mathsf{dAfndAlg}$. Since $\pi_0$ is left adjoint to the inclusion of discrete objects in $\mathsf{CAlg}(D_{\geqslant 0}(\mathsf{CBorn}_K))$, it commutes with pushouts. In particular one has $\pi_0(A \widehat{\otimes}_B^\mathbf{L} C) \simeq \pi_0(A) \widehat{\otimes}_{\pi_0(B)}\pi_0(C)$. Now, the homology functors are valued in the \emph{left heart} and so, the claim that  $\pi_0(A \widehat{\otimes}_B^\mathbf{L} C)$ is an affinoid algebra will follow if we can show that the tensor product $\pi_0(A) \widehat{\otimes}_{\pi_0(B)}^{LH} \pi_0(C)$, which is taken with respect to the monoidal structure on the left heart, coincides with the completed tensor product of Banach spaces\footnote{I would like to thank Jack Kelly for explaining this argument to me.}. Let us set $A^\prime:= \pi_0(A), B^\prime := \pi_0(B)$ and $C^\prime := \pi_0(C)$. We first treat the case when $B^\prime = K$. Let us temporarily write $T_n := K \langle x_1, \dots, x_n\rangle$ and take a presentation 
\begin{equation}
    T_n^{\oplus l} \to T_n \twoheadrightarrow A^\prime.
\end{equation}
Tensoring with $C^\prime$ and using flatness of the Tate algebra gives 
\begin{equation}
     T_n^{\oplus l} \widehat{\otimes}_KC^\prime  \xrightarrow[]{f} T_n \widehat{\otimes}_KC^\prime \twoheadrightarrow A^\prime \widehat{\otimes}_K^{LH} C^\prime,
\end{equation}
so that $A^\prime \widehat{\otimes}_K^{LH}C^\prime$ is the cokernel of $f$. The image of $f$ is an ideal, and all ideals in affinoid algebras are closed, so $f$ is strict by the open mapping theorem. Hence $\operatorname{coker}f$ belongs to $\mathsf{CBorn}_K$ and coincides with $A^\prime \widehat{\otimes}_K C^\prime$ as required. In general, using this case we know that there is an exact sequence 
\begin{equation}
    A^\prime \widehat{\otimes}_K B^\prime \widehat \otimes_K C^\prime  \xrightarrow[]{g}   A^\prime \widehat \otimes_K C^\prime  \twoheadrightarrow A^\prime \widehat \otimes_K^{LH} C^\prime,
\end{equation}
and the image of $g$ is once again an ideal, so $g$ is strict by the same reasoning, and so $\operatorname{coker}g$ belongs to $\mathsf{CBorn}_K$ and coincides with $A^\prime \widehat\otimes_{B^\prime} C^\prime$. Now we claim that each $\pi_m(A \widehat{\otimes}_B^\mathbf{L} C)$ is finitely-generated as a $\pi_0(A) \widehat{\otimes}_{\pi_0(B)}\pi_0(C)$-module. This follows from the convergence of the Tor-spectral sequence\footnote{(Here and elsewhere, $\operatorname{Tor}$ denotes the homotopy groups of the derived \emph{completed} tensor product).} 
\begin{equation}
  E^2_{pq}:  \operatorname{Tor}_{\pi_*(B)}^p(\pi_*(A), \pi_*(C))_q \Rightarrow \pi_{p+q}(A \widehat{\otimes}^\mathbf{L}_B C),
\end{equation}
(c.f. \cite[Lemma 4.5.55]{DAnG}) combined with Noetherianity of affinoid algebras. (ii): This is clear from the isomorphism $\pi_*(\prod_{i=1}^n A_i) \cong \prod_{i=1}^n \pi_*(A_i)$ which holds for any finite collection of objects $\{A_i\}_{i=1}^n$ of objects of $\mathsf{dAfndAlg}$.  (iii): If $A$ is a retract of $B \in  \mathsf{dAfnd}$ then $\pi_*(A)$ is a retract of $\pi_*(B)$, and hence $A \in \mathsf{dAfnd}$. 
\end{proof}
\begin{defn}
A morphism $A \to B$ in $\mathsf{CAlg}(D(\mathsf{CBorn}_K))$ is called a \emph{homotopy epimorphism} if the codiagonal morphism $B \widehat{\otimes}^\mathbf{L}_A B \to B$ is an equivalence.
\end{defn}
\begin{defn}
A morphism $A \to B$ in $\mathsf{CAlg}(D_{\geqslant 0}(\mathsf{CBorn}_K))$ is called \emph{derived strong} if for every $m \geqslant 0$, the natural morphism
\begin{equation}\label{eq:Lderivedstrong}
        \pi_m(A) \widehat{\otimes}^\mathbf{L}_{\pi_0(A)} \pi_0(B) \to \pi_m(B)
    \end{equation}
    is an equivalence.
\end{defn}
\begin{defn}\label{defn:derivedrational}
A morphism $A\to B$ in  $\mathsf{dAfndAlg}$ is a \emph{derived rational localization} if:
\begin{enumerate}[(i)]
    \item $\pi_0(A) \to \pi_0(B)$ is a rational localization (of affinoid algebras in the classical sense).
    \item $A \to B$ is derived strong.
\end{enumerate}
We denote the class of derived rational localizations in $\mathsf{dAfndAlg}$ by $\mathscr{L}$.
\end{defn}
\begin{rmk}
Because each $\pi_m(A)$ is finitely-generated as a $\pi_0(A)$-module, and finitely-generated $\pi_0(A)$-modules are transverse to rational localizations of $\pi_0(A)$, condition (ii) in Definition \ref{defn:derivedrational} is equivalent to:
\begin{enumerate}
    \item[(ii)${}^\prime$] For every $m \geqslant 0$, the natural morphism
    \begin{equation}\label{eq:Lstrong}
        \pi_m(A) \widehat{\otimes}_{\pi_0(A)} \pi_0(B) \to \pi_m(B)
    \end{equation}
    is an equivalence. 
\end{enumerate}
In the terminology of \cite{DAnG}, one says that $A\to B$ is \emph{strong}. 
\end{rmk}
\begin{lem}\label{lem:Lclassproperties}
\begin{enumerate}[(i)]
    \item The class $\mathscr{L}$ contains all equivalences and is stable under composition.
    \item $\mathscr{L}$ is stable under base-change by arbitrary morphisms of $\mathsf{dAfndAlg}$.
    \item Every morphism of $\mathscr{L}$ is a homotopy epimorphism.
\end{enumerate}
\end{lem}
\begin{proof}
(i): It is clear that $\mathscr{L}$ contains all equivalences. Let $A \to B$ and $B \to C$ be morphisms of $\mathscr{L}$. Since rational localizations are stable under composition, the composite $\pi_0(A) \to \pi_0(B) \to \pi_0(C)$ is a rational localization. It is easy to show that $A \to C$ is (derived) strong, by using that $A \to B$ and $B \to C$ are (derived) strong together with associativity of $\widehat{\otimes}$. 

(iii): This is \cite[Proposition 2.6.165(2)]{DAnG}. We first observe that, by \cite[Theorem 5.16]{BBKNonArch}, the morphism $\pi_0(A) \to \pi_0(B)$ is a homotopy epimorphism, meaning that $\pi_0(B) \widehat{\otimes}_{\pi_0(A)}^\mathbf{L} \pi_0(B) \xrightarrow[]{\sim} \pi_0(B)$ is an isomorphism. Using this, and the derived strong property \eqref{eq:Lderivedstrong}, one has 
\begin{equation}\label{eq:Lishomotopyepi}
\begin{aligned}
    \pi_*(B) \widehat{\otimes}^\mathbf{L}_{\pi_*(A)} \pi_*(B) &\simeq \pi_*(A) \widehat{\otimes}^\mathbf{L}_{\pi_0(A)} \pi_0(B) \widehat{\otimes}^\mathbf{L}_{\pi_0(A)} \pi_0(B) \\
    &\simeq \pi_*(A) \widehat{\otimes}^\mathbf{L}_{\pi_0(A)} \pi_0(B) \\
    &\simeq \pi_*(B).
\end{aligned}
\end{equation}
Now we consider the $\operatorname{Tor}$-spectral sequence 
\begin{equation}
   E^2_{pq}: \operatorname{Tor}_{\pi_*(A)}^p(\pi_*(B),\pi_*(B))_q \Rightarrow \pi_{p+q}(B \widehat{\otimes}^\mathbf{L}_A B),
\end{equation}
and observe that, because of \eqref{eq:Lishomotopyepi}, this collapses on the first page. In combination with \eqref{eq:Lishomotopyepi} this gives $\pi_*(B \widehat{\otimes}^\mathbf{L}_A B) \cong \pi_*(B)$ which shows that $B \widehat{\otimes}^\mathbf{L}_A B \to B$ is an equivalence.

(ii): Let $A \to A^\prime$ be a further morphism and let $B^\prime$ be defined as the pushout (using Lemma \ref{lem:dAfndAlgbasic1}) in $\mathsf{dAfndAlg}$:
\begin{equation}
\begin{tikzcd}
	A & B \\
	{A^\prime} & {B^\prime}
	\arrow[from=1-1, to=1-2]
	\arrow[from=1-1, to=2-1]
	\arrow[from=1-2, to=2-2]
	\arrow[from=2-1, to=2-2]
	\arrow["\lrcorner"{anchor=center, pos=0.125, rotate=180}, draw=none, from=2-2, to=1-1]
\end{tikzcd}
\end{equation}
Since $\pi_0$ commutes with pushouts and rational localizations of classical affinoid algebras are stable under base-change, we see that $\pi_0(A^\prime) \to \pi_0(B^\prime)$ is a rational localization.

Since $\pi_0(A^\prime) \to \pi_0(B^\prime)$ is a rational localization, and $\pi_m(A^\prime)$ is finitely-generated as a module over $\pi_0(A^\prime)$, the canonical morphism
\begin{equation}\label{eq:pimAprimetransverse}
    \pi_m(A^\prime) \widehat{\otimes}^\mathbf{L}_{\pi_0(A^\prime)} \pi_0(B^\prime) \xrightarrow[]{\sim}  \pi_m(A^\prime) \widehat{\otimes}_{\pi_0(A^\prime)} \pi_0(B^\prime)
\end{equation}
is an equivalence. In the terminology of \cite{DAnG} one says that \emph{$\pi_m(A^\prime)$ is transversal to $\pi_0(B^\prime)$ over $\pi_0(A^\prime)$}.  
Additionally, we claim that the natural morphism
\begin{equation}\label{eq:Bprimehomotopyepi}
    B^\prime \widehat{\otimes}^\mathbf{L}_{A^\prime} B^\prime \xrightarrow[]{\sim} B^\prime
\end{equation}
is an equivalence. Indeed, we know that $B^\prime \simeq B \widehat{\otimes}^\mathbf{L}_A A^\prime$ and by (iii) we have $B \widehat{\otimes}^\mathbf{L}_AB \xrightarrow[]{\sim} B$. Therefore, by the associativity properties of $\widehat{\otimes}^\mathbf{L}$, we obtain \eqref{eq:Bprimehomotopyepi}. 

In order to prove (ii) what we needed to verify was that the morphism $A^\prime \to B^\prime$ satisfies the derived strong property \eqref{eq:Lderivedstrong}. We claim that this follows from the properties \eqref{eq:pimAprimetransverse} and \eqref{eq:Bprimehomotopyepi}\footnote{Indeed, this is \cite[Proposition 2.6.160]{DAnG}. However, in this specific setting the proof of \emph{loc. cit.} simplifies and so we record it here.}. 

By basic properties of the cotangent complex, c.f. \cite[Proposition 2.1.40, Proposition 2.1.41, Lemma 2.1.42]{DAnG}, the homotopy epimorphism property \eqref{eq:Bprimehomotopyepi} implies that the cotangent complex $\mathbf{L}_{B^\prime/A^\prime} \simeq 0$. Let $C$ be defined by the pushout in $\mathsf{dAfndAlg}$ (again using Lemma \ref{lem:dAfndAlgbasic1}):
\begin{equation}\label{eq:Cpushout}
\begin{tikzcd}
	{A^\prime} & {B^\prime} \\
	{\pi_0(A^\prime)} & C
	\arrow[from=1-1, to=1-2]
	\arrow[from=1-1, to=2-1]
	\arrow[from=1-2, to=2-2]
	\arrow[from=2-1, to=2-2]
	\arrow["\lrcorner"{anchor=center, pos=0.125, rotate=180}, draw=none, from=2-2, to=1-1]
\end{tikzcd}
\end{equation}
We claim that the natural morphism $C \to \pi_0(B^\prime)$ induced by the universal property of pushouts, is an equivalence. It is clear, by applying $\pi_0$ to \eqref{eq:Cpushout}, that this is an isomorphism on $\pi_0$. Therefore it suffices to show that $\mathbf{L}_{\pi_0(B^\prime)/C} \simeq 0$. There is a fiber sequence 
\begin{equation}
    \mathbf{L}_{C/\pi_0(A^\prime)} \widehat{\otimes}^\mathbf{L}_C \pi_0(B^\prime) \to \mathbf{L}_{\pi_0(B^\prime) / \pi_0(A^\prime)} \to \mathbf{L}_{\pi_0(B^\prime) /C}.
\end{equation}
Now $\mathbf{L}_{\pi_0(B^\prime) / \pi_0(A^\prime)} \simeq 0$ because $ \pi_0(A^\prime) \to \pi_0(B^\prime)$ is a homotopy epimorphism (being a rational localization). And $\mathbf{L}_{C/\pi_0(A^\prime)} \simeq 0$ because $\pi_0(A^\prime) \to C$ is a homotopy epimorphsim, by base-change. So indeed $\mathbf{L}_{\pi_0(B^\prime) /C}  \simeq 0$.

Now we know that the morphism
\begin{equation}
    \pi_0(A^\prime) \widehat{\otimes}^\mathbf{L}_A B^\prime \xrightarrow[]{\sim} \pi_0(B^\prime)
\end{equation}
is an equivalence. Using this and the transversality property \eqref{eq:pimAprimetransverse}, an inductive argument as described in \cite[Proposition 2.3.90]{DAnG}, guarantees that $A^\prime \to B^\prime$ is derived strong.

\end{proof}
\begin{cor}\cite[Remark 2.4.4]{mann_p-adic_2022}
Let $f, g$ be composable morphisms in $\mathsf{dAfndAlg}$. If $gf$ and $f$ both belong to $\mathscr{L}$ then so does $g$. 
\end{cor}
\begin{proof}
Let us write $f:A \to B$ and $g: B \to C$. One can write $g$ as the composite 
\begin{equation}\label{eq:algcancellation}
    B \to C \widehat{\otimes}^\mathbf{L}_A B \simeq C \widehat{\otimes}^\mathbf{L}_B B \widehat{\otimes}^\mathbf{L}_A B \simeq C,
\end{equation}
where we used that $f$ is a homotopy epimorphism, c.f. Lemma \ref{lem:Lclassproperties}(iii). The first map in \eqref{eq:algcancellation} is the base change of $gf: A \to C$, which belongs to $\mathscr{L}$ by Lemma \ref{lem:Lclassproperties}(ii).  
\end{proof}

Let $K[T_1,\dots, T_n]$ denote the polynomial algebra in $n$ variables endowed with the \emph{fine bornology}. As a bornological vector space one views $K[T_1,\dots, T_n]$ as the colimit of its finite-dimensional $K$-subspaces each endowed with their canonical bornology. This is a projective object in $\mathsf{CAlg}(D_{\geqslant 0}(\mathsf{CBorn}_K))$. In particular, given any $B \in \mathsf{dAfndAlg}$ and any morphism $K[T_1,\dots,T_n] \to \pi_0B$ there exists a lift $K[T_1,\dots,T_n] \to \pi_0B$ making the following diagram commute up to homotopy:
\begin{equation}
\begin{tikzcd}[cramped]
	& B \\
	{K[T_1,\dots,T_n]} & {\pi_0B}
	\arrow[from=1-2, to=2-2]
	\arrow[dotted, from=2-1, to=1-2]
	\arrow[from=2-1, to=2-2]
\end{tikzcd}
\end{equation}
Now suppose we are given $A \in \mathsf{dAfndAlg}$ and elements $a_0,\dots,a_n \in \pi_0A$ such that the ideal generated by $a_0,\dots,a_n$ is all of $\pi_0A$. This determines a morphism 
\begin{equation}
    K[T_1,\dots, T_n] \to  \pi_0A \widehat{\otimes}_K K \langle X_1, \dots, X_n\rangle
\end{equation}
sending $T_i \mapsto a_0X_i-a_i$, which, by the above discussion, lifts to 
\begin{equation}
    K[T_1, \dots, T_n] \to A \widehat{\otimes}^\mathbf{L}_K K \langle X_1, \dots, X_n\rangle. 
\end{equation}
We define 
\begin{equation}
\begin{aligned}
    A^\prime &:=  (A \widehat{\otimes}^\mathbf{L}_K K \langle X_1, \dots, X_n\rangle)/\!/^\mathbf{L}(a_0X_1-a_1,\dots,a_0X_n-a_n)\\ &:= (A \widehat{\otimes}^\mathbf{L}_K K \langle X_1, \dots, X_n\rangle)\widehat{\otimes}^\mathbf{L}_{K[T_1,\dots,T_n]}K.
\end{aligned}
\end{equation}
We claim that:
\begin{lem}\label{lem:derivedrational}
\begin{enumerate}[(i)]
    \item $A \to A^\prime$ is a derived rational localization;
    \item Every derived rational localization of $A$ arises in this way.
\end{enumerate}
\end{lem}
\begin{proof}
(i): It is certainly true that $\pi_0 A \to \pi_0A^\prime$ is a rational localization; what we need to show that that $A \to A^\prime$ is derived strong. This is quite similar to the proof of Lemma \ref{lem:Lclassproperties}(ii). By \cite[Lemma 4.5.79]{DAnG}, $A \to A^\prime$ is a homotopy epimorphism. Therefore, the same argument as in Lemma \ref{lem:Lclassproperties}(ii) shows that $\pi_0(A) \widehat{\otimes}_A^\mathbf{L} A^\prime \to \pi_0(A^\prime)$ is an equivalence. Also, since each $\pi_m(A)$ is finitely-generated over $\pi_0(A)$, it is transversal to $\pi_0(A^\prime)$ over $\pi_0(A)$. Therefore \cite[Proposition 2.3.90]{DAnG} guarantees that $A \to A^\prime$ is derived strong. Finally, we note that this implies that $A^\prime \in \mathsf{dAfndAlg}$ since $\pi_m(A^\prime) \cong \pi_m(A) \widehat{\otimes}_{\pi_0(A)} \pi_0(A^\prime)$ is then finitely-generated as a $\pi_0(A^\prime)$-module. 

(ii): Let $A \to B$ be a derived rational localization. By Lemma \ref{lem:Lclassproperties}(iii), $A \to B$ is a homotopy epimorphism, in particular it is formally \'etale in the sense of \cite[Corollary 2.1.36]{DAnG}.

Now $\pi_0A \to \pi_0B$ is a rational localization, so by the construction in (i) we can construct another rational localization $A \to B^\prime$ which is of the form in (i), and is such that $[\pi_0A \to \pi_0B^\prime] = [\pi_0A \to \pi_0B]$. Now the \'etale lifting property \cite[Corollary 2.1.36]{DAnG}, guarantees that the identification $\pi_0 B^\prime = \pi_0B$ lifts, uniquely up to contractible choice, to an equivalence $B^\prime \simeq B$ under $A$.
\end{proof}
Let us isolate the following facts from the proof of Lemma \ref{lem:derivedrational}:
\begin{scholium}\label{scholium:rationalloc} Let $A \in \mathsf{dAfnd}$.
\begin{enumerate}[(i)]
    \item Let $A \to A^\prime$, $A \to A^{\prime \prime}$ be derived rational localizations. Any isomorphism $\pi_0 A^\prime \cong \pi_0A^{\prime \prime}$ lifts (uniquely up to contractible choice) to an equivalence $A^{\prime} \simeq A^{\prime \prime}$ under $A$, 
    \item Let $\pi_0 A \to B_0$ be a rational localization. Then there exists a (unique up to contractible choice) derived rational localization $A \to B$ reducing to $\pi_0 A \to B_0$ on $\pi_0$. 
\end{enumerate}
\end{scholium}
\begin{defn}
\begin{enumerate}[(i)]
    \item We define $\mathsf{dAfnd}$ to be the opposite $\infty$-category to $\mathsf{dAfndAlg}$. The objects of $\mathsf{dAfnd}$ are denoted by $\operatorname{dSp}(A)$, for $A \in \mathsf{dAfndAlg}$.
    \item We say that a morphism $\operatorname{dSp}(B) \to \operatorname{dSp}(A)$ in $\mathsf{dAfnd}$ is a \emph{rational subdomain of $\operatorname{dSp}(A)$} if $[A \to B] \in \mathscr{L}$. 
\end{enumerate}
\end{defn}
Now we define some Grothendieck topologies as follows. 
\begin{defn}\label{defn:derivedweak}
Let $X = \operatorname{dSp}(A)$ be a derived affinoid rigid space.
\begin{enumerate}[(i)]
    \item The \emph{small weak analytic site} of $X$, is the $\infty$-site with:
    \begin{enumerate}
        \item underlying $\infty$-category given by the full subcategory of $\mathsf{dAfnd}_{/X}$ on rational subdomains $\operatorname{dSp}(B) \to \operatorname{dSp}(A)$.
        \item covering sieves generated by finite families of rational subdomains $\{\operatorname{dSp}(B_i) \to \operatorname{dSp}(B)\}_{i=1}^n$ such that $\{\operatorname{Sp}(\pi_0(B_i)) \to \operatorname{Sp}(\pi_0(B))\}_{i=1}^n$ is an admissible covering of the classical rigid space $\operatorname{Sp}(\pi_0(B))$ in the weak G-topology.
    \end{enumerate}
    \item The \emph{big weak analytic site} on $\mathsf{dAfnd}$, is the $\infty$-site with:
    \begin{enumerate}
        \item underlying $\infty$-category given by $\mathsf{dAfnd}$;
        \item covering sieves generated given by finite families of rational subdomains $\{\operatorname{dSp}(B_i) \to \operatorname{dSp}(B)\}_{i=1}^n$ such that $\{\operatorname{Sp}(\pi_0(B_i)) \to \operatorname{Sp}(\pi_0(B))\}_{i=1}^n$ is an admissible covering of the classical rigid space $\operatorname{Sp}(\pi_0(B))$ in the weak G-topology. 
    \end{enumerate}
\end{enumerate}
\end{defn}
\begin{defn}
Let $X = \operatorname{dSp}(A)$ be a derived affinoid rigid space. We define the $\infty$-category of quasicoherent sheaves as
\begin{equation}
    \operatorname{QCoh}(\operatorname{dSp}(A)) := \operatorname{Mod}_A(D(\mathsf{CBorn}_K)),
\end{equation}
where the latter is the category of modules over the monoid $A \in \mathsf{CAlg}(D(\mathsf{CBorn}_K))$.
\end{defn}
For a morphism $f: \operatorname{dSp}(A) \to \operatorname{dSp}(B)$ in $\mathsf{dAfnd}$ the induced pullback functor 
\begin{equation}
    f^*: \operatorname{QCoh}(\operatorname{dSp}(B)) \to \operatorname{QCoh}(\operatorname{dSp}(A))
\end{equation}
is left adjoint to the restriction of scalars 
\begin{equation}
    f_*: \operatorname{QCoh}(\operatorname{dSp}(A)) \to \operatorname{QCoh}(\operatorname{dSp}(B)).
\end{equation} 
In particular via the pullbacks we obtain a functor
\begin{equation}
    \operatorname{QCoh}: \mathsf{dAfnd}^\mathsf{op} \to \mathsf{CAlg}(\mathsf{Pr}^{L}_{\mathsf{st}}).
\end{equation}
\begin{lem}\label{lem:basechangeaffinoid}
\begin{enumerate}[(i)]
    \item The functor $\operatorname{QCoh}: \mathsf{dAfnd}^\mathsf{op} \to \mathsf{CAlg}(\mathsf{Pr}^{L}_{\mathsf{st}})$ extends to a six-functor formalism 
    \begin{equation}
        \operatorname{QCoh}: \operatorname{Corr}(\mathsf{dAfnd},\mathrm{all})^\otimes \to \mathsf{Pr}^{L,\otimes}_\mathsf{st}
    \end{equation}
    such that for every morphism $f$ in $\mathsf{dAfnd}$ one has $f_* = f_!$. 
    \item For every morphism $f$ in $\mathsf{dAfnd}$ functor $f_*$ is conservative and colimit-preserving.
\end{enumerate}
\end{lem}
\begin{proof}
(i): The proof is identical to Proposition \ref{prop:QcohBasic}. (ii) This is clear because $f_*$ identifies with the forgetful functor at the level of modules.
\end{proof}
\begin{defn}
We define 
\begin{equation}
    \operatorname{Shv}_{\mathrm{weak}}(\mathsf{dAfnd}) := \operatorname{Shv}_{\mathrm{weak}}(\mathsf{dAfnd}, \infty\mathsf{Grpd})
\end{equation}
to be the $\infty$-category of sheaves on the $\infty$-site $\mathsf{dAfnd}$ equipped with the weak topology.
\end{defn}
\begin{lem}\label{lem:descentAfnd}
\begin{enumerate}[(i)]
    \item For every derived affinoid rigid space $X = \operatorname{dSp}(A)$ the functor $\operatorname{dSp}(A) : \mathsf{dAfnd}^\mathsf{op} \to \infty\mathsf{Grpd}$ represented by $\operatorname{dSp}(A)$\footnote{Via the $\infty$-categorical Yoneda embedding.} is a sheaf on $\mathsf{dAfnd}$ in the weak topology. That is to say, this topology is subcanonical. 
    \item The functor $\operatorname{QCoh}: \mathsf{dAfnd}^\mathsf{op} \to \mathsf{CAlg}(\mathsf{Pr}^{L}_{\mathsf{st}})$ is a sheaf on $\mathsf{dAfnd}$ in the weak topology.
\end{enumerate}
\end{lem}
\begin{proof}
We prove (ii) first. Let $\{\operatorname{dSp}(B_i) \to \operatorname{dSp}(B)\}_{i=1}^n$ be a covering in the weak topology. Let $\mathcal{I}$ be the collection of finite subsets $I \subseteq \{1, \dots, n\}$ (we always view such $I$ as being totally ordered). By definition, $\{\operatorname{Sp}(\pi_0(B_i)) \to \operatorname{Sp}(\pi_0(B))\}_{i=1}^n$ is a covering in the classical weak topology. The acyclicity of the ordered \v{C}ech complex together with the Dold-Kan correspondence \cite[Example 1.2.4.10]{HigherAlgebra} implies that 
\begin{equation}\label{eq:pi0Bcover}
\begin{aligned}
    \pi_0(B) &\xrightarrow[]{\sim} \underset{I = (i_1,\dots,i_k) \in \mathcal{I}}{\operatorname{lim}} \pi_0(B_{i_1}) \widehat{\otimes}_{\pi_0(B)} \pi_0(B_{i_2}) \dots \widehat{\otimes}_{\pi_0(B)} \pi_0(B_{i_k})\\
    &\simeq  \underset{I = (i_1,\dots,i_k) \in \mathcal{I}}{\operatorname{lim}} \pi_0(B_{i_1} \widehat{\otimes}^\mathbf{L}_{B} B_{i_2} \dots \widehat{\otimes}_{B}^\mathbf{L} B_{i_k})
\end{aligned}
\end{equation}
is an equivalence in the $\infty$-category $\operatorname{QCoh}(\operatorname{dSp}(\pi_0(A)))$. Using that the class $\mathscr{L}$ of derived rational localizations is stable under base change, c.f. Lemma \ref{lem:Lclassproperties}(ii), we know that each $B \to B_{i_1} \widehat{\otimes}^\mathbf{L}_{B} B_{i_1} \dots \widehat{\otimes}_{B}^\mathbf{L} B_{i_k}$ satisfies the derived strong property \eqref{eq:Lderivedstrong}. Therefore if we apply the functor $\pi_q(B) \widehat{\otimes}^\mathbf{L}_{\pi_0(B)}-$ to both sides of \eqref{eq:pi0Bcover}, using that the limit is finite and the categories are stable, we see that
\begin{equation}
    \pi_q(B)\xrightarrow[]{\sim} \underset{I = (i_1,\dots,i_k) \in \mathcal{I}}{\operatorname{lim}} \pi_q(B_{i_1} \widehat{\otimes}^\mathbf{L}_{B} B_{i_2} \dots \widehat{\otimes}_{B}^\mathbf{L} B_{i_k})
\end{equation}
is an equivalence in $\operatorname{QCoh}(\operatorname{dSp}(\pi_0(B)))$. In particular 
\begin{equation}
\begin{aligned}
    H^p\Big(\underset{(i_1,\dots,i_k) \in \mathcal{I}}{\operatorname{lim}} \pi_q(B_{i_1} \widehat{\otimes}^\mathbf{L}_{B} B_{i_2} \dots \widehat{\otimes}_{B}^\mathbf{L} B_{i_k})\Big) \cong 0 && \text{ for } && p > 0. 
\end{aligned}
\end{equation}
This implies that the spectral sequence 
\begin{equation}
     H^p\Big(\underset{(i_1,\dots,i_k) \in \mathcal{I}}{\operatorname{lim}} \pi_q(B_{i_1} \widehat{\otimes}^\mathbf{L}_{B} \dots \widehat{\otimes}_{B}^\mathbf{L} B_{i_k})\Big) \Rightarrow  \pi_{q-p}\Big(\underset{(i_1,\dots,i_k) \in \mathcal{I}}{\operatorname{lim}} B_{i_1} \widehat{\otimes}^\mathbf{L}_{B}\dots \widehat{\otimes}_{B}^\mathbf{L} B_{i_k}\Big)
\end{equation}
degenerates and gives an isomorphism  
\begin{equation}\label{eq:spectrallimdegenerate}
    \pi_q\Big(\underset{(i_1,\dots,i_k) \in \mathcal{I}}{\operatorname{lim}} B_{i_1} \widehat{\otimes}^\mathbf{L}_{B} B_{i_1} \dots \widehat{\otimes}_{B}^\mathbf{L} B_{i_k}\Big) \cong \operatorname{eq}\Big(\prod_{i=1}^n \pi_q(B_i) \rightrightarrows \prod_{1 \leqslant i < j \leqslant n} \pi_q(B_i \widehat{\otimes}^\mathbf{L}_B B_j)\Big). 
\end{equation}
On the other hand since $B \to B_{i_1} \widehat{\otimes}^\mathbf{L}_{B} B_{i_1} \dots \widehat{\otimes}_{B}^\mathbf{L} B_{i_k}$ satisfies the derived strong property \eqref{eq:Lderivedstrong} and $\pi_q(B)$ is transversal to rational localizations over $\pi_0(B)$ since it is finitely-generated, we see that 
\begin{equation}\label{eq:covertranversal}
\begin{aligned}
    \pi_q(B_{i} \widehat{\otimes}^\mathbf{L}_{B} B_{j}) &\simeq \pi_q(B)\widehat{\otimes}^\mathbf{L}_{\pi_0(B)}  \pi_0(B_{i} \widehat{\otimes}_{B}^\mathbf{L} B_{j}) \\ &\simeq \pi_q(B)\widehat{\otimes}_{\pi_0(B)}  \pi_0(B_{i}) \widehat{\otimes}_{\pi_0(B)}  \pi_0(B_{j}).
\end{aligned}
\end{equation}
Using \eqref{eq:covertranversal} in the right-side of \eqref{eq:spectrallimdegenerate} we obtain 
\begin{equation}
\begin{aligned}
    &\operatorname{eq}\Big(\prod_{i=1}^n \pi_q(B_i) \rightrightarrows \prod_{1 \leqslant i < j \leqslant n} \pi_q(B_i \widehat{\otimes}^\mathbf{L}_B B_j)\Big) \\
    &\cong \operatorname{eq}\Big(\prod_{i=1}^n \pi_q(B)\widehat{\otimes}_{\pi_0(B)} \pi_0(B_i) \rightrightarrows \prod_{1 \leqslant i < j \leqslant n} \pi_q(B) \widehat{\otimes}_{\pi_0(B)} \pi_0(B_i) \widehat{\otimes}_{\pi_0(B)} \pi_0(B_j)\Big) \\
    &\cong \pi_q(B),
\end{aligned}
\end{equation}
where in the last line we used the classical theorem of descent for coherent sheaves on affinoid rigid spaces. Putting this all together we deduce that for each $q \geqslant 0$, the morphism 
\begin{equation}
    \pi_q(B) \to \pi_q\Big(\underset{(i_1,\dots,i_k) \in \mathcal{I}}{\operatorname{lim}} B_{i_1} \widehat{\otimes}^\mathbf{L}_B \dots \widehat{\otimes}^\mathbf{L}_B B_{i_k}\Big)
\end{equation}
is an isomorphism and therefore the natural morphism
\begin{equation}\label{eq:Bfinite}
    B \xrightarrow[]{\sim} \underset{(i_1,\dots,i_k) \in \mathcal{I}}{\operatorname{lim}} B_{i_1} \widehat{\otimes}^\mathbf{L}_B \dots \widehat{\otimes}^\mathbf{L}_B B_{i_k},
\end{equation}
is an equivalence. This shows that the canonical morphism $B \to \prod_{i=1}^n B_i$ is descendable in the sense of Mathew \cite[\S 3.3]{MathewGalois}. The result of \emph{loc. cit.} then implies that if one sets $Y := \coprod_{i=1}^n \operatorname{dSp}(B_i) \to \operatorname{dSp}(B) =: X$ then the natural morphism 
\begin{equation}
    \operatorname{QCoh}(X) \xrightarrow[]{\sim} \underset{[m] \in \Delta}{\operatorname{lim}} \operatorname{QCoh}(Y^{m+1/X})
\end{equation}
is an equivalence, proving (ii). Looking at the unit object we see that 
\begin{equation}\label{eq:Bprodiso}
    B \xrightarrow[]{\sim} \underset{[m] \in \Delta}{\operatorname{lim}} \big(\prod_{i=1}^n B_i\big)^{\widehat{\otimes}^\mathbf{L}_B(m+1)}
\end{equation}
is an equivalence. If $A \in \mathsf{dAfndAlg}$ is a derived affinoid algebra, we can apply the functor $\operatorname{Map}_{\mathsf{dAfndAlg}}(A,-)$ to both sides of \eqref{eq:Bprodiso} to deduce (i). 
\end{proof}
\begin{rmk}[Notational remark] In derived geometry, it appears to be conventional to mix homological and cohomological indexing conventions, that is $\pi_i$ stands for the \emph{homology} functors, $H^i$ stands for the \emph{cohomology} functors, and they are related by $\pi_i = H^{-i}$ for $i \in \mathbf{Z}$. Usually, one uses $\pi_i$ for the homotopy groups of commutative algebra objects, and $H^i$ for the cohomology groups of ``linear" objects like cochain complexes (though this is not a hard and fast rule). The historical reason for this is that ``derived rings" are often modelled on simplicial commutative rings, which have ``homotopy groups", whereas algebraic geometers prefer cohomological indexing convention for derived categories of quasi-coherent sheaves.  
\end{rmk}
\begin{rmk}\label{rmk:colimitcovering}
Let $X = \operatorname{dSp}(A)$ be a derived affinoid rigid space and let $\{U_i \to X\}_{i=1}^n$ be a covering in the weak topology. Let $\mathcal{I}$ be the family of finite nonempty subsets $I \subseteq \{1,\dots,n\}$ ordered by inclusion. For $I = (i_0,\dots,i_k) \in \mathcal{I}$, we set $U_I := \bigcap_{j \in I} U_j := U_{i_0} \times_X \dots \times_X U_{i_k}$. In the course of proving Lemma \ref{lem:descentAfnd} we obtained two facts.
\begin{enumerate}[(i)]
    \item There is a canonical equivalence 
    \begin{equation}
        \underset{I \in \mathcal{I}}{\operatorname{colim}} 
 U_I \xrightarrow[]{\sim} X
    \end{equation} in $\operatorname{Shv}_{\mathrm{weak}}(\mathsf{dAfnd})$. This is a consequence of \eqref{eq:Bfinite}.
    \item Set $Y := \coprod_{i=1}^n U_i$. Then there is a canonical equivalence 
    \begin{equation}
        \underset{[m] \in \Delta^\mathsf{op}}{\operatorname{colim}}Y^{m+1/X} \xrightarrow[]{\sim} X 
    \end{equation}
    in $\operatorname{Shv}_{\mathrm{weak}}(\mathsf{dAfnd})$. This is a consequence of \eqref{eq:Bprodiso}.
\end{enumerate}
\end{rmk}

\subsection{The gluing procedure}\label{subsec:gluing}
Now in a similar manner to \cite[\S2.4]{mann_p-adic_2022} we will glue derived affinoid rigid spaces to obtain our desired category of derived rigid spaces. Since the weak topology is subcanonical we can view $\mathsf{dAfnd} \subseteq \operatorname{Shv}_{\mathrm{weak}}(\mathsf{dAfnd})$ as a full subcategory. In fact we will slightly abuse the terminology in order to make the following definition.
\begin{defn}\label{defn:analyticsubspace}
\begin{enumerate}[(i)]
    \item A \emph{affinoid derived rigid space} is an object of $\operatorname{Shv}_{\mathrm{weak}}(\mathsf{dAfnd})$ which is isomorphic to $\operatorname{dSp}(A)$ for some $A \in \mathsf{dAfndAlg}$. 
    \item Let $X = \operatorname{dSp}(A)$ be an affinoid derived rigid space. An analytic subspace $U \hookrightarrow X$ is a subsheaf $U$ of $X$ such that:
    \begin{itemize}
        \item[$\star$] There exists a small collection $\{\operatorname{dSp}(A_i)\}_{i \in \mathcal{I}}$ of derived affinoid spaces and an effective epimorphism $\coprod_{i \in \mathcal{I}} \operatorname{dSp}(A_i) \to U$ in $\operatorname{Shv}_{\mathrm{weak}}(\mathsf{dAfnd})$ such that each $\operatorname{dSp}(A_i) \to \operatorname{dSp}(A)$ is a rational subdomain.
    \end{itemize}
    \item Let $X \in \operatorname{Shv}_\mathrm{weak}(\mathsf{dAfnd})$. An analytic subspace $Y \hookrightarrow X$ is a subsheaf such that for every affinoid derived rigid space $\operatorname{dSp}(A) \to X$ mapping to $X$, the pullback $Y \times_X \operatorname{dSp}(A) \hookrightarrow \operatorname{dSp}(A)$ is an analytic subspace in the sense of (ii). The morphism $Y \to X$ is then called an \emph{open immersion.}
    \item A \emph{derived rigid space} is an object $X \in \operatorname{Shv}_{\mathrm{weak}}(\mathsf{dAfnd})$ such that there exists a small collection $\{\operatorname{dSp}(A_i)\}_{i \in \mathcal{I}}$ of derived affinoid subspaces such that $\coprod_{i \in \mathcal{I}} \operatorname{dSp}(A_i) \to X$ is an effective epimorphism. We denote the full subcategory of $\operatorname{Shv}_{\mathrm{weak}}(\mathsf{dAfnd})$ on derived rigid spaces by $\mathsf{dRig}$. 
\end{enumerate}
\end{defn}
We remark that the category $\mathsf{dRig}$ admits all fiber products. The terminal object is $\operatorname{dSp}(K)$. We can formally define some Grothendieck topologies as follows.
\begin{defn}\label{defn:derivedstrongtopology}
Let $X  \in \mathsf{dRig}$ be a derived rigid space.
\begin{enumerate}[(i)]
    \item The \emph{small strong analytic site} of $X$, is the $\infty$-site with:
    \begin{enumerate}[(a)]
        \item underlying $\infty$-category given by the full subcategory of $\mathsf{dRig}_{/X}$ on analytic subspaces $Y \hookrightarrow X$;
        \item covering sieves generated by small families of analytic subspaces $\{Y_i \hookrightarrow X\}_{i \in \mathcal{I}}$ such that $\coprod_{i \in \mathcal{I}} Y_i \to X$ is an effective epimorphism.
    \end{enumerate}
    \item The \emph{big strong analytic site} on $\mathsf{dRig}$, is the $\infty$-site with: 
    \begin{enumerate}
        \item underlying $\infty$-category given by $\mathsf{dRig}$;
        \item  covering sieves generated by small families of analytic subspaces $\{Y_i \hookrightarrow X\}_{i \in \mathcal{I}}$ such that $\coprod_{i \in \mathcal{I}} Y_i \to X$ is an effective epimorphism.
    \end{enumerate}
\end{enumerate}
\end{defn}
\subsection{The underlying topological space}\label{subsec:topological}
In order to make this theory more workable we will associate to each $X \in \mathsf{dRig}$ a topological space. We always assume the axiom of choice\footnote{This is necessary in order to know that every filter is contained in a maximal filter.}.
\begin{defn}
An object $X \in \mathsf{dRig}$ is called \emph{classical} if it admits a covering $\coprod_{i \in \mathcal{I}} \operatorname{dSp}(A_i) \to X$ by affinoid subspaces such that all the $A_i$ are discrete. 
\end{defn}
\begin{lem}
The inclusion of classical rigid spaces admits a right adjoint $X \mapsto X_0$ which extends $\operatorname{dSp}(A) \mapsto \operatorname{dSp}(\pi_0(A))$. $Y \hookrightarrow X$ is an analytic subspace if and only if $Y_0 \hookrightarrow X_0$ is an analytic subspace.
\end{lem}
\begin{proof}
Let us temporarily denote $\mathscr{X} := \operatorname{Shv}_\mathrm{weak}(\mathsf{dAfnd})$. The functor $(-)_0: \mathsf{dAfnd} \to \mathsf{dAfnd}: \operatorname{dSp}(A) \mapsto \operatorname{dSp}(\pi_0A)$ preserves fiber products. For every finite covering family $\{\operatorname{dSp}(B_i) \to \operatorname{dSp}(B)\}_{i=1}^n$ in the weak topology, by Definition \ref{defn:derivedweak} one has that $\{\operatorname{dSp}(\pi_0 B_i) \to \operatorname{dSp}(\pi_0B)\}_{i=1}^n$ is a covering family in the weak topology. Since the topology is subcanonical by Lemma \ref{lem:descentAfnd}, this implies that $\coprod_{i=1}^n \operatorname{dSp}(\pi_0 B_i) \to \operatorname{dSp}(\pi_0B)$ is an effective epimorphism in $\mathscr{X}$. Therefore one may use the local Yoneda embedding \cite[Proposition 6.2.30]{HigherToposTheory} to extend $(-)_0$ to a colimit-preserving left-exact functor $(-)_0: \mathscr{X} \to \mathscr{X}$. In particular $(-)_0$ preserves effective epimorphisms and subobjects, and therefore sends objects of $\mathsf{dRig} \subseteq \mathscr{X}$ to classical rigid spaces (in the sense of the above definition), and also restricts to the identity on classical rigid spaces. Let $Y$ be a classical rigid space and let $X \in \mathsf{dRig}$. Then $(-)_0$ induces $\operatorname{Map}(Y,X) \to \operatorname{Map}(Y,X_0)$. We claim that this is an equivalence. Fix a covering $\mathcal{U} = \{U_i\}$ of $X$ by derived affinoids. Given a covering $\mathcal{V} = \{V_j\}$ of $Y$ by classical affinoids one can consider the full subspace $\operatorname{Map}^\mathcal{VU}(Y,X) \subseteq \operatorname{Map}(Y,X)$ spanned by morphisms $f$ such that $\mathcal{V}$ refines the pullback of $\mathcal{U}$ along $f$. One then has $\operatorname{colim}_{\mathcal{V}}\operatorname{Map}^\mathcal{VU}(Y,X) \simeq  \operatorname{Map}(Y,X)$, where the system of coverings $\mathcal{V}$ is ordered by refinement. Hence one can reduce the claim to the case when $X$ is a derived affinoid and then further to when $Y$ is a classical affinoid. The claim then follows since $\pi_0$ is a left adjoint at the level of algebra. 
\end{proof}

Now in a similar manner to \cite[Lecture XIV]{LecturesOnAnalyticGeometry} we will associate a topological space to every $X \in \mathsf{dRig}$ as follows. Viewed as an object of $\operatorname{Shv}_{\mathrm{weak}}(\mathsf{dAfnd})$, the collection $\operatorname{Sub}(X)$ of subobjects of $X$ forms a locale, as can be seen by applying the results of \cite[\S 6.4.5]{HigherToposTheory} to the $\infty$-topos $\operatorname{Shv}_{\mathrm{weak}}(\mathsf{dAfnd})_{/X}$. We will write the operations of join and meet in this locale as $\cup$ and $\cap$, respectively. Now we have the following Lemma.
\begin{lem}
With respect to the same operations of join and meet, the poset of analytic subspaces of $X$, in the sense of Definition \ref{defn:derivedweak}, also forms a locale.   
\end{lem} 
\begin{proof}
It is clear that the meet of two analytic subspaces of $X$ is again an analytic subspace: one has $U \cap V= U \times_X V$. The join of two objects $U, V  \in \operatorname{Sub}(X)$ is given explicitly as 
\begin{equation}\label{eq:finitejoin}
  U \cup V = U \underset{U \cap V}{\coprod}V, 
\end{equation}
from which one can of course deduce the formula for finite joins. More generally, if $\{U_i\}_{i \in \mathcal{I}}$ is a small family of analytic subspaces of $X$, then one has explicitly 
\begin{equation}\label{eq:infinitejoin}
    \bigcup_{i \in \mathcal{I}} U_i = \underset{I}{\operatorname{colim}} \bigcup_{i \in I} U_i, 
\end{equation}
where the colimit runs over the finite subsets $I \subseteq \mathcal{I}$. Now we recall that colimits in any $\infty$-topos are universal (this is one of the Giraud--Rezk--Lurie axioms \cite[Theorem 6.1.0.6]{HigherToposTheory}), so that if $Y = \operatorname{dSp}(A) \to X$ is a derived affinoid mapping to $X$, by \eqref{eq:finitejoin} and \eqref{eq:infinitejoin} one has
\begin{equation}
     \Big(\bigcup_{i \in \mathcal{I}} U_i\Big) \times_X Y = \bigcup_{i \in \mathcal{I}} U_i^\prime,
\end{equation}
where $U_i^\prime := U_i \times_X Y$. By assumption each $U_i^\prime$ is an analytic subspace of the affinoid space $Y$, in the sense of Definition \ref{defn:analyticsubspace}(ii). By looking at the formulas \eqref{eq:finitejoin} and \eqref{eq:infinitejoin} with $U_i^\prime$ in place of $U_i$, it is clear that $\bigcup_{i \in \mathcal{I}} U_i^\prime$ is an analytic subspace of the affinoid space $Y$, in the sense of Definition \ref{defn:analyticsubspace}(ii). That is, if $\{V_{ij}^\prime\}_{j \in \mathcal{J}(i)}$ is a small collection of rational opens of $Y$ such that $\coprod_{j \in \mathcal{J}(i)}V_{ij}^\prime \to U_i^\prime$ is an effective epimorphism, then 
\begin{equation}
    \coprod_{i \in \mathcal{I}}\coprod_{j \in \mathcal{J}(i)}V_{ij}^\prime \to \bigcup_{i \in \mathcal{I}} U_i^\prime 
\end{equation}is an effective epimorphism. This completes the proof, c.f. Definition \ref{defn:analyticsubspace}(iii).   
\end{proof}
We denote the locale of analytic subspaces of $X$ by $\operatorname{An}(X)$. It is an immediate consequence of Definition \ref{defn:derivedstrongtopology}(i) that for any $\infty$-category $\mathscr{D}$ admitting small limits there is a canonical equivalence of $\infty$-categories
\begin{equation}
    \operatorname{Shv}_{\mathrm{strong}}(X,\mathscr{D}) \simeq \operatorname{Shv}(\operatorname{An}(X),\mathscr{D}),
\end{equation}
where the latter is the category of sheaves on the locale. A basis for the locale $\operatorname{An}(X)$ is given by the affinoid analytic subspaces of $X$, which are quasi-compact. In particular, the locale $\operatorname{An}(X)$ is locally compact. Therefore by Hoffman-Lawson duality \cite[VII,\S4]{JohnstoneStone} the locale is spatial; that is, if one sets 
\begin{equation}
    |X| := \operatorname{pt}(\operatorname{An}(X)) 
\end{equation}
to be the topological space of \emph{points}, i.e., the completely prime filters on the locale, then there is a canonical isomorphism of locales $\Omega(|X|) \cong \operatorname{An}(X)$. Here $\Omega$ is the functor which sends a topological space to its locale of open subsets. In particular, for any $\infty$-category $\mathscr{D}$ admitting small limits one obtains a canonical equivalence of $\infty$-categories
\begin{equation}
    \operatorname{Shv}(\operatorname{An}(X),\mathscr{D}) \simeq \operatorname{Shv}(|X|,\mathscr{D}).
\end{equation}
It follows by functoriality of the above constructions, that $|\cdot|$ determines a covariant functor $|\cdot|: \mathsf{dRig} \to \mathsf{Top}$, where the latter is (the nerve of) the ordinary category of topological spaces. As it arises from a spatial locale, $|\cdot|$ factors through sober topological spaces. We will prove that this topological space obeys one of the principles of derived geometry, which says that ``$X_0$ contains all the geometry". To prove this we will essentially follow the same recipe as \cite[\S 2.9]{mann_p-adic_2022}. 
\begin{lem}\label{lem:injectiverational}
\begin{enumerate}
    \item[(i)] Let $\{\operatorname{dSp}(A_i) \to \operatorname{dSp}(A)\}_{i=1}^n$ be a finite collection of derived rational subdomains of $\operatorname{dSp}(A)$. Then $\coprod_{i=1}^n|\operatorname{dSp}(A_i)| \to |\operatorname{dSp}(A)|$ is a surjective morphism of topological spaces if and only if $\coprod_{i=1}^n|\operatorname{dSp}(\pi_0A_i)| \to |\operatorname{dSp}(\pi_0A)|$ is a surjective morphism of topological spaces.
    \item[(ii)] Let $X \in \mathsf{dRig}$ and let $U, U^\prime \subseteq X$ be analytic subspaces. Then one has $|U^\prime| \subseteq |U|$ if and only if $|(U^\prime)_0| \subseteq |U_0|$. In particular $|U| = |U^\prime|$ if and only if $|(U^\prime)_0| = |U_0|$.
\end{enumerate}
\end{lem}
\begin{proof}
(i): This is an immediate consequence of Definition \ref{defn:derivedweak}.
(ii): This follows from (i). 
\end{proof}
\begin{lem}\label{lem:surjectiverational}
    Let $X = \operatorname{dSp}(A)$ be a derived affinoid space and let $V \subseteq X_0$ be a rational subdomain. Then there exists a (unique up to contractible choice) rational subdomain $U \subseteq X$ such that $U_0 = V$ as subobjects of $X_0$. 
\end{lem}
\begin{proof}
This is Scholium \ref{scholium:rationalloc}(ii). 
\end{proof}
\begin{cor}
Let $X \in \mathsf{dRig}$ and let $V \subseteq |X_0|$ be an open subset. Then there exists an analytic subspace $U \subseteq X$ such that $|U_0| = V$. 
\end{cor}
\begin{proof}
The problem is local on $X$, so we may assume that $X$ is a derived affinoid. Then the claim follows from Lemma \ref{lem:surjectiverational}, since the rational subdomains form a basis for the topology. 
\end{proof}
\begin{thm}\label{thm:topologicalinvariance}
Let $X \in \mathsf{dRig}$. The functor $(-)_0$ induces an isomorphism of topological spaces $|X_0| \xrightarrow[]{\sim} |X|$. 
\end{thm}
\begin{proof}
This follows from Lemma \ref{lem:injectiverational} and Lemma \ref{lem:surjectiverational}, using that $|\cdot|$ takes values in sober topological spaces (so that $|X|, |X_0|$ are determined by their lattice of opens), and using also that the open subsets biject with analytic subspaces. 
\end{proof}
\begin{rmk}\label{rmk:topolopgicalspaceremark}
\begin{enumerate}[(i)]
    \item Let $X = \operatorname{dSp}(A)$ be an affinoid derived rigid space. The lattice $\operatorname{Special}(X)$ of finite unions of rational subdomains forms a basis for the basis for the locale $\operatorname{An}(X)$, and so completely prime filters on $\operatorname{An}(X)$ biject with prime filters on the lattice $\operatorname{Special}(X)$. We note that $\operatorname{Special}(X)$ is a bounded distributive lattice and hence by the Stone representation theorem for distributive lattices, c.f. \cite[II, \S3]{JohnstoneStone}, we see that $|X|$ is a spectral topological space.\footnote{For arbitrary, non-affinoid $X$ this shows that $|X|$ is locally spectral.} 
    We may also make the following definition. Let $|X|_{Ber}$ be the collection of maximal filters on $\operatorname{Special}(X)$. This is naturally a compact topological space \cite[II, \S3]{JohnstoneStone} equipped with an injective morphism $|X|_{Ber} \to |X|$. This construction defines a functor $|\cdot|_{Ber}: \mathsf{dAfnd} \to \mathsf{Top}$, equipped with a natural transformation $|\cdot|_{Ber} \to |\cdot|$.
    \item In the situation of (ii), \cite[II, Proposition 3.7]{JohnstoneStone} asserts that the following are equivalent for a derived affinoid space $X$:
    \begin{enumerate}
        \item Whenever $U_1, U_2 \in \operatorname{Special}(X)$ with $U_1 \cap U_2 = \emptyset$, we can find $V_1, V_2 \in \operatorname{Special}(X)$ with $V_1 \cup V_2 = X$, $V_1 \cap U_2 = \emptyset$ and $U_1 \cap V_2 = \emptyset$,  c.f.\footnote{Note that our situation is dual to Johnstone's, because we use filters rather than \emph{ideals} of lattices.}  \cite[II, $\dagger$3.6]{JohnstoneStone}. 
        \item $|X|_{Ber}$ is Hausdorff. 
        \item Every prime filter\footnote{Here we have identified points of $|X|$ with prime filters on $\operatorname{Special}(X)$.} $x \in |X|$ is contained in a unique maximal filter $r(x) \in |X|_{Ber}$.
        \item The canonical morphism $|X|_{Ber} \to |X|$ is a split monomorphism.
    \end{enumerate}
    In (d) the retraction can be chosen as the canonical map $x \mapsto r(x)$ coming from (c). If any of the above equivalent conditions are satisfied then (it is trivial to see that) the topology on $|X|_{Ber}$ is the quotient topology induced by the surjection $r: |X| \twoheadrightarrow |X|_{Ber}$.
    \item In the situation of (ii), suppose further that $A$ is discrete, so that $X$ is a classical affinoid. Then the results of \cite{SchneiderPointsandtopologies} and \cite[Corollary 4.5]{huber_continuous_1993} imply that there are canonical isomorphisms 
    \begin{equation}
    \begin{aligned}
        |\operatorname{Spa}(A, A^\circ)| \cong |X| && \text{ and } && \mathcal{M}(X) \cong |X|_{Ber}
    \end{aligned}
    \end{equation}
    of functors on the category of classical affinoid algebras, where $|\operatorname{Spa}(A, A^\circ)|$ is the Huber spectrum and $\mathcal{M}(X)$ is the Berkovich spectrum. These isomorphisms fit into the commutative diagram
       \begin{equation}
        \begin{tikzcd}
	{|\operatorname{Spa}(A,A^\circ)|} & {|X|} \\
	{\mathcal{M}(A)} & {|X|_{Ber}}
	\arrow["\cong", no head, from=1-1, to=1-2]
	\arrow[to=1-1, from=2-1]
	\arrow[to=1-2, from=2-2]
	\arrow["\cong", no head, from=2-1, to=2-2]
\end{tikzcd}
    \end{equation}
    of functors on the category of classical affinoid algebras. By \cite[Lemma 6]{SchneiderPointsandtopologies}, the morphism $\mathcal{M}(X) \to |\operatorname{Spa}(A, A^\circ)|$ is a split monomorphism; therefore the equivalent conditions in (ii) above are satisfied. 
    \item Now if $X \in \mathsf{dAfnd}$ is a derived affinoid space then it follows from Theorem \ref{thm:topologicalinvariance} and (iii) that there is a commutative diagram
    \begin{equation}
\begin{tikzcd}
	{|\operatorname{Spa}(\pi_0A,(\pi_0A)^\circ)|} & {|X|} \\
	{\mathcal{M}(\pi_0A)} & {|X|_{Ber}}
	\arrow["\cong", no head, from=1-1, to=1-2]
	\arrow[to=1-1, from=2-1]
	\arrow[to=1-2, from=2-2]
	\arrow["\cong", no head, from=2-1, to=2-2]
\end{tikzcd}
    \end{equation}
    of functors on $\mathsf{dAfnd}$. In particular, we see that the equivalent conditions in (ii) are satisfied for arbitrary derived affinoid spaces $X$.
    \item As a corollary of \cite[Proposition 7.3.6.10]{HigherToposTheory} and Theorem \ref{thm:topologicalinvariance} one sees that the canonical morphism of $\infty$-topoi  $r_*:\operatorname{Shv}(|X|) \to \operatorname{Shv}(|X|_{Ber})$ induced by $r: |X|_{Ber} \to |X|$, is cell-like. This can be regarded as a higher-categorical generalization of the characterisation of sheaves of sets on the Berkovich space as ``overconvergent sheaves". 
    \item The definition of $|X|_{Ber}$ given in (i) above generalises to qcqs derived rigid spaces. It is not so clear how to generalize this to arbitrary derived rigid spaces. The problem is that it is not necessarily true that every maximal filter on the locale $\operatorname{An}(X)$ is completely prime\footnote{In \cite[\S5]{SchneiderPointsandtopologies}, the authors do not seem to consider this problem.} but for qcqs derived rigid spaces one can circumvent this problem by using that the lattice has a canonical finitary basis formed by quasi-compact subspaces.
\end{enumerate}
\end{rmk}

\subsection{Quasi-compact and quasi-separated morphisms}\label{subsubsec:qcqs}
\begin{defn}\label{defn:quasicompactquasiseparated}
\begin{enumerate}[(i)]
    \item An object $X \in \mathsf{dRig}$ is called \emph{quasi-compact} if every cover of $X$ in the strong topology admits a finite subcover.
    \item A morphism $f : X \to Y$ in $\mathsf{dRig}$ is called \emph{quasi-compact} if, for every quasi compact $Z$ with a morphism $Z \to Y$, the pullback $ X \times_Y Z$ is quasi-compact. 
    \item A morphism $f: X \to Y$ in $\mathsf{dRig}$ is called \emph{quasi-separated} if its diagonal $\Delta_f: X \to X \times_Y X$ is quasi-compact.
    \item A object $X \in \mathsf{dRig}$ is called \emph{quasi-separated} if the structure morphism $X \to \operatorname{dSp}(K)$ is quasi-separated. 
    \item We abbreviate \emph{quasi-compact and quasi-separated} to \emph{qcqs}.
\end{enumerate}
\end{defn}
\begin{lem}\label{lem:qcqsproperties0}
\begin{enumerate}[(i)]
    \item The classes of quasi-compact and quasi-separated morphisms are stable under base change and composition.
    \item Let $f$ and $g$ be composable morphisms in $\mathsf{dRig}$. If $fg$ is quasi-compact and $f$ is quasi-separated then $g$ is quasi-compact.
    \end{enumerate}
\end{lem}
\begin{proof}
(i): Omitted. 
(ii): Let us write $f: Y \to X$ and $g: Z \to Y$. We can factor $g$ as $Z \to Z \times_X Y \to Y$. Here the first map is the base-change of $\Delta_f$, and the second map is the base-change of $fg$. Therefore, since quasi-compact morphisms are stable under base-change and composition, we see that $g$ is quasi-compact.  
\end{proof}
\begin{rmk}
\begin{enumerate}[(i)]
    \item The notions defined in Definition \ref{defn:quasicompactquasiseparated} make sense in any $\infty$-site. In particular they make sense in $\operatorname{Shv}_{\mathrm{weak}}(\mathsf{dAfnd})$ equipped with the effective epimorphism topology.
    \item The weak topology on $\mathsf{dAfnd}$ is a \emph{finitary} Grothendieck topology \cite[Definition 3.17]{LurieSpectralSchemes}. Therefore, by \cite[Proposition 3.19]{LurieSpectralSchemes} we obtain the following. Every object of $\mathsf{dAfnd}$ is quasi-compact and quasi-separated, when viewed as an object of $\operatorname{Shv}_{\mathrm{weak}}(\mathsf{dAfnd})$. In particular every affinoid derived rigid space is quasi-compact object in $\mathsf{dRig}$. Moreover, the $\infty$-topos $\operatorname{Shv}_{\mathrm{weak}}(\mathsf{dAfnd})$ is \emph{locally coherent} \cite[Definition 3.12]{LurieSpectralSchemes}.

\end{enumerate}
\end{rmk}
\begin{lem}\label{lem:qcproperties}
Let $f: X \to Y$ be a morphism in $\mathsf{dRig}$. The following are equivalent.
\begin{enumerate}[(i)]
    \item $f$ is quasi-compact. 
    \item For every affinoid subspace $U \subseteq X$, $Y \times_X U$ is quasi-compact. 
    \item There exists a covering of $\mathcal{U} = \{U_i\}_{i \in \mathcal{I}}$ of $Y$ by affinoid subspaces such that for each $i \in \mathcal{I}$, $X \times_Y U_i$ is quasi-compact.
\end{enumerate}
\end{lem}
\begin{proof}
(i) $\implies$ (ii) $\implies$ (iii) is obvious. We prove (iii) $\implies$ (i). Let $g: Z \to Y$ be a morphism from a quasi-compact object $Z \in \mathsf{dRig}$. Choose a finite covering $\mathcal{V} = \{V_j\}_{j \in \mathcal{J}}$ of $Z$ by affinoid subspaces such that $\mathcal{V}$ refines the pullback of $\mathcal{U}$ along $g$. Then $\{X \times_Y V_j\}_{j \in \mathcal{J}}$ is a finite covering of $X \times_Y Z$ by quasi-compact subspaces. Therefore $X \times_Y Z$ is quasi-compact.  
\end{proof}
\begin{lem}\label{lem:qsproperties1}
Let $f: X \to Y$ be a morphism in $\mathsf{dRig}$. The following are equivalent.
\begin{enumerate}[(i)]
    \item $f$ is quasi-separated. 
    \item For every pair of affinoid subspaces $U, V \subseteq X$ mapping into a commmon affinoid subspace of $Y$, $U \times_X V$ is a admits a finite cover by affinoid subspaces of $X$.
    \item There exists a cover $\{U_i\}_{i \in \mathcal{I}}$ of $Y$ by affinoid subspaces and for each $i \in \mathcal{I}$ a cover $\{ V_j\}_{j \in \mathcal{J}(i)}$ of $X \times_Y U_i$ by affinoid subspaces such that for every $j, j^\prime \in \mathcal{J}(i)$, $V_j  \times_X V_{j^\prime}$ admits a finite cover by affinoid subspaces of $X$. 
\end{enumerate}
\end{lem}
\begin{proof}
(i) $\implies$ (ii) $\implies$ (iii) is obvious. To prove (iii) we note that $\bigcup_{i \in \mathcal{I}} \bigcup_{j, j^\prime \in \mathcal{J}(i)} V_j \times_{U_i} V_{j^\prime}$ is a covering of $X \times_Y X$ by affinoid subspaces and the restriction of $\Delta_f$ is $V_j \times_X V_{j^\prime} \to V_j \times_{U_i} V_{j^\prime}$, which is quasi-compact, since affinoids are quasi-separated. Hence, we may conclude by Lemma \ref{lem:qcproperties}(iii).
\end{proof}
\begin{lem}\label{lem:qsproperties2}
Let $f$ and $g$ be composable morphisms of $\mathsf{dRig}$. If $fg$ is quasi-separated then so is $g$. 
\end{lem}
\begin{proof}
Let us write $g: Z \to Y$ and $f: Y \to X$. Choose a covering $\mathcal{W}$ of $X$ by affinoid subspaces and a covering $\mathcal{U}$ of $Y$ by affinoid subspaces which refines the pullback of $\mathcal{W}$ along $f$. For each $U \in \mathcal{U}$ let $V, V^\prime$ be affinoid subspaces of $Z$ mapping into $U$, and say that $U$ maps into the affinoid subspace $W \in \mathcal{W}$. By quasi-separatedness of $fg$, the morphism $V \times_Z V^\prime \to V \times_X V^\prime = V \times_W V^\prime$ is quasi-compact, and the latter is affinoid. Therefore $V \times_Z V^\prime$ is a finite union of affinoid subspaces of $Z$. Since such $V, V^\prime$ cover $Z$, we may appeal to Lemma \ref{lem:qsproperties1}(iii) to conclude that $g$ is quasi-separated.
\end{proof}
\begin{cor}\label{cor:qcqsdiagonal}
Let $f, g$ be composable morphisms in $\mathsf{dRig}$. If $f$ and $fg$ are both qcqs then so is $g$.
\end{cor}
The following facts were used in \S\ref{subsec:sixf3}, but for the purposes of this section it will be helpful to be more verbose. Recall that we have a diagram of $\infty$-categories 
\begin{equation}
\begin{tikzcd}
	{\mathsf{dRig}} & {\operatorname{Shv}_{\mathrm{weak}}(\mathsf{dAfnd})} & {\operatorname{Psh}(\mathsf{dAfnd})} & {\mathsf{dAfnd}}
	\arrow[hook, from=1-1, to=1-2]
	\arrow[shift left, hook, from=1-2, to=1-3]
	\arrow["L", shift left, from=1-3, to=1-2]
	\arrow["j"', hook', from=1-4, to=1-3]
\end{tikzcd}
\end{equation}
where $L$ stands for sheafification \cite{HigherToposTheory} and $j$ is the $\infty$-categorical Yoneda embedding. In particular for any $\infty$-category $\mathscr{D}$ admitting small limits we obtain functors 
\begin{equation}\label{eq:Shvgeneral}
    \operatorname{Shv}(\operatorname{Shv}_\mathrm{weak}(\mathsf{dAfnd}), \mathscr{D}) \hookrightarrow \operatorname{Fun}_0(\operatorname{Psh}(\mathsf{dAfnd})^\mathsf{op},\mathscr{D}) \xrightarrow[]{\sim} \operatorname{Fun}(\mathsf{dAfnd}^\mathsf{op}, \mathscr{D})
\end{equation}
where $\operatorname{Fun}_0$ denotes the limit-preserving functors; the first morphism is induced by precomposition with $L$ and the second is induced by precomposition with $j$. Here we have equipped $\operatorname{Shv}_\mathrm{weak}(\mathsf{dAfnd})$ with the effective epimorphism topology. The image of this composite is $\operatorname{Shv}_{\mathrm{weak}}(\mathsf{dAfnd}, \mathscr{D})$, c.f. \cite[Proposition 1.3.1.7]{SpectralAlgebraicGeometry}.

For our purposes this means the following. In Lemma \ref{lem:descentAfnd} we proved that the functor $\operatorname{QCoh}$ belongs to $\operatorname{Shv}_{\mathrm{weak}}(\mathsf{dAfnd}, \mathsf{CAlg}(\mathsf{Pr}_{\mathsf{st}}^L))$. We also proved that the weak topology on $\mathsf{dAfnd}$ is subcanonical. Hence, by the above procedure we obtain a functor
\begin{equation}\label{eq:QCohshv}
    \operatorname{QCoh}: \operatorname{Shv}_{\mathrm{weak}}(\mathsf{dAfnd})^\mathsf{op} \to \mathsf{CAlg}(\mathsf{Pr}_{\mathsf{st}}^L)  
\end{equation}
which is a sheaf in the effective epimorphism topology and which extends $\operatorname{QCoh}$ on $\mathsf{dAfnd}$. 
\begin{rmk}\label{rmk:QCohKanandColim}
\begin{enumerate}[(i)]
    \item Chasing the definitions, and using that the weak topology is subcanonical, we see that for $X \in \operatorname{Shv}_{\mathrm{weak}}(\mathsf{dAfnd}^\mathsf{op})$ one has 
\begin{equation}
    \operatorname{QCoh}(X) \xrightarrow[]{\sim} \underset{Y \in \mathsf{dAfnd}_{/X}^\mathsf{op}}{\operatorname{lim}} \operatorname{QCoh}(Y). 
\end{equation}
In particular the functor $\operatorname{QCoh}$ on $\operatorname{Shv}_{\mathrm{weak}}(\mathsf{dAfnd}^\mathsf{op})$ is right Kan extended from the full subcategory $\mathsf{dAfnd}^\mathsf{op}$.
\item We can also see from \eqref{eq:Shvgeneral} that the functor $\operatorname{QCoh}$ of \eqref{eq:QCohshv} commutes with all small limits. That is, if $(X_k)_{k \in \mathcal{K}}$ is a diagram in $\operatorname{Shv}_{\mathrm{weak}}(\mathsf{dAfnd})$, then the canonical morphism
\begin{equation}
\operatorname{QCoh}\big(\underset{k \in \mathcal{K}}{\operatorname{colim}} X_k\big) \to  \underset{k \in \mathcal{K}}{\operatorname{lim}}\operatorname{QCoh}(X_k). 
\end{equation}
is an equivalence. 
\end{enumerate}

\end{rmk}
It is clear from Definition \ref{defn:derivedstrongtopology} that sheaves on $\operatorname{Shv}_{\mathrm{weak}}(\mathsf{dAfnd})$ restrict to sheaves on $\mathsf{dRig}$ in the strong topology. In particular we obtain a functor 
\begin{equation}
    \operatorname{QCoh}: \mathsf{dRig}^\mathsf{op} \to \mathsf{CAlg}(\mathsf{Pr}_{\mathsf{st}}^L) 
\end{equation}
which is a sheaf on $\mathsf{dRig}$ in the strong topology and extends $\operatorname{QCoh}$ on $\mathsf{dAfnd}$. For $X \in \mathsf{dRig}$ we write $\widehat{\otimes}_X$ for the monoidal structure\footnote{We have chosen to suppress the fact that this is given by the derived tensor product for $X$ affinoid.} on $\operatorname{QCoh}(X)$.  For each morphism $f: X \to Y$ in $\mathsf{dRig}$ we denote the induced pullback functor by 
\begin{equation}
    f^*: \operatorname{QCoh}(Y) \to \operatorname{QCoh}(X).
\end{equation}
By construction this is symmetric monoidal and colimit-preserving. Since the categories are presentable this admits a right adjoint denoted by
\begin{equation}
    f_*: \operatorname{QCoh}(X) \to \operatorname{QCoh}(Y).
\end{equation}
Since $f^*$ is symmetric-monoidal and left adjoint to $f_*$ there is a canonical morphism
\begin{equation}\label{eq:projectionformuladRig}
    f_* \widehat{\otimes}_Y \operatorname{id} \to f_*(\operatorname{id} \widehat{\otimes}_X f^*). 
\end{equation}
If we are given a Cartesian square in $\mathsf{dRig}$:
\begin{equation}\label{eq:cartesiansquarerig}
\begin{tikzcd}[cramped]
	{X^\prime} & {Y^\prime} \\
	X & Y
	\arrow["{f^\prime}", from=1-1, to=1-2]
	\arrow["{g^\prime}"', from=1-1, to=2-1]
	\arrow["\ulcorner"{anchor=center, pos=0.125}, draw=none, from=1-1, to=2-2]
	\arrow["g"', from=1-2, to=2-2]
	\arrow["f", from=2-1, to=2-2]
\end{tikzcd}
\end{equation}
the fact that pullbacks are compatible with composition and left adjoint to pushforwards implies that there is a Beck-Chevalley transformation
\begin{equation}\label{eq:basechangedRig}
    g^*f_* \to f^\prime_* g^{\prime, *}. 
\end{equation}
\begin{lem}\label{lem:initialbasechangeprojformula}
With notations as above, assume that the morphism $f:X \to Y$ is \emph{qcqs}, c.f. Definition \ref{defn:quasicompactquasiseparated}. Then: 
\begin{enumerate}[(i)]
    \item (Base change). The morphism \eqref{eq:basechangedRig} is an equivalence, for any $g: Y^\prime \to Y$. 
    \item (Projection formula). The morphism \eqref{eq:projectionformuladRig} is an equivalence. 
    \item The functor $f_*$ commutes with all colimits.
\end{enumerate}
\end{lem}
Before proving Lemma \ref{lem:initialbasechangeprojformula}, we prove an intermediate Lemma.
\begin{lem}\label{lem:commutewithrestrictions}
Let $f: X \to Y$ be a qcqs morphism in $\mathsf{dRig}$. Then $f$ \emph{commutes with restrictions}, that is, if $g: Y^\prime \to Y$ is an analytic subspace of $Y$, then the Beck-Chevalley morphism \eqref{eq:basechangedRig} is an equivalence. 
\end{lem}
\begin{proof}
The proof is exactly the same as \cite[Lemma 2.4.16]{mann_p-adic_2022}; we reproduce it here for the reader's convenience. We proceed in stages.

\textit{Step 1}: First assume that $X$ and $Y$ are both derived affinoids. One chooses a cover $\{V_j^\prime \to Y^\prime\}_{j \in \mathscr{J}}$ of $Y^\prime$ by rational subspaces of $Y$. Let $\mathcal{J}$ be the family of finite nonempty subsets of $\mathscr{J}$ and for each $J \in \mathcal{J}$ set $V_J^\prime := \bigcap_{j \in J} V_j^\prime$ and $U_J^\prime := X^\prime \times_{Y^\prime} V_J^\prime$. Set $g_J: V_J^\prime \to Y$ and $g_J^\prime: U_J^\prime \to X$ to be the restrictions and let $f_J^\prime: U_J^\prime \to V_J^\prime$ be the base-change of $f^\prime$. Each $U_J^\prime$ and $V_J^\prime$ is affinoid, and hence by descent and Lemma \ref{lem:basechangeaffinoid} one has 
\begin{equation}
    g^*f_* \simeq \underset{J \in \mathcal{J}}{\operatorname{lim}} g_J^* f_* \simeq \underset{J \in \mathcal{J}}{\operatorname{lim}} f_{J,*}^\prime g_J^{\prime,*} \simeq f^\prime_* g^{\prime,*},
\end{equation}
proving the Lemma in this case. 

\textit{Step 2}: Now assume that $Y$ is affinoid and $X$ is an analytic subspace of an affinoid space $Z$. Let $\{U_i \to X\}_{i=1}^n$ be a covering by rational subspaces of $Z$ and let $\mathcal{I}$ be the family of all finite nonempty subsets $I \subseteq \{1,\dots, n\}$. For $I \in \mathcal{I}$ we set $U_I := \bigcap_{i \in I}U_i$. According to Remark \ref{rmk:colimitcovering}(i) then $f_* \xrightarrow[]{\sim} \operatorname{lim}_{I \in \mathcal{I}}f_{I,*}$ where $f_I : U_I \to X$ is the restriction. Using that $g^*$ commutes with finite limits, since we are working with stable $\infty$-categories, one deduces base-change in this case from Step 1. 

\textit{Step 3}: Now assume that $Y$ is affinoid and $X$ is arbitrary. Choose a finite covering $\{U_i \to X\}_{i=1}^n$ by affinoid subspaces, let $\mathcal{I}$ be as before and for each $I \in \mathcal{I}$ set $U_I := \bigcap_{i \in I} U_i$. This is an analytic subspace of an affinoid space and we again have $f_* \xrightarrow[]{\sim} \operatorname{lim}_{I \in \mathcal{I}}f_{I,*}$. Again using that $g^*$ commutes with finite limits, one deduces base-change in this case from Step 2.

\textit{Step 4}: Now assume that $Y$ is an analytic subspace of an affinoid space $Z$ and $X$ is arbitrary. One chooses a cover $\{V_j \to Y\}_{j \in \mathscr{J}}$ of  $Y$ by rational subspaces of $Z$. Let $\mathcal{J}$ be the family of finite nonempty subsets of $\mathscr{J}$ and for each $J \in \mathcal{J}$ set $V_J := \bigcap_{j \in J} V_j$ and $U_J := X \times_Y V_J$ and let $f_J: U_J \to V_J$ be the base-change of $f$. Each $V_J$ is affinoid and hence by Step 3 each $f_J$ commutes with restrictions. Therefore for each Cartesian section $(M_J)_{J \in \mathcal{J}} \in \operatorname{lim}_{J \in \mathcal{J}} \operatorname{QCoh}(U_J) \simeq \operatorname{QCoh}(X)$ one obtains a Cartesian section $(f_{J,*}M_J)_{J \in \mathcal{J}} \in \operatorname{lim}_{J \in \mathcal{J}} \operatorname{QCoh}(V_J) \simeq \operatorname{QCoh}(Y)$. The functor obtained in this way agrees with $f_*$. In particular, since the covering of $Y$ was arbitrary this implies (by definition of a Cartesian section), that $f_*$ commutes with restrictions. 

\textit{Step 5}: For arbitrary $X$, $Y$ one can take a cover $\{V_j \to Y\}_{j \in \mathcal{J}}$ of $Y$ by affinoid subspaces. Let $\mathcal{J}$ be the family of finite subsets of $\mathscr{J}$ and for each $J \in \mathcal{J}$ set $V_J := \bigcap_{j \in J} V_j$ and $U_J := X \times_Y V_J$ and let $f_J: U_J \to V_J$ be the base-change of $f$. Each $V_J$ is an analytic subspace of an affinoid space and hence by Step 4 each $f_J$ commutes with restrictions. Hence by the same reasoning as in Step 4, $f_*$ commutes with restrictions. 
\end{proof}
\begin{proof}[Proof of Lemma \ref{lem:initialbasechangeprojformula}]
(i): This is proved in exactly the same way as \cite[Proposition 2.4.21]{mann_p-adic_2022}. We reproduce the proof here for the reader's convenience. We proceed in steps.

\textit{Step 1}: Take affinoid coverings $\mathcal{U}$ of $Y$ and $\mathcal{U}^\prime$ of $Y^\prime$ such that $\mathcal{U}^\prime$ refines the pullback of $\mathcal{U}$ along $g$. For each $U^\prime \in \mathcal{U}$ let $t_{U^\prime}: U^\prime \to Y^\prime$ be the inclusion. By descent, the collection of pullback functors $\{t_{U^\prime}^*: U^\prime \in \mathcal{U}^\prime\}$ is conservative. Hence it suffices to check \eqref{eq:basechangedRig} after applying each $t_{U^\prime}^*$. By the commutation with restrictions proven in Lemma \ref{lem:commutewithrestrictions}, this reduces the proof of the Lemma to the case when $Y$ and $Y^\prime$ are both affinoid. In the remainder of the proof we will make this assumption. In particular this implies that $X$ and $X^\prime$ are both qcqs. 

\textit{Step 2}: Suppose that $X$ is an analytic subspace of an affinoid space $Z$. Choose a finite covering $\{U_i \to X\}_{i=1}^n$ of $X$ by rational subspaces of $Z$. Let $\mathcal{I}$ be the family of finite nonempty subsets of $\{1,\dots,n\}$. For each $I \in \mathcal{I}$ set $U_I := \bigcap_{j \in I} U_j$ and $U_I^\prime := U_I \times_Y Y^\prime$. By descent one has $f_* \xrightarrow[]{\sim} \operatorname{lim}_{I \in \mathcal{I}} f_{I,*}$ and  $f_*^\prime \xrightarrow[]{\sim} \operatorname{lim}_{I \in \mathcal{I}} f_{I,*}^\prime$. Each $U_I$ and $U_I^\prime$ is affinoid. Hence, using that $g^*$ commutes with finite limits, and Lemma \ref{lem:basechangeaffinoid}, one has 
\begin{equation}
    g^* f_* \simeq \underset{I \in \mathcal{I}}{\operatorname{lim}} g^*f_{I,*} \simeq \underset{I \in \mathcal{I}}{\operatorname{lim}} f_{I,*}^\prime g^{\prime,*} \simeq f^\prime_* g^{\prime,*},
\end{equation}
proving the Lemma in this case.

\textit{Step 3}: Suppose now that $X$ is arbitrary. Choose a finite covering $\{U_i \to X\}_{i=1}^n$ of $X$ by affinoid subspaces. Let $\mathcal{I}$ be the family of finite nonempty subsets of $\{1,\dots,n\}$. For each $I \in \mathcal{I}$ set $U_I := \bigcap_{j \in I} U_j$ and $U_I^\prime := U_I \times_Y Y^\prime$. By descent one has $f_* \xrightarrow[]{\sim} \operatorname{lim}_{I \in \mathcal{I}} f_{I,*}$ and  $f_*^\prime \xrightarrow[]{\sim} \operatorname{lim}_{I \in \mathcal{I}} f_{I,*}^\prime$. Each $U_I$ and $U_I^\prime$ is an analytic subspace of an affinoid space. Hence, using that $g^*$ commutes with finite limits, and Step 2, one has 
\begin{equation}
    g^* f_* \simeq \underset{I \in \mathcal{I}}{\operatorname{lim}} g^*f_{I,*} \simeq \underset{I \in \mathcal{I}}{\operatorname{lim}} f_{I,*}^\prime g^{\prime,*} \simeq f^\prime_* g^{\prime,*},
\end{equation}
proving the Lemma in this case.

(ii): Take affinoid coverings $\mathcal{V}$ of $Y$ and $\mathcal{U}$ of $X$ such that $\mathcal{U}$ refines the pullback of $\mathcal{V}$ along $f$. For each $U^\prime \in \mathcal{U}$ let $t_{U}: U \to X$ be the inclusion. By descent the collection of pullback functors $\{t_U^*: U \in \mathcal{U}\}$ is conservative. Hence it suffices to check that \eqref{eq:projectionformuladRig} is an equivalence after applying each $t^*_U$. By the commutation with restrictions proven in Lemma \ref{lem:commutewithrestrictions}, and using that pullback functors are symmetric-monoidal, this then reduces the Lemma to the case when $X$ and $Y$ are both affinoids, which is Lemma \ref{lem:basechangeaffinoid}. 

(iii): Let $(M_k)_{k \in \mathcal{K}}$ be a diagram in $\operatorname{QCoh}(X)$. We need to show that the canonical morphism
\begin{equation}\label{eq:colimK}
\underset{k \in \mathcal{K}}{\operatorname{colim}} f_* M_k \to f_* \underset{k \in \mathcal{K}}{\operatorname{colim}} M_k,
\end{equation}
is an equivalence. Using the same notations as in the proof of (ii), it suffices to check that \eqref{eq:colimK} is an equivalence after applying each $t_U^*$. By the commutation with restrictions proven in Lemma \ref{lem:commutewithrestrictions}, and using that each $t_U^*$ is colimit-preserving (it is a left adjoint), one reduces to the case when $X$ and $Y$ are both affinoids, which follows from Lemma \ref{lem:basechangeaffinoid}. 
\end{proof} 
\begin{cor}\label{cor:QCohqcqs}
The functor $\operatorname{QCoh}$ extends to a six-functor formalism 
\begin{equation}
    \operatorname{QCoh}: \operatorname{Corr}(\mathsf{dRig}, \mathrm{qcqs}) \to \mathsf{Pr}^{L, \otimes}_{\mathsf{st}}. 
\end{equation}
In this six-functor formalism every qcqs morphism $f$ satisfies $f_! = f_*$. 
\end{cor}
\begin{proof}
This is immediate from Lemma \ref{lem:initialbasechangeprojformula} and \cite[Proposition A.5.10]{mann_p-adic_2022}. To be clear, in the language of \emph{loc. cit.} we take the \emph{suitable decomposition} to be $(I,P) = (\text{equivalences},\text{qcqs})$. 
\end{proof}
\subsection{Six-functor formalism in derived rigid geometry}\label{sec:derivedrigidsix}
We will apply the formalism of \S\ref{subsec:sixf3} in the following set-up (with notations as in that section):
\begin{itemize}
    \item[$\star$] We take $\mathscr{V} := D(\mathsf{CBorn}_K)$, so that $\mathscr{E} := \mathsf{CAlg}(D(\mathsf{CBorn}_K))^\mathsf{op}$ and we consider $\mathsf{dAfnd} \subseteq \mathscr{E}$. We take $\tau := \mathrm{weak}$ to be the weak topology on $\mathsf{dAfnd}$.  
\end{itemize}
By Lemma \ref{lem:dAfndAlgbasic1}, the assumptions of \S\ref{subsec:sixf3} are satisfied; we obtain a six-functor formalism 
\begin{equation}\label{eq:sixfunctorShvweak}
    \operatorname{QCoh}: \operatorname{Corr}(\operatorname{Shv}_\mathrm{weak}(\mathsf{dAfnd}), E)^\otimes \to \mathsf{Pr}^{L,\otimes}_\mathsf{st}, 
\end{equation}
with the following properties:
\begin{itemize}
    \item[$\star$] The class $E \supseteq \mathrm{rep}$ is stable under disjoint unions, $*$-local on the target, $!$-local on the source, is tame, and satisfies $E \subseteq \delta E$. 
    \item[$\star$] Every morphism $f \in \mathrm{rep}$ satisfies $f_! \simeq f_*$ (Corollary \ref{cor:lowershriekcor}). 
\end{itemize}
The purpose of this section is to prove the following Theorem:
\begin{thm}\label{thm:belongstoE}
In the six-functor formalism \eqref{eq:sixfunctorShvweak}, every qcqs morphism $f: X \to Y$ of derived rigid varieties belongs to the class $E$. The restriction of this six-functor formalism along $\operatorname{Corr}(\mathsf{dRig}, \mathrm{qcqs}) \to  \operatorname{Corr}(\operatorname{Shv}_\mathrm{weak}(\mathsf{dAfnd}), E)$ is equivalent to the six-functor formalism constructed in Corollary \ref{cor:QCohqcqs}. In particular, for every qcqs morphism $f: X \to Y$ between objects of $\mathsf{dRig}$ there is a canonical equivalence $f_! \simeq f_*$. 
\end{thm}
The proof is postponed to the end of this subsection. Essentially, we want to prove that the six-functor formalism $\operatorname{QCoh}$ extends uniquely from $(\mathsf{dAfnd}, \mathrm{all})$ to $(\mathsf{dRig}, \mathrm{qcqs}$).  
\begin{lem}\label{lem:finiteuniv!descent}
Let $X \in \mathsf{dRig}$, and let $\{U_i \to X\}_{i=1}^n$ be a finite covering by analytic subspaces. Let $\mathcal{I}$ be the family of finite nonempty subsets of $\{1, \dots,n\}$ and for each $I \in \mathcal{I}$ set $U_I := \bigcap_{i \in I} U_i$. Assume that for each $I \in \mathcal{I}$, the morphism $U_I \to X$ is quasi-compact\footnote{It is automatically quasi-separated, being a monomorphism.}. Then the canonical morphism 
\begin{equation}
    \operatorname{QCoh}^!(X) \to \underset{I \in \mathcal{I}}{\operatorname{lim}} \operatorname{QCoh}^!(U_I)
\end{equation}
is an equivalence. Here we are using the six-functor formalism from Corollary \ref{cor:QCohqcqs}. 
\end{lem}
\begin{proof}
Let $t_I: U_I \to X$ be the inclusions. Since $\operatorname{QCoh}^*$ is a sheaf, we know that there is an equivalence of categories 
\begin{equation}
     \operatorname{QCoh}^*(X) \to \underset{I \in \mathcal{I}}{\operatorname{lim}} \operatorname{QCoh}^*(U_I)
\end{equation}
where the functor from left to right sends $M \to (t_I^*M)_I$ and the functor from right to left sends $(M_I)_I \to \operatorname{lim}_I t_{I,*}M_I$. In particular the counit $\operatorname{id}_X \to \operatorname{lim}_I t_{I,*}t_I^*$ is an equivalence. In order to prove the Lemma we need to show that $(t_I^!)_I$ induces an equivalence 
\begin{equation}
    \operatorname{QCoh}^!(X) \to \underset{I \in \mathcal{I}}{\operatorname{lim}} \operatorname{QCoh}^!(U_I).
\end{equation}
There is a natural adjunction in which this functor is right adjoint to $\operatorname{colim}_I t_{I,!}$. We note that all the $t_I$ are qcqs and therefore $t_{I,!} = t_{I,*}$.

Firstly, we will check that the counit is an equivalence. By the previous, we know that $1_X \xrightarrow[]{\sim} \operatorname{lim}_I t_{I,*} 1_{U_I}$. Let $M \in \operatorname{QCoh}(X)$. We have equivalences 
\begin{equation}
\begin{aligned}
        M &\simeq \underline{\operatorname{Hom}}_X(1_X,M)\\
        &\simeq \underset{I}{\operatorname{colim}}\underline{\operatorname{Hom}}_X(t_{I,*} 1_{U_I},M) \\
        &\simeq \underset{I}{\operatorname{colim}}t_{I,*}\underline{\operatorname{Hom}}_{U_I}(1_{U_I},t_I^!M) \\
        &\simeq  \underset{I}{\operatorname{colim}}t_{I,*}t_I^!M. 
\end{aligned}
\end{equation}
Here, in the second line we used the property of stable $\infty$-categories, and in the third line we used the identity which is adjoint to the projection formula, c.f. Remark \ref{rmk:sixfunctorsbasic}(ii). Therefore the counit is an equivalence. To check that the unit is an equivalence we need to show that, given $(M_I)_I$, for each $J \in \mathcal{I}$ the natural morphism $ M_J \to t_{J}^! \operatorname{colim}_{I} t_{I,*}M_I$
is an equivalence. Since the relevant $\infty$-categories are stable, we can exchange $t_I^!$ with this finite colimit, and then use base-change, whence the covering becomes split and the morphism is obviously an equivalence. 
\end{proof}
\begin{cor}\label{cor:finiteuniv!descent2}
With notations as in Lemma \ref{lem:finiteuniv!descent}. Set $Y := \coprod_{i=1}^n U_i$. Then, the canonical morphism
\begin{equation}
    \operatorname{QCoh}^!(X) \to \underset{[m] \in \Delta}{\operatorname{lim}}\operatorname{QCoh}^!(Y^{m+1/X})
\end{equation}
is an equivalence. In particular $t:= \coprod_{i=1}^n t_i: Y \to X$ is of universal $!$-descent. 
\end{cor}
\begin{proof}
This is quite similar to \cite[Proposition 2.6.3]{mann_p-adic_2022} which is a consequence of the Lurie-Beck-Chevalley condition \cite[Corollary 4.7.5.3]{HigherAlgebra}. 

We first check condition (1) in \cite[Corollary 4.7.5.3]{HigherAlgebra}. Let $(M_m)_{[m] \in \Delta_+^\mathsf{op}}$ be a $t^!$-split simplicial object in $\operatorname{QCoh}(X)$. In particular, it is $t_*t^!$-split and in fact $t_{I,*}t_I^!$-split for every $I \in \mathcal{I}$. This implies that $(\operatorname{sk}_k t_{I,*}t_I^! M_\bullet)_{k \geqslant0}$ is a constant Ind-object. By dual arguments to \cite[Proposition 3.10]{MathewGalois}, it then follows that $(\operatorname{sk}_k \operatorname{colim}_I t_{I,*}t_I^! M_\bullet)_{k \geqslant0}$ is a constant Ind-object. By Lemma \ref{lem:finiteuniv!descent} above this is equivalent to the Ind-object $(\operatorname{sk}_k M_\bullet)_{k \geqslant0}$. In this case it is obvious that $t^!$ commutes with the geometric realization of $M_\bullet$, since the relevant $\infty$-categories are stable. 

Now condition (2) in \cite[Corollary 4.7.5.3] {HigherAlgebra} follows from base-change, and conservativity of $t^!$ follows from Lemma \ref{lem:finiteuniv!descent}. Therefore, the result of \emph{loc. cit.} gives the desired equivalence.
\end{proof}
\begin{defn}
A morphism $f: X \to Y$ in $\operatorname{Shv}_{\mathrm{weak}}(\mathsf{dAfnd})$ is called \emph{semi-separated} if the diagonal $\Delta_f : X \to X \times_YX$ belongs to the class $\mathrm{rep}$ of representable morphisms. An object $X \in \operatorname{Shv}_{\mathrm{weak}}(\mathsf{dAfnd})$ is called \emph{semi-separated} if the morphism $X \to *$ is semi-separated. 
\end{defn}
\begin{prop}\label{prop:qcqsUnique}
The six-functor formalism $\operatorname{QCoh}$ extends uniquely from $(\mathsf{dAfnd}, \mathrm{all})$ to $(\mathsf{dRig}, \mathrm{qcqs})$. 
\end{prop}
\begin{proof}
By \cite[Proposition A.5.16]{mann_p-adic_2022} the six-functor formalism $\operatorname{QCoh}$ extends uniquely from $(\mathsf{dAfnd}, \mathrm{all})$ to $(\mathsf{qcssdRig}, \mathrm{rep})$, where $\mathsf{qcssdRig}$ denotes the category of quasi-compact and \emph{semi-separated} derived rigid spaces, i.e., those having representable diagonal. Now Corollary \ref{cor:finiteuniv!descent2} together with \cite[Proposition A.5.14]{mann_p-adic_2022} implies that $\operatorname{QCoh}$ further extends uniquely to a six-functor formalism on $(\mathsf{qcssdRig},\mathrm{all})$. To be clear, for the application of \cite[Proposition A.5.14]{mann_p-adic_2022} here, one takes the class $S$ of \emph{special covers} (in the language of \emph{loc. cit.}) to be the class of finite coverings by affinoids.

Now \cite[Proposition A.5.16]{mann_p-adic_2022} again implies that the six-functor formalism $\operatorname{QCoh}$ extends uniquely from $(\mathsf{qcssdRig},\mathrm{all})$ to $(\mathsf{qcqsdRig}, F)$ where $F$ is the collection of edges which are representable in $\mathsf{qcssdRig}$. Now Corollary \ref{cor:finiteuniv!descent2} again with \cite[Proposition A.5.14]{mann_p-adic_2022} implies that $\operatorname{QCoh}$ further extends uniquely to a six-functor formalism on $(\mathsf{qcqsdRig},\mathrm{all})$ (again, for the application of \cite[Proposition A.5.14]{mann_p-adic_2022}, one takes the class $S$ of \emph{special covers} to be the class of finite coverings by affinoids).  

Finally, an application of \cite[Proposition A.5.16]{mann_p-adic_2022} implies that $\operatorname{QCoh}$ extends uniquely from $(\mathsf{qcqsdRig},\mathrm{all})$ to $(\mathsf{dRig}, \mathrm{qcqs})$.
\end{proof}
\begin{prop}\label{prop:Econtainsqcqs}
In the six-functor formalism \eqref{eq:QCohshv}, the class $E$ contains all qcqs morphisms between derived rigid spaces. 
\end{prop}
\begin{proof}
Since the class $E$ is \emph{$*$-local on the target} one reduces immediately to the case when $f: X \to Y$ is a qcqs morphism with $Y$ affinoid. 

We claim that if $g:Z \to W$ is a quasi-compact and \emph{semi-separated} morphism of derived rigid varieties with affinoid target, there is a canonical equivalence $g_! \simeq g_*$ in this six-functor formalism. Indeed, because $\Delta_g \in \mathrm{rep}$, one has $\Delta_{g,!} \simeq \Delta_{g,*}$ and from this one obtains a canonical morphism $g_! \to g_*$, which is the adjoint to the composite
\begin{equation}
    g^*g_! \simeq \pi_{1,!} \pi_2^* \to \pi_{1,!} \Delta_{g,*} \Delta^*_g \pi_2^* \simeq \pi_{1,!} \Delta_{g,!} \Delta^*_g \pi_2^* \simeq \operatorname{id};
\end{equation}
here $\pi_1, \pi_2: Z \times_W Z \to Z$ are the projections, and we used base-change. Now this morphism $g_! \to g_*$ is an equivalence: by $*$-descent on the source, this follows if $g_*$ satisfies base-change, but we have already proved this in Lemma \ref{lem:initialbasechangeprojformula} above. 

Now take a finite cover $\{U_i \to X\}_i$ of $X$ by affinoids; then all the $t_I: U_I \to X$ are quasi-compact\footnote{And semi-separated because they are monomorphisms.}, hence by the above they satisfy $t_{I,!} \simeq t_{I,*}$, hence same argument as in Lemma \ref{lem:finiteuniv!descent} and Corollary \ref{cor:finiteuniv!descent2} above shows that $\{U_i \to X\}_i$ is of universal $!$-descent. Since the class $E$ is \emph{$!$-local on the source}, we conclude. 
\end{proof}
\begin{proof}[Proof of Theorem \ref{thm:belongstoE}]
By Proposition \ref{prop:Econtainsqcqs} above, we may consider the restriction from $(\operatorname{Shv}_\mathrm{weak}(\mathsf{dAfnd}), E)$ to $(\mathsf{dRig},\mathrm{qcqs})$. Now the Theorem follows immediately from the unicity proved in Proposition \ref{prop:qcqsUnique} together with Corollary \ref{cor:QCohqcqs}.  
\end{proof}
\subsection{Infinite covers of universal $!$-descent}\label{sec:univ!descent}
In the remainder of this section, we introduce the following notation. Let $X \in \mathsf{dRig}$ and let $S \subseteq |X|$ be a closed subset of the underlying topological space of $X$. Set $V := |X| \setminus S \subseteq |X|$, which we may identify with an analytic subspace $j: V \hookrightarrow X$. We define full subcategories 
\begin{equation}
\begin{aligned}
    \Gamma_S \operatorname{QCoh}(X) \subseteq \operatorname{QCoh}(X) && \text{ and } && \operatorname{L}_S \operatorname{QCoh}(X) \subseteq \operatorname{QCoh}(X),
\end{aligned}
\end{equation}
as the full subcategories spanned by objects $M$ such that $j^*M \simeq 0$, and $j^!M \simeq 0$, respectively. If we assume that $j$ is quasi-compact, so that $j_! = j_*$, then base-change implies that the inclusions of these full subcategories admit adjoints given by 
\begin{equation}
\begin{aligned}
    \Gamma_S := \operatorname{Fib}(\operatorname{id} \to j_*j^*) && \text{ and } &&  \operatorname{L}_S := \operatorname{Cofib}(j_!j^! \to \operatorname{id}).
\end{aligned}
\end{equation}
To be clear, $\Gamma_S$ is right adjoint to the inclusion $\Gamma_S \operatorname{QCoh}(X) \subseteq \operatorname{QCoh}(X)$ and $\operatorname{L}_S$ is left adjoint to the inclusion $\operatorname{L}_S \operatorname{QCoh}(X) \subseteq \operatorname{QCoh}(X)$. We note that $\Gamma_S$ is colimit-preserving and $\operatorname{L}_S$ is limit-preserving. 
\begin{lem}\label{lem:prosystems}
Let $X \in \mathsf{dRig}$. Suppose that we are given sequences $V_1 \supseteq V_2 \supseteq V_3 \supseteq \dots$ and $U_1 \subseteq U_2 \subseteq U_3 \subseteq \dots$ of analytic subspaces of $X$ such that:
\begin{itemize}
    \item[$\star$] For every $n \geqslant1$ the morphisms $V_n \to X$ and $U_n \to X$ are quasi-compact and one has $U_n \cup V_n = X$ and $U_n \cap V_{n+1} = \emptyset$.
\end{itemize}
Let us write $S_n := |X| \setminus V_n$ and $\Gamma_n \operatorname{QCoh}(X) := \Gamma_{S_n} \operatorname{QCoh}(X)$. Then, there is an equivalence of Pro-systems 
\begin{equation}
     \underset{n}{\prolim}\operatorname{QCoh}^!(U_n) \simeq  \underset{n}{\prolim} \Gamma_n \operatorname{QCoh}(X),
\end{equation}
which is compatible with the maps from $\operatorname{QCoh}(X)$.
\end{lem}
\begin{proof}
Let us write $j_n: V_n \to X$ and $t_n: U_n \to X$ for the inclusions. Let us also abbreviate $\Gamma_n := \Gamma_{S_n}$ and write $\operatorname{incl}_n$ for the inclusion of the full subcategory $\operatorname{\Gamma}_n \operatorname{QCoh}(X) \subseteq \operatorname{QCoh}(X)$. Consider the diagram
\begin{multline}
    \dots \leftarrow \Gamma_n \operatorname{QCoh}(X) \xleftarrow[]{\Gamma_n t_{n,*}} \operatorname{QCoh}(U_n) \xleftarrow[]{t_n^! \operatorname{incl}_{n+1}} \\ \Gamma_{n+1} \operatorname{QCoh}(X) \xleftarrow[]{\Gamma_{n+1}t_{n+1,*}} \operatorname{QCoh}(U_{n+1}) \leftarrow \dots
\end{multline}
We need to show that there are natural equivalences 
\begin{equation}
    t_n^! \operatorname{incl}_{n+1} \Gamma_{n+1} t_{n+1,*} \simeq t_{n,n+1}^!
\end{equation}
and 
\begin{equation}
  \Gamma_n t_{n,*} t_n^! \operatorname{incl}_{n+1} \simeq \Gamma_{n,n+1}.
\end{equation}
In order to do this, we first observe that the left adjoint to $t_n^! \operatorname{incl}_{n+1}$ is given by $t_{n,*}$ since this already factors through $\Gamma_{n+1}\operatorname{QCoh}(X)$: by base-change one has $j^*_{n+1} t_{n,*} \simeq 0$ as $U_n \cap V_{n+1} = \emptyset$. Therefore, by passing to adjoints, it is enough to show that
\begin{equation}\label{eq:t_n1}
    t_{n+1}^* \operatorname{incl}_{n+1} t_{n,*} \simeq t_{n,n+1,*}
\end{equation}
and
\begin{equation}\label{eq:t_n2}
    t_{n,*} t_n^* \operatorname{incl}_n \simeq \operatorname{incl}_{n,n+1}. 
\end{equation}
The equivalence \eqref{eq:t_n1} is a consequence of base-change. The equivalence \eqref{eq:t_n2} is a consequence of descent applied to the covering $X = V_n \cup U_n$.
\end{proof}
\begin{lem}\label{lem:prosystems2}
Let $X \in \mathsf{dRig}$. Suppose that we are given sequences $V_1 \supseteq V_2 \supseteq V_3 \supseteq \dots$ and $U_1 \subseteq U_2 \subseteq U_3 \subseteq \dots$ of analytic subspaces of $X$ such that:
\begin{itemize}
    \item[$\star$] For every $n \geqslant 1$ the morphisms $V_n \to X$ and $U_n \to X$ and $U_n \cap V_n \to X$ are quasi-compact, and one has $U_n \cup V_n = X$ and $U_n \cap V_{n+1} = \emptyset$.
\end{itemize}
Let us write $S_n := |X| \setminus V_n$ and $\operatorname{L}_n \operatorname{QCoh}(X) := \operatorname{L}_{S_n} \operatorname{QCoh}(X)$. Then, there is an equivalence of Pro-systems 
\begin{equation}
     \underset{n}{\prolim}\operatorname{QCoh}^*(U_n) \simeq  \underset{n}{\prolim} \operatorname{L}_n \operatorname{QCoh}(X),
\end{equation}
which is compatible with the maps from $\operatorname{QCoh}(X)$.
\end{lem}
\begin{proof}
Let us write $j_n: V_n \to X$ and $t_n: U_n \to X$ for the inclusions. Let us also abbreviate $\operatorname{L}_n := \operatorname{L}_{S_n}$ and write $\operatorname{incl}_n^\prime$ for the inclusion of the full subcategory $\operatorname{L}_n \operatorname{QCoh}(X) \subseteq \operatorname{QCoh}(X)$. Consider the diagram
\begin{multline}
    \dots \leftarrow \operatorname{L}_n \operatorname{QCoh}(X) \xleftarrow[]{\operatorname{L}_n t_{n,*}} \operatorname{QCoh}(U_n) \xleftarrow[]{t_n^* \operatorname{incl}_{n+1}^\prime} \\ \operatorname{L}_{n+1} \operatorname{QCoh}(X) \xleftarrow[]{\operatorname{L}_{n+1}t_{n+1,*}} \operatorname{QCoh}(U_{n+1}) \leftarrow \dots
\end{multline}
We need to show that there are natural equivalences 
\begin{equation}
    t_n^* \operatorname{incl}^\prime_{n+1} \operatorname{L}_{n+1}t_{n+1,*} \simeq t_{n,n+1}^*
\end{equation}
and
\begin{equation}
    \operatorname{L}_n t_{n,*} t_n^* \operatorname{incl}^\prime_{n+1} \simeq \operatorname{L}_{n,n+1}.
\end{equation}
In order to do this, we first observe that the right adjoint to $t_n^* \operatorname{incl}^\prime_{n+1}$ is given by $t_{n,*}$, since this already factors through $\operatorname{
L}_{n+1} \operatorname{QCoh}(X)$: base-change one has $j_{n+1}^!t_{n,*} \simeq 0$ as $U_n \cap V_{n+1} = \emptyset$. Therefore, by passing to adjoints, it is enough to show that
\begin{equation}\label{eq:tnn+1equiv}
    t_{n+1}^! \operatorname{incl}_{n+1}^\prime t_{n,*} \simeq t_{n,n+1,*} 
\end{equation}
and
\begin{equation}\label{eq:inclnprimeequiv}
    t_{n,*}t_n^! \operatorname{incl}_n^\prime \simeq \operatorname{incl}_{n,n+1}^\prime.
\end{equation}
The equivalence \eqref{eq:tnn+1equiv} is an consequence of base-change. The equivalence \eqref{eq:inclnprimeequiv} is a consequence of Lemma \ref{lem:finiteuniv!descent} applied to the covering $X = V_n \cup U_n$.
\end{proof}
\begin{prop}\label{prop:countableshriekdescent}
Let $X \in \mathsf{dRig}$. Suppose that we are given sequences $V_1 \supseteq V_2 \supseteq V_3 \supseteq \dots$ and $U_1 \subseteq U_2 \subseteq U_3 \subseteq \dots $ of analytic subspaces of $X$ such that, 
\begin{itemize}
    \item[$\star$] For every $n \geqslant1$ the morphisms $V_n \to X$ and $U_n \to X$ are quasi-compact and one has $U_n \cup V_n = X$ and $U_n \cap V_{n+1} = \emptyset$.
    \item[$\star$] One has $\bigcup_{n \geqslant1} U_n = X$. 
\end{itemize}
Then, the natural morphism $\operatorname{QCoh}^!(X) \to \operatorname{lim}_{n} \operatorname{QCoh}^!(U_n)$ is an equivalence. 
\end{prop}
\begin{proof}
We use notations as in Lemma \ref{lem:prosystems}. We claim that the canonical morphism 
\begin{equation}\label{eq:Gammanequiv}
    \operatorname{QCoh}(X) \to \underset{n}{\operatorname{lim}} \Gamma_n \operatorname{QCoh}(X) 
\end{equation}
sending $M \mapsto (\Gamma_n M)_n$, is an equivalence. This functor is right adjoint to the functor sending $(M_n)_n \mapsto \operatorname{colim}_n M_n$. Let $M \in \operatorname{QCoh}(X)$. We first check that the counit morphism
\begin{equation}\label{eq:gammancounit}
  \underset{n}{\operatorname{colim}} \operatorname{incl}_n\Gamma_n M \to M
\end{equation}
is an equivalence. By descent, it suffices to check this after applying $t_{m}^*$ for each $m \geqslant0$. By using that $t_m^*$ commutes with colimits, and that $U_m \cap V_{m+1} = \emptyset$, together with base-change, we see that \eqref{eq:gammancounit} is indeed an equivalence after applying $t_m^*$. Now we check the unit morphism. We need to show that for each Cartesian section $(M_n)_n$ belonging to the right side of \eqref{eq:Gammanequiv}, and each $m \geqslant0$, the morphism
\begin{equation}\label{eq:Gammanunitmorphism}
    M_m \to \Gamma_m \big(\underset{n}{\operatorname{colim}} \operatorname{incl}_n M_n\big),
\end{equation}
is an equivalence. This follows from the fact that $\Gamma_m$ commutes with colimits. Therefore \eqref{eq:Gammanequiv} is an equivalence, and we can appeal to Lemma \ref{lem:prosystems} to deduce the statement of the Proposition.
\end{proof}
\begin{cor}\label{cor:countableshriekdescent}
With notations as in Proposition \ref{prop:countableshriekdescent}. Set $Y := \coprod_{n \geqslant1} U_n$. Then, the canonical morphism 
\begin{equation}
    \operatorname{QCoh}^!(X) \to \underset{[m] \in \Delta}{\operatorname{lim}} \operatorname{QCoh}^!(Y^{m+1/X})
\end{equation}
is an equivalence. In particular $Y \to X$ is of universal $!$-descent.
\end{cor}
\begin{proof}
By using that coproducts in $\mathsf{dRig}$ are universal, and that $\operatorname{QCoh}^!$ commutes with products, this follows from Proposition \ref{prop:countableshriekdescent} and the Bousfield-Kan formula for $\infty$-categories \cite{MazelGeeGrothendieck}.
\end{proof}
\begin{example}\label{ex:Affineline!able}
Let $\varrho > 1, \varrho \in |K^\times|$ be an element of the valuation group of $K$, and let $X := \mathbf{A}^1_K := \operatorname{colim}_n \mathbf{D}^1_K(\varrho^n)$ be the rigid-analytic affine line, regarded as an object of $\mathsf{dRig}$. Here $\mathbf{D}_K(\varrho^n) := \operatorname{dSp}(K \langle \varrho^{-n} T \rangle)$ is the rigid-analytic disk of radius $\varrho^n$. With notations as in Proposition \ref{prop:countableshriekdescent}, one can take $U_n:= \mathbf{D}^1_K(\varrho^{2n})$ and $V_n:= \mathbf{A}_{[\varrho^{2n-1}, \infty)}$, and the hypotheses of the Proposition are satisfied, where the latter is the rigid-analytic annulus of radius $\geqslant\varrho^{2n-1}$. 
\end{example}
\subsection{Local cohomology}\label{sec:localcoh}
In this section, we will develop in different generality an idea that was used in \S\ref{sec:univ!descent}. A classical theory of local cohomology for rigid- and complex-analytic varieties was developed by Kisin in \cite{KisinLocalCohomology}. 

Let $X \in \mathsf{dRig}$ and let $S \subseteq |X|$ be a closed subset of the underlying topological space of $X$. Then $U := |X| \setminus S \subseteq |X|$ may be identified with an analytic subspace $j: U \hookrightarrow X$\footnote{Indeed, from now on, we may deliberately confuse open subsets $U \subseteq |X|$ with analytic subspaces $U \hookrightarrow X$.}. We define
\begin{equation}
    \Gamma_S \operatorname{QCoh}(X) \subseteq \operatorname{QCoh}(X)
\end{equation}
as the full subcategory of $\operatorname{QCoh}(X)$ spanned by objects $M$ such that $j^*M \simeq 0$. Since $j^*$ commutes with colimits, it follows that the inclusion $\operatorname{incl}_S$ of this full subcategory commutes with colimits, and therefore\footnote{Since the categories are presentable.} admits a right adjoint:
    \begin{equation}
    \operatorname{incl}_S : \Gamma_S\operatorname{QCoh}(X) \leftrightarrows \operatorname{QCoh}(X): \Gamma_S,
\end{equation}
which we have denoted by $\Gamma_S$. Now suppose in addition that $j$ belongs to the class $E$ of \S\ref{sec:derivedrigidsix} so that it is $!$-able. There is a canonical morphism $j^! \to j^*$, defined as the composite
\begin{equation}
    j^! \to j^!j_*j^* \simeq j^*,
\end{equation}
here the first morphism comes from the unit of $j^* \dashv j_*$ and the second is by base-change. \emph{In this section we will be interested in the situation when the canonical morphism $j^! \to j^*$ is an equivalence.}\footnote{Note that this is \emph{not} satisfied for the inclusions $j_n: U_n \to X$ appearing in the proof of Proposition \ref{prop:countableshriekdescent}.} Under this assumption, $\Gamma_S\operatorname{QCoh}(X)$ may equivalently be described as the full subcategory of $\operatorname{QCoh}(X)$ spanned by objects $M$ such that $j^!M \simeq 0$. Since $j^!$ commutes with limits, it follows that the inclusion $\operatorname{incl}_S$ also commutes with limits, and therefore admits a left adjoint:
\begin{equation}
    \operatorname{L}_S : \Gamma_S\operatorname{QCoh}(X) \leftrightarrows \operatorname{QCoh}(X): \operatorname{incl}_S,
\end{equation}
which we have denoted\footnote{I chose this notation because $\operatorname{L}$ looks like $\Gamma$ upside down, and also because the functor is a left adjoint.} by $\operatorname{L}_S$. Various formal properties of these functors are listed in the next Proposition.
\begin{prop}\label{prop:localcohformalproperties}
Let $X \in \mathsf{dRig}$ and let  $S \subseteq |X|$ be a closed subset. Set $U := |X| \setminus S$ and let $j: U \to X$ be the inclusion of the corresponding analytic subspace. Assume that $j \in E$ and that the canonical morphism $j^! \to j^*$ is an equivalence. Then:
\begin{enumerate}[(i)]
    \item There is an equivalence $\Gamma_S \simeq \operatorname{Fib}(\operatorname{id} \to j_*j^*)$.
    \item There are equivalences $\operatorname{L}_S \simeq \operatorname{Cofib}(j_!j^! \to \operatorname{id}) \simeq \operatorname{Cofib}(j_!j^! 1_X \to 1_X) \widehat{\otimes}_X \operatorname{id}$, so that the functor $\operatorname{L}_S$ is given by tensoring with the idempotent algebra object $\operatorname{Cofib}(j_!j^! 1_X \to 1_X)$. In particular $\Gamma_S\operatorname{QCoh}(X)$ is symmetric-monoidal, with the monoidal structure inherited from $\operatorname{QCoh}(X)$ and tensor-unit given by $\operatorname{Cofib}(j_!j^! 1_X \to 1_X)$. 
    \item In the sequence 
    \begin{equation}\label{eq:localcohomologyexactseq}
    \Gamma_S(\operatorname{QCoh}(X)) \xrightarrow[]{\operatorname{incl}_S} \operatorname{QCoh}(X) \xrightarrow[]{j^*} \operatorname{QCoh}(U)
    \end{equation}
the right adjoints\footnote{Note especially here, that we do not assert that the right adjoints $j_*$ and $\Gamma_S$ commute with colimits.} $\Gamma_S$ and $j_*$ satisfy $\Gamma_S\operatorname{incl}_S \simeq \operatorname{id}$ and $j^*j_* \simeq \operatorname{id}$.
\item In the sequence 
\begin{equation}\label{eq:inverseimageexactseq}
    \operatorname{QCoh}(U) \xrightarrow[]{j_!} \operatorname{QCoh}(X) \xrightarrow[]{\operatorname{L}_S} \Gamma_S\operatorname{QCoh}(X),
\end{equation}
the composite $\operatorname{L}_S j_! \simeq 0$, and the right adjoints $j^!$ and $\operatorname{incl}_S$ satisfy $j^!j_! \simeq \operatorname{id}$ and $\operatorname{L}_S \operatorname{incl}_S \simeq \operatorname{id}$. Moreover, the right adjoints $j^!$ and $\operatorname{incl}_S$ commute with colimits, so that \eqref{eq:inverseimageexactseq} is a split-exact sequence in the sense of \cite[Appendix B]{jiang_grothendieck_2023}.
\end{enumerate}
\end{prop}
\begin{proof}
(i): Let us temporarily denote $F:= \operatorname{Fib}(\operatorname{id} \to j_*j^*)$. We will show that 
\begin{equation}\label{eq:F_iproperty}
\begin{aligned}
    F \operatorname{incl}_S \simeq \operatorname{id} && \text{ and } && j^* F \simeq 0.
\end{aligned}
\end{equation}
By using that $j^! \xrightarrow[]{\sim} j^*$, and base-change, one has 
\begin{equation}
    j^*F \simeq \operatorname{Fib}(j^* \to j^*j_*j^*) \simeq \operatorname{Fib}(j^* \to j^!j_*j^*) \simeq \operatorname{Fib}(j^* \to j^*) \simeq 0. 
\end{equation}
By definition, we have $j^* \operatorname{incl}_S \simeq 0$, so $F \operatorname{incl}_S \simeq \operatorname{id}$. We can define a counit $\varepsilon : \operatorname{incl}_S F \to \operatorname{id}$ coming from that canonical morphism $\operatorname{Fib}(\operatorname{id} \to j_*j^*) \to \operatorname{id}$ and a unit morphism $\eta : \operatorname{id} \simeq F \operatorname{incl}_S$, and one can verify the zig-zag identities using \eqref{eq:F_iproperty}. Therefore, by the uniqueness of adjoints, we obtain $\Gamma_S \simeq  \operatorname{Fib}(\operatorname{id} \to j_*j^*)$. 

(ii): The proof that $\operatorname{L}_S \simeq \operatorname{Cofib}(j_!j^! \to \operatorname{id})$ is quite similar to the proof of (i) and so we omit it. For the second part, it suffices to show that there is an natural equivalence
\begin{equation}
    j_! j^! 1_X \widehat{\otimes}_X \operatorname{id} \xrightarrow[]{\sim} j_!j^!.
\end{equation}
Indeed, one has equivalences $j_!j^! 1_X \widehat{\otimes}_X \operatorname{id} \simeq j_!(j^!1_X \widehat{\otimes}_U j^*) \simeq j_!j^!$, where the first is the projection formula and the second is because $j^! \simeq j^*$ is symmetric-monoidal. Due to the fact that $\operatorname{incl}_S \Gamma_S$ is an idempotent monad we see that $\operatorname{Cofib}(j_!j^!1_X \to 1_X)$ is an idempotent algebra object and $\operatorname{incl}_S \Gamma_S$ is given by tensoring with this algebra object. It follows formally from this that $\Gamma_S\operatorname{QCoh}(X)$ is symmetric monoidal with monoidal structure inherited from $\operatorname{QCoh}(X)$ and tensor-unit given by $\operatorname{Cofib}(j_!j^!1_X \to 1_X)$.\footnote{This can be regarded as a ``categorification" of the following. Let $R$ be a commutative ring and let $e \in R$ be an idempotent. Then $I := \operatorname{ker}(R \to R_e)$ is a ring, isomorphic to $(1-e)R$, with multiplication inherited from $R$ and unit $1-e$.} 

(iii): The only thing to check is that $j^*j_* \to \operatorname{id}$ is an equivalence. As noted above this can be deduced from the equivalence $j^! \xrightarrow[]{\sim} j^*$ and base-change.

(iv): The identity $\operatorname{L}_S j_! \simeq 0$ follows by passing to left adjoints in $j^! \operatorname{incl}_S \simeq 0$. The formula $j^!j_! \simeq \operatorname{id}$ can be deduced from the equivalence $j^! \xrightarrow[]{\sim} j^*$ and base-change.
\end{proof}
Let $X \in \mathsf{dRig}$. In the remainder of this section, we will produce examples of inclusions $j: U \to X$ of analytic subspaces satisfying $j^! \xrightarrow[]{\sim} j^*$, and also give some different formulas for the functors $\Gamma_S$ and $\operatorname{L}_S$. 

\begin{prop}\label{prop:localcohincreasing}
Let $X \in \mathsf{dRig}$ and let $S \subseteq |X|$ be a closed subset of the underlying topological space. Set $U := |X|\setminus S \subseteq |X|$. Suppose that we are given sequences $V_1 \supseteq V_2 \supseteq \dots \supseteq S $ and $U_1 \subseteq U_2 \subseteq \dots \subseteq U$ of open subsets of $|X|$ such that:
\begin{itemize}
    \item[$\star$] For every $n \geqslant 1$ the morphisms $V_n \to X$ and $U_n \to X$ are quasi-compact and one has $U_n \cup V_n = X$ and $U_n \cap V_{n+1} = \emptyset$.
    \item[$\star$] One has $\bigcup_{n \geqslant 1} U_n = U$. 
\end{itemize}
Then:
\begin{enumerate}[(i)]
    \item The morphism $j: U \to X$ belongs to the class $E$, i.e., it is $!$-able.
    \item Let $k_n: V_n \to X$ be the inclusions.
    Then:
    \begin{enumerate}
        \item There is an equivalence of functors $\operatorname{colim}_n k_{n,*}k_n^* \simeq \operatorname{Cofib}(j_! j^! \to  \operatorname{id})$.
        \item There is an equivalence of functors $\operatorname{lim}_n k_{n,*}k_n^! \simeq \operatorname{Fib}(\operatorname{id} \to j_*j^*)$.
    \end{enumerate}
    \item The canonical morphism $j^! \to j^*$ is an equivalence. 
\end{enumerate}
\end{prop}
Before the proof, we give an example.
\begin{example}\label{ex:zariskiopen}
Let $X = \operatorname{dSp}(A) \in \mathsf{dRig}$ be an affinoid and let $I \subseteq \pi_0A$ be an ideal. Choose generators $f_1,\dots,f_k$ for $I$. By using the homeomorphism $r: |\operatorname{Spa}(\pi_0A,(\pi_0A)^\circ)| \xrightarrow[]{\sim} |X|$ coming from Remark \ref{rmk:topolopgicalspaceremark}, we may define open subsets $U_n$ and $V_n$ of $|X|$ by:
\begin{equation}
\begin{aligned}
    r^{-1}V_n &:= \{ x \in |\operatorname{Spa}(\pi_0A,(\pi_0A)^\circ)|: |f_i|_x \leqslant  p^{-2n} \text{ for all }1 \leqslant  i \leqslant  k\},\\
    r^{-1}U_n &:= \{ x \in |\operatorname{Spa}(\pi_0A,(\pi_0A)^\circ)|: |f_i|_x \geqslant p^{-(2n+1)} \text{ for some }1 \leqslant  i \leqslant  k\}.
\end{aligned}
\end{equation}
Then: 
\begin{enumerate}[(i)]
    \item The sequences $\{V_n\}_{n \geqslant 1}$ and $\{U_n\}_{n \geqslant 1}$ satisfy the hypotheses of Proposition \ref{prop:localcohincreasing}, with 
    \begin{equation}
    \begin{aligned}
        r^{-1}S &= \{ x \in |\operatorname{Spa}(\pi_0A,(\pi_0A)^\circ)|: |f_i|_x = 0 \text{ for all }1 \leqslant  i \leqslant  k\},  \\
        r^{-1}U &= \{ x \in |\operatorname{Spa}(\pi_0A,(\pi_0A)^\circ)|: |f_i|_x > 0 \text{ for some }1 \leqslant  i \leqslant  k\}.
    \end{aligned}
    \end{equation}
    \item The sequence $\{V_n\}_{n \geqslant 1}$ is a cofinal system of open neighbourhoods of $S$ in $|X|$. Indeed, let $S \subseteq V^\prime \subseteq |X|$ be another open subset. Then $\{V^\prime\} \cup \{U_n\}_{n \geqslant 1}$ is a covering of the quasi-compact space $|X|$. Therefore $V^\prime \cup U_{m} = |X|$, for some $m \geqslant 1$. In particular $V_{m+1} \subseteq V^\prime$. 
\end{enumerate}

\end{example}
\begin{proof}[Proof of Proposition \ref{prop:localcohincreasing}]
(i): By Corollary \ref{cor:countableshriekdescent}, the morphism $\coprod_{n \geqslant 1} U_n \to U$ is of universal $!$-descent. Since the class $E \supseteq \mathrm{qcqs}$ is stable under disjoint unions and $!$-local on the source (c.f. Theorem \ref{thm:belongstoE}) we conclude that $j \in E$. 

(ii)(a): Let $j_n: U_n \to X$ be the inclusions. By Proposition \ref{prop:countableshriekdescent}, we know that $\operatorname{Cofib}(j_! j^! \to  \operatorname{id}) \simeq \operatorname{colim}_n \operatorname{Cofib}(j_{n,*}j_n^! \to \operatorname{id})$. By Lemma \ref{lem:prosystems}, we know that 
\begin{equation}
    \underset{n}{\operatorname{colim}} j_{n,*}j_n^! \simeq \underset{n}{\operatorname{colim}} \operatorname{Fib}(\operatorname{id} \to k_{n,*}k_n^*) \simeq \operatorname{Fib}(\operatorname{id} \to \underset{n}{\operatorname{colim}} k_{n,*}k_n^*),
\end{equation} 
where we used the property of stable $\infty$-categories. By using the property of stable $\infty$-categories this also implies that $ \operatorname{Cofib}(\operatorname{colim}_n j_{n,*}j_n^! \to \operatorname{id}) \simeq \operatorname{colim}_n k_{n,*}k_n^*$, proving (ii)(a). 

(ii)(b): Again let $j_n: U_n \to X$ be the inclusions. By descent we know that $\operatorname{Fib}(\operatorname{id} \to j_*j^*) \simeq 
 \operatorname{lim}_n \operatorname{Fib}(\operatorname{id} \to j_{n,*}j_n^*)$. By Lemma \ref{lem:prosystems2} we know that \begin{equation}
    \underset{n}{\operatorname{lim}} j_{n,*}j_n^* \simeq \underset{n}{\operatorname{lim}} \operatorname{Cofib}(k_{n,*}k_n^! \to \operatorname{id}) \simeq  \operatorname{Cofib}(\underset{n}{\operatorname{lim}} k_{n,*}k_n^! \to \operatorname{id}).
 \end{equation} By using the property of stable $\infty$-categories this implies that $\operatorname{Fib}(\operatorname{id} \to \underset{n}{\operatorname{lim}} j_{n,*}j_n^*) \simeq \underset{n}{\operatorname{lim}} k_{n,*} k_n^!$, proving (ii)(b).

(iii): Let $k^\prime_n : V_n^\prime :=  V_n \cap U \to U$ and $j_n^\prime : U_n \to U$ be the inclusions. By base-change, and since $j^*$ is exact and colimit-preserving we deduce that $\operatorname{colim}_n k_{n,*}^\prime k_n^{\prime,*} \simeq \operatorname{Cofib}(j^! \to j^*).$
Therefore, in order to prove the claim, it is enough to show that 
\begin{equation}\label{eq:knprime}
   \underset{n}{\operatorname{colim}} k_{n,*}^\prime k_n^{\prime,*} \simeq 0.
\end{equation} 
By descent, it is enough to check this after applying $j_m^*$ for each $m \geqslant 1$. The functors $j_m^*$ commute with colimits, and by base change, using that $U_m \cap V_{m+1} = \emptyset$, one has $j_m^* k_{m+1,*}^\prime \simeq 0$. Therefore \eqref{eq:knprime} is an equivalence and $j^! \simeq j^*$.   
\end{proof}
\subsection{Zariski-closed immersions}\label{subsec:Zariskiclosed}
\begin{defn}\label{defn:Zariskiclosed}
Let $f: X \to Y$ be a morphism in $\mathsf{dRig}$. We say that $f$ is a \emph{Zariski-closed immersion} if there exists a covering $\{U_i \to Y\}_{i \in \mathcal{I}}$ of $Y$ by affinoid subspaces $U_i = \operatorname{dSp}(A_i)$ such that, for each $i \in \mathcal{I}$, the pullback $X \times_Y U_i$ is represented by an affinoid $\operatorname{dSp}(B_i)$ and the induced morphism $A_i \to B_i$ is surjective on $\pi_0$. 
\end{defn}
\begin{lem}\label{lem:Zariskiclosedbasechange}
\begin{enumerate}[(i)]
    \item The class of Zariski-closed immersions is stable under composition and base-change. 
    \item Every Zariski-closed immersion is quasi-compact.
\end{enumerate}
\end{lem}
\begin{proof}
(i): Omitted. (ii): Follows from Lemma \ref{lem:qcproperties}(iii).
\end{proof}
\begin{lem}
Let $f: X \to Y$ be a Zariski-closed immersion in $\mathsf{dRig}$. The image of $|X|$ under $|f|$ is a closed subset of $|Y|$. 
\end{lem}
\begin{proof}
One reduces to the case when $Y = \operatorname{dSp}(A)$ is an affinoid and $X = \operatorname{dSp}(B) \to \operatorname{dSp}(A)$ is a morphism of affinoids such that $\pi_0A \to \pi_0B$ is surjective. By Remark \ref{rmk:topolopgicalspaceremark} it is then sufficient to show that the image of $\operatorname{Spa}(\pi_0 B, (\pi_0B)^\circ) \to \operatorname{Spa}(\pi_0A, (\pi_0A)^\circ)$ is closed. This is a consequence of \cite[1.4.1]{huber_etale_1996}.  
\end{proof}

\subsection{Zariski-open immersions}\label{subsec:Zariskiopen} 
\begin{defn}
A morphism $f: X \to Y$ in $\mathsf{dRig}$ is called a \emph{Zariski-open immersion} if there exists a Zariski-closed immersion $Z \to Y$ such that $f$ is equivalent to the inclusion of the analytic subspace $U \hookrightarrow Y$ corresponding to $|Y| \setminus |Z|  \subseteq |Y|$.
\end{defn}
\begin{prop}\label{prop:Zariskiopen}
Let $j: X \to Y$ be a Zariski-open immersion in $\mathsf{dRig}$. Then:
\begin{enumerate}[(i)]
    \item The morphism $j: X \to Y$ belongs to the class $E$, i.e., it is $!$-able.
    \item The canonical morphism $j^! \to j^*$ is an equivalence. 
\end{enumerate}
\end{prop}
\begin{proof}
(i): By using that $E$ is \emph{$*$-local on the target} one reduces to the case when $Y$ is an affinoid. Then the claim follows from Example \ref{ex:zariskiopen} and Proposition \ref{prop:localcohincreasing}(i). 

(ii): We make the following temporary definitions. A Zariski-closed immersion $T \to S$ is called \emph{basic} if $S = \operatorname{dSp}(A)$ and $T = \operatorname{dSp}(B)$ are both affinoid and $\pi_0 A \to \pi_0 B$ is surjective. A Zariski-open immersion $R \to S$ is called \emph{basic} if $S$ is affinoid and there exists a basic Zariski-closed immersion $T \to S$ such that $R$ corresponds to the complement of $|T|$ in $|S|$. Now, we proceed in steps.

\textit{Step 1}: When $X \to Y$ is a basic Zariski-open immersion, the statement follows from Example \ref{ex:zariskiopen} and Proposition \ref{prop:localcohincreasing}(iii). 

\textit{Step 2}: Now assume that $Y$ is an analytic subspace of an affinoid subspace $Y^\prime$, and $X \to Y$ is induced by a basic Zariski-open immersion $X^\prime \to Y^\prime$.  Choose a cover $\{U_i \to Y\}_{i \in \mathscr{I}}$ be a covering of $Y$ by affinoid subspaces of $Z$. Let $\mathcal{I}$ be the family of finite nonempty subsets of $\mathscr{I}$. For each $I \in \mathcal{I}$ set $U_I := \bigcap_{i \in I} U_i$ and let $t_I: U_I \to Y$ be the inclusions. Each $X \cap U_I \to U_I$ is a basic Zariski-open immersion. By descent, base-change and using that $j^!$ commutes with limits, one has
\begin{equation}
    j^! \simeq \underset{I \in \mathcal{I}}{\operatorname{lim}} j^!t_{I,*} t_I^* \simeq  \underset{I \in \mathcal{I}}{\operatorname{lim}} t_{I,*}^\prime j^{\prime,!} t_I^* \simeq \underset{I \in \mathcal{I}}{\operatorname{lim}} t_{I,*}^\prime j^{\prime,*} t_I^* \simeq \underset{I \in \mathcal{I}}{\operatorname{lim}} t_{I,*}^\prime t_I^{\prime,*} j^* \simeq j^*,
\end{equation}
where we used Step 1. 

\textit{Step 3}: Now $Y$ is arbitrary. By the Definition \ref{defn:Zariskiclosed} of a Zariski-closed immersion, we may choose a cover $\{U_i \to Y\}_{i \in \mathscr{I}}$ be a covering of $Y$ by affinoid subspaces such that each $U_i \cap Y \to U_i$ is a basic Zariski-open immersion. Let $\mathcal{I}$ be the family of finite nonempty subsets of $\mathscr{I}$. For each $I \in \mathcal{I}$ set $U_I := \bigcap_{i \in I} U_i$ and let $t_I: U_I \to Y$ be the inclusions. We proceed as in Step 2, using the result of Step 2, to conclude that $j^! \to j^*$ is an equivalence. 
\end{proof}
\subsection{Algebras of germs}\label{subsec:germs}
For the sake of brevity let us introduce the following notations:
\begin{defn}\label{defn:PStk}
We define $\mathsf{dAlg} := \mathsf{CAlg}(D_{\geqslant 0}(\mathsf{CBorn}_K))$ and $\mathsf{dAff} := \mathsf{dAlg}^\mathsf{op}$. The object of $\mathsf{dAff}$ corresponding to $A \in \mathsf{dAlg}$ is denoted by the formal expression $\operatorname{dSp}(A)$. For such we  define $\operatorname{QCoh}(\operatorname{dSp}(A)) := \operatorname{Mod}_AD(\mathsf{CBorn}_K)$. We define $\mathsf{PStk} := \operatorname{PSh}(\mathsf{dAff}) = \operatorname{Fun}(\mathsf{dAlg}, \infty\mathsf{Grpd})$.  
\end{defn}
\begin{defn}\label{defn:germ}
Let $X = \operatorname{dSp}(A) \in \mathsf{dAfnd}$ be a derived affinoid space and let $i: \operatorname{dSp}(B) = Z \to X$ be a Zariski-closed immersion defined by a morphism $A \to B$ which is surjective on $\pi_0$. The \emph{algebra of germs along $Z$} is defined to be 
\begin{equation}
    A^\dagger_Z := \underset{U \supseteq |Z|}{\operatorname{colim}} A_U
\end{equation}
where the colimit is taken in $\mathsf{CAlg}(D_{\geqslant 0}(\mathsf{CBorn}_K))$ and runs over all (affinoid) opens $U \supseteq |Z|$. We denote the corresponding object of $\mathsf{dAff}$ by $(Z \subseteq X)^\dagger$. 
\end{defn}
\begin{lem}\label{lem:germhomotopymono}
With notations as in Definition \ref{prop:Zariskiopen}. Let $\iota: (Z \subseteq X)^\dagger \to X$ be the canonical morphism. Then $\iota$ is a homotopy monomorphism. 
\end{lem}
\begin{proof}
We must show that $A \to A_Z^\dagger$ is a homotopy epimorphism. Since the class of homotopy epimorphisms is stable under colimits, this follows immediately from the fact that each $A \to A_U$ is a homotopy epimorphism.
\end{proof}
\begin{prop}\label{prop:germequivalence}
With notations as in Definition \ref{defn:germ}. There is a natural equivalence of $\infty$-categories 
\begin{equation}
    \operatorname{QCoh}((Z \subseteq X)^\dagger) \simeq \Gamma_Z\operatorname{QCoh}(X).
\end{equation}
\end{prop}
\begin{proof}
By Proposition \ref{prop:Zariskiopen} and Proposition \ref{prop:localcohformalproperties} there is an equivalence between $\Gamma_Z\operatorname{QCoh}(X)$ and algebra objects in $\operatorname{QCoh}(X)$ over the idempotent algebra object 
\begin{equation}
\operatorname{Cofib}(j_!j^! 1_X \to 1_X).
\end{equation} 
By Proposition \ref{prop:localcohincreasing}(ii) and Example \ref{ex:zariskiopen} we can identify this idempotent algebra object with $A_Z^\dagger$. Now the Proposition follows from the transitivity\footnote{For any symmetric monoidal $\infty$-category $(\mathscr{V}, \otimes)$ such that $\otimes$ is compatible with colimits separately in each variable, and any morphism $A \to B$ of commutative algebra objects in $\mathscr{V}$, there is an equivalence of $\infty$-categories $$\operatorname{Mod}_B (\operatorname{Mod}_A \mathscr{V}) \simeq \operatorname{Mod}_B \mathscr{V};$$
    on the left side here, we view $B$ as a commutative algebra object in $\operatorname{Mod}_A \mathscr{V}$. This is proved by applying Barr--Beck--Lurie to the forgetful functor $\operatorname{Mod}_A \mathscr{V} \to \operatorname{Mod}_B\mathscr{V}$.} property 
    \begin{equation}
        \operatorname{Mod}_{A_Z^\dagger} \operatorname{Mod}_A \operatorname{QCoh}(*) \simeq \operatorname{Mod}_{A_Z^\dagger}\operatorname{QCoh}(*).
    \end{equation}
\end{proof}
\begin{cor}\label{cor:recollementCor}
With notations as in Definition \ref{prop:Zariskiopen}. Let $\iota: (Z \subseteq X)^\dagger \to X$ be the canonical morphism. Let $U$ be the open subspace of $X$ corresponding to the complement of $|Z|$, and let $j: U \hookrightarrow X$ be the inclusion. In the sequence 
\begin{equation}
    \operatorname{QCoh}(U) \xrightarrow[]{j_!} \operatorname{QCoh}(X) \xrightarrow[]{\iota^*} \operatorname{QCoh}((Z \subseteq X)^\dagger)
\end{equation}
the composite $\iota^*j_! \simeq 0$ and the right adjoints $j^!$ and $\iota_*$ satisfy $j^!j_! \simeq \operatorname{id}$ and $\iota^*\iota_* \simeq \operatorname{id}$. Moreover the right adjoints $\iota_*$ and $j^!$ commute with colimits. 
\end{cor}
\begin{proof}
Combine Proposition \ref{prop:localcohformalproperties}, Proposition \ref{prop:Zariskiopen} and Proposition \ref{prop:germequivalence}. 
\end{proof}
\begin{defn}\label{defn:Pairs}
We define the category $\mathsf{Pairs}$ as the full subcategory of $\operatorname{Fun}(\Delta^1, \mathsf{dAfnd})$ on objects $Z=\operatorname{dSp}(B) \to \operatorname{dSp}(A)=X$ induced by a morphism $A \to B$ which is surjective on $\pi_0$.
\end{defn}
The category $\mathsf{Pairs}$ has fiber products:
\begin{equation}\label{eq:Germfiberproduct}
    (Z \to X) \times_{(Z^\prime \to X^\prime)}(Z^{\prime \prime} \to X^{\prime \prime}) = (Z\times_{Z^\prime} Z^{\prime \prime} \to X\times_{X^\prime}X^{\prime \prime}). 
\end{equation}
\begin{lem}\label{lem:Germfiberproduct}
The functor $\mathsf{Pairs} \to \mathsf{dAff} : (Z \to X) \mapsto (Z \subseteq X)^\dagger$ preserves fiber products.
\end{lem}
\begin{proof}
We use notations as in \eqref{eq:Germfiberproduct}. Say $X = \operatorname{dSp}(A)$ and $Z= \operatorname{dSp}(B)$, and similarly for $Z^\prime, X^\prime$, etc. Say $\pi_0 A \to \pi_0B$ is defined by an ideal $I = (f_1,\dots,f_r)$ and $\pi_0 A^{\prime \prime} \to \pi_0B^{\prime \prime}$ is defined by $I^\prime = (f_1^\prime,\dots,f_s^\prime)$. Then $\pi_0(A \otimes_{A^\prime} A^{\prime \prime}) \to \pi_0(B \otimes_{B^\prime} B^{\prime \prime})$ is defined by 
\begin{equation}
    (f_1\otimes1, \dots, f_r\otimes 1, 1 \otimes f_1^{\prime \prime} ,\dots, 1 \otimes f_s^{\prime \prime}). 
\end{equation}
Using the homeomorphisms $r: |\operatorname{Spa}(\pi_0A, (\pi_0A)^\circ)| \xrightarrow[]{\sim} |X|$, etc., we may define (rational) open subsets $V_n$ of $X$ as in Example \ref{ex:zariskiopen}:
\begin{equation}
    r^{-1} V_n := \{ x \in |\operatorname{Spa}(\pi_0A,(\pi_0A)^\circ)| : |f_i|_x \leqslant  p^{-n} \text{ for all }1\leqslant  i \leqslant  k\}. 
\end{equation}
and similarly for $V_n^{\prime \prime}$. Then by cofinality of this system we find that the pushout of the algebras of germs is
\begin{equation}
\underset{n}{\operatorname{colim}} A_{V_n} \widehat{\otimes}^\mathbf{L}_{A^\prime} A^{\prime \prime}_{V_n^{\prime\prime}}  = \underset{n}{\operatorname{colim}}(A \widehat{\otimes}^\mathbf{L}_{A^\prime} A^{\prime \prime})_{V_n \times_{X^\prime} V_n^{\prime \prime}}
\end{equation}
and we note that $V_n \times_X V_n^{\prime \prime}$ is the open subset with 
\begin{equation}
    r^{-1}(V_n \times_X V_n^{\prime \prime}) = \{ x : |f_i \otimes 1| \leqslant p^{-n}, |1 \otimes f_j^{\prime \prime} | \leqslant  p^{-n} \text{ for all } i, j\}, 
\end{equation}
which is again a cofinal system of neighbourhoods of $Z \times_{Z^\prime} Z^{\prime \prime}$. 
\end{proof}
\begin{lem}\label{lem:pairsdescent}
Let $(Z \to X) \in \mathsf{Pairs}$. Let affinoid opens $\{U_i\}_{i =1}^n$ of $X$ be given such that $|Z| \subseteq \bigcup_{i} U_i$. Then: 
\begin{enumerate}[(i)]
\item Let $\mathcal{I}$ be the family of finite nonempty subsets of $\{1,\dots,n\}$. The natural morphism   
\begin{equation}\label{eq:pairsdescent}
    \operatorname{QCoh}^*((Z\subseteq X)^\dagger) \xrightarrow[]{} \underset{I \in \mathcal{I}}{\operatorname{lim}}\operatorname{QCoh}^*((Z_I \subseteq X_I)^\dagger),
\end{equation}
is an equivalence; here $Z_I := U_I \times_X Z$. 
\item Let $Y := \coprod_{i} U_i$ and $S := \coprod_{i} Z_i \to Y$. Then, the canonical morphism
\begin{equation}\label{eq:pairsdescent2}
    \operatorname{QCoh}((Z \subseteq X)^\dagger) \to \underset{[m] \in \Delta^{\mathsf{op}}}{\operatorname{lim}} \operatorname{QCoh}((S \subseteq Y)^{\dagger, m+1/(Z \subseteq X)^\dagger}) 
\end{equation}
is an equivalence. 
\end{enumerate}
\end{lem}
\begin{proof}
Let us say $X = \operatorname{dSp}(A)$. According to \cite[Lemma 2.3]{KisinLocal99} and Lemma \ref{lem:derivedrational}, the system of rational opens of $X$ containing $|Z|$, is a cofinal system of open neighbourhoods of $|Z|$ in $X$. Hence we may find a rational open subset $V$ with $|Z| \subseteq V \subseteq \bigcup_{i} U_i$. For such $V$ one then has 
\begin{equation}
    A_V \xrightarrow[]{\sim} \underset{(i_1,\dots,i_k) \in \mathcal{I}}{\operatorname{lim}} A_{V_{i_1}}\widehat{\otimes}^\mathbf{L}_A \dots \widehat{\otimes}^\mathbf{L}_A A_{V_{i_k}},
\end{equation} where $V_i = V \cap U_i$, c.f. the proof of Lemma \ref{lem:descentAfnd} and especially \eqref{eq:Bfinite}. By compact generation (Proposition \ref{prop:compactDA}), filtered colimits commute with finite limits (Lemma \ref{lem:kappafiltered}), and hence we obtain
\begin{equation}
    A_Z^\dagger \xrightarrow[]{\sim} \underset{(i_1,\dots,i_k) \in \mathcal{I}}{\operatorname{lim}} A_{Z_{i_1}}^\dagger\widehat{\otimes}^\mathbf{L}_A \dots \widehat{\otimes}^\mathbf{L}_A A_{Z_{i_k}}^\dagger,
\end{equation}
which by a standard argument implies (i). This also shows that $A_Z^\dagger \to \prod_i A_{Z_i}^\dagger$ is descendable, which gives (ii). 
\end{proof}
\subsection{Six-functor formalism for prestacks}\label{sec:sixfPStk}
In the next section we will work in a ``bigger" six-functor formalism constructed as follows. We use notations as in Definition \ref{defn:PStk}. We will apply the formalism of \S\ref{subsec:sixf3} in the following set-up (with notations as in that section): 
\begin{itemize}
    \item[$\star$] We take $\mathscr{V} := D(\mathsf{CBorn}_K)$, so that $\mathscr{E}:= \mathsf{CAlg}(D(\mathsf{CBorn}_K))^\mathsf{op}$ and we consider $\mathsf{dAff} = \mathsf{CAlg}(D_{\geqslant 0}(\mathsf{CBorn}_K))^\mathsf{op} \subseteq \mathscr{E}$. We take $\tau$ to be the trivial topology on $\mathsf{dAff}$. We define $\mathsf{PStk} := \operatorname{Psh}(\mathsf{dAff})$. 
\end{itemize}
It is clear that the assumptions of \S3.2 are satisfied and hence we obtain:
\begin{thm}\label{thm:6FPStk}
The functor $\operatorname{QCoh}$ extends to a six-functor formalism
\begin{equation}
    \operatorname{QCoh}: \operatorname{Corr}(\mathsf{PStk}, \widetilde{E})^\otimes \to \mathsf{Pr}^{L, \otimes}_\mathsf{st}
\end{equation}
such that the class $\widetilde{E} \supseteq \mathrm{rep}$ is stable under disjoint unions, $*$-local on the target, $!$-local on the source, is tame and satisfies $\widetilde{E} \subseteq \delta \widetilde{E}$. Further, every morphism $f \in \mathrm{rep}$ satisfies $f_! \simeq f_*$. 
\end{thm}
\subsection{Internal groupoid objects in an $\infty$-category}
\begin{defn}\label{defn:groupoidobject}
Let $\mathscr{C}$ be an $\infty$-category admitting all fiber products.
\begin{enumerate}[(i)]
    \item A \emph{groupoid object} in $\mathscr{C}$ is a simplicial object $X \in s\mathscr{C} := \operatorname{Fun}(N(\Delta^{\mathsf{op}}) ,\mathscr{C})$ such that, for every $n \geqslant 0$ and for all subsets $S, S^\prime \subseteq [n]$ with $S \cup S^\prime = [n]$ and $|S\cap S^\prime| = 1$, the diagram 
    \begin{equation}
\begin{tikzcd}[cramped]
	{X([n])} & {X(S)} \\
	{X(S^\prime)} & {X(S\cap S^\prime)}
	\arrow[from=1-1, to=1-2]
	\arrow[from=1-1, to=2-1]
	\arrow[from=1-2, to=2-2]
	\arrow[from=2-1, to=2-2]
\end{tikzcd}
    \end{equation}
    is Cartesian.
    \item Assume that $\mathscr{C}$ admits geometric realizations. Let $X_\bullet$ be a groupoid object in $\mathscr{C}$. We obtain from $X_{\bullet}$ an augmented simplicial object of $\mathscr{C}$ by setting $X_{-1} := \operatorname{colim}_{[n] \in \Delta^\mathsf{op}} X_n$. We say that $X_\bullet$ is \emph{effective} if the canonical morphism $X_\bullet \to N(X_0 \to X_{-1})$ is an equivalence in  $s\mathscr{C}$. 
\end{enumerate}
\end{defn}
\begin{rmk}\begin{enumerate}[(i)]
    \item Due to Definition \ref{defn:derivedrational}(i),  Definition \ref{defn:groupoidobject}(ii) is equivalent saying that $X_0 \to X_{-1}$ is an effective epimorphism. 
    \item If $\mathscr{C}$ is an $\infty$-topos then every groupoid object in $\mathscr{C}$ is effective. Indeed, this is one of the Giraud-Rezk-Lurie axioms \cite[Theorem 6.1.0.6]{HigherToposTheory}.
\end{enumerate}

\end{rmk}
\begin{defn}
Let $\mathscr{C}$ be an $\infty$-category admitting all fiber products. We say that a morphism $X_\bullet \to Y_\bullet$ in $s\mathscr{C}$ is a \emph{homotopy Kan fibration} if for all $n \geqslant 1$ and for all $0 \leqslant  k \leqslant  n$, the canonical morphism 
\begin{equation}\label{eq:homotopyKandefinition}
    X(\Delta^n) \to X(\Lambda^n_k) \times_{Y(\Lambda^n_k)} Y(\Delta^n)
\end{equation}
is an effective epimorphism.
\end{defn}
\begin{rmk}
Let $\mathscr{C}$ be an $\infty$-category admitting all fiber products. It is immediate that the class of homotopy Kan fibrations in $s\mathscr{C}$ is stable under composition and base-change.
\end{rmk}
\begin{lem}\label{lem:groupoidhomotopyKan}
Let $\mathscr{C}$ be an $\infty$-category admitting all finite limits and geometric realizations. Let $f: X_\bullet \to Y_\bullet$ be a morphism between groupoid objects of $\mathscr{C}$. Then, the following are equivalent. 
\begin{enumerate}[(i)]
    \item $f$ is a homotopy Kan fibration,
    \item For $i= 0,1$, the morphisms 
    \begin{equation}
        X(\Delta^1) \to X(\Lambda^1_i) \times_{Y(\Lambda^1_i)} Y(\Delta^1), 
    \end{equation}
    are effective epimorphisms.
\end{enumerate}
\end{lem}
\begin{proof}
(i) $\implies$ (ii): This is trivial.

(ii) $\implies$ (i): Due to \cite[Proposition 6.1.2.6(3)]{HigherToposTheory}, for every $n \geqslant 2$ and $0 \leqslant  k \leqslant  n$ the canonical morphism $X(\Delta^n) \to X(\Lambda^n_k)$ is an equivalence. More precisely the result of \emph{loc. cit.} says that the morphism $\mathscr{C}_{/X(\Delta^n)} \to \mathscr{C}_{/X(\Lambda^n_k)}$, induced by postcomposition, is an equivalence of $\infty$-categories. In particular, $[X(\Delta^n) \to X(\Lambda^n_k)]$ is terminal in $\mathscr{C}_{/X(\Lambda^n_k)}$, so we can deduce the claim from uniqueness of the terminal object. Obviously, the same is true for $Y$. Therefore, the condition \eqref{eq:homotopyKandefinition} in the definition of a homotopy Kan fibration, is automatically satisfied for every $n \geqslant 2$. 
\end{proof}
\begin{prop}\label{prop:homotopyKanfiberproduct}
Let $\mathscr{X}$ be an $\infty$-topos. Let $X_\bullet \to Y_\bullet$ and $Z_\bullet \to Y_\bullet$ be morphisms in $s\mathscr{X}$. Assume that $X_\bullet \to Y_\bullet$ is a homotopy Kan fibration. Then, the canonical morphism 
\begin{equation}
    |X_\bullet \times_{Y_\bullet} Z_{\bullet}| \to  |X_\bullet| \times_{|Y_\bullet|} |Z_{\bullet}|
\end{equation}
is an equivalence. Here we have abbreviated $|X_\bullet| := \operatorname{colim}_{[n] \in \Delta^\mathsf{op}} X_n$, etc. 
\end{prop}
\begin{proof}
The case when $\mathscr{X} = \infty\mathsf{Grpd}$ is proven in \cite[Corollary 6.7]{mazel-gee_model_2015}. The case when  $\mathscr{X} = \operatorname{Psh}(\mathscr{D}, \infty \mathsf{Grpd})$ for some $\infty$-category $\mathscr{D}$, follows from the case when $\mathscr{C}= \infty\mathsf{Grpd}$, because limits and colimits in $\operatorname{Psh}(\mathscr{D}, \infty \mathsf{Grpd})$ are computed pointwise, c.f. \cite[\S 5.1.2]{HigherToposTheory}. Now if $\mathscr{X}$ is an $\infty$-topos we may write $\mathscr{X}$ as a localization $L: \operatorname{Psh}(\mathscr{D}, \infty\mathsf{Grpd}) \leftrightarrows \mathscr{X} :i $ for some $\mathscr{D}$, where $L$ is left exact. We recall that colimits in $\mathscr{X}$ are computed by first taking the colimit in $\operatorname{Psh}(\mathscr{D}, \infty\mathsf{Grpd})$ and then applying $L$, whence the claim follows.  
\end{proof}
Let $\mathscr{X}$ be an $\infty$-topos. Motivated by \cite{mazel-gee_model_2015}, we define two classes of morphisms (which we call \emph{weak equivalences} and \emph{fibrations}) in $s\mathscr{X}$ as follows:
\begin{enumerate}
  \item[(W)] The weak equivalences are precisely the morphisms in $s\mathscr{X}$ which are sent to equivalences under the geometric realization functor $|\cdot| : s\mathscr{X} \to \mathscr{X}$.
  \item[(F)] The fibrations are the homotopy Kan fibrations.
\end{enumerate}
For morphisms $X_\bullet \to Y_{\bullet}$ and $Z_\bullet \to Y_\bullet$ in $s\mathscr{X}$, Proposition \ref{prop:homotopyKanfiberproduct} above gives us a strategy to calculate the fiber product $|X_\bullet| \times_{|Y_\bullet|} |Z_{\bullet}|$ as follows. One finds a factorization of the morphism $X_\bullet \to Y_\bullet$ as $X_\bullet \to X^\prime_\bullet \to Y_\bullet$, where the first morphism is a weak equivalence and the second is a homotopy Kan fibration. Then, by Proposition \ref{prop:homotopyKanfiberproduct}, one has
\begin{equation}
    |X_\bullet| \times_{|Y_\bullet|} |Z_{\bullet}| \simeq |X^\prime_\bullet| \times_{|Y_\bullet|} |Z_{\bullet}| \simeq |X^\prime_\bullet \times_{Y_\bullet} Z_\bullet|.
\end{equation}
An example of this is the following. We may define an endofunctor $-\oplus[0]$ of the simplex category $\Delta$, where $\oplus$ denotes the operation of \emph{ordinal sum}. By precomposition, we then obtain an endofunctor $\operatorname{Dec}_0: s\mathscr{X} \to s\mathscr{X}$, called \emph{d\'ecalage}. For $X_\bullet \in s\mathscr{X}$, one has $(\operatorname{Dec}_0X_\bullet)_n = X_{n+1}$. 
\begin{lem}\label{lem:decalage}
Let $\mathscr{X}$ be an $\infty$-topos and let $X_\bullet$ be a groupoid object in $\mathscr{X}$. The morphism $X_0 \to X_\bullet$, where the former is viewed as a constant simplicial object, can be factored as 
\begin{equation}
    X_0 \to \operatorname{Dec}_0X_{\bullet} \to X_\bullet, 
\end{equation}
where the first morphism is a weak equivalence and the second is a homotopy Kan fibration. In particular, for any morphism $Y_\bullet \to X_\bullet$ in $s\mathscr{X}$ the fiber product $X_0 \times_{|X_\bullet|} |Y_{\bullet}|$ can be computed as $|(\operatorname{Dec}_0X_\bullet) \times_{X_\bullet} Y_{\bullet}|$. 
\end{lem}
\begin{proof}
That $X_0 \to \operatorname{Dec}_0X_{\bullet}$ is a weak equivalence is an immediate consequence of \cite[Lemma 6.1.3.16]{HigherToposTheory}. To check that $\operatorname{Dec}_0X_\bullet \to X_\bullet$ is a homotopy Kan fibration, by Lemma \ref{lem:groupoidhomotopyKan} it suffices to show that $ X_2 \to X_1 \times_{X_0} X_1$ is an effective epimorphism, but this is even an equivalence (by assumption, $X_\bullet$ is a groupoid object). 
\end{proof}
\subsection{Formal theory of stratifications}\label{sec:stratification}
Let $f: X \to Y$ be a morphism in $\mathsf{dAfnd}$. Then, for every $n \geqslant 1$, the diagonal morphism $\Delta_{n,f}: X \to X^{n/Y }$ determines an object of the category $\mathsf{Pairs}$ (Definition \ref{defn:Pairs}). Hence we may consider the germ $(X \subseteq X^{n/Y})^\dagger$. In particular, the following definition makes sense.  
\begin{defn}
Let $f: X \to Y$ be a morphism in $\mathsf{dAfnd}$. 
\begin{enumerate}[(i)]
    \item The \emph{infinitesimal groupoid} of $f$, denoted $\operatorname{Inf}(X/Y)$, is the simplicial object of $\mathsf{PStk}$ with 
    \begin{equation}
        \operatorname{Inf}(X/Y)_n := (X \subseteq X^{n+1/Y})^\dagger.
    \end{equation}
    \item The \emph{stratifying stack} of $f$, denoted $(X/Y)_{\mathrm{str}}$, is defined to be  the object of $\mathsf{PStk}$ given by 
    \begin{equation}
        (X/Y)_{\mathrm{str}} := \underset{[n] \in \Delta^{\mathsf{op}}}{\operatorname{colim}} \operatorname{Inf}(X/Y)_n.
    \end{equation}
    \item Let $X \in\mathsf{dAfnd}$. We define the \emph{infinitesimal groupoid of }$X$ to be $\operatorname{Inf}(X) := \operatorname{Inf}(X/*)$ and the \emph{stratifying stack} of $X$ to be $X_{\mathrm{str}} := (X/*)_{\mathrm{str}}$.
\end{enumerate}
\end{defn}
\begin{rmk}\label{rmk:Infremark}\begin{enumerate}[(i)]
    \item Let $f: X \to Y$ be a morphism in $\mathsf{dAfnd}$. Then $\operatorname{Inf}(X/Y)$ is a groupoid object of $\mathsf{dAff}$ in the sense of Definition \ref{defn:groupoidobject}. This follows from Lemma \ref{lem:Germfiberproduct}.
    \item For a fixed object $Y \in \mathsf{dAfnd}$, these constructions define functors 
    \begin{equation*}
    \begin{aligned}
        \operatorname{Inf}(-/Y): \mathsf{dAfnd}_{/Y} \to s\mathsf{dAfnd}_{/Y} && \text{and} && (-/Y)_{\mathrm{str}} : \mathsf{dAfnd}_{/Y} \to \mathsf{PStk}_{/Y}.
    \end{aligned}
    \end{equation*} 
    \item Let $k: X \to Z$, $h: Y \to Z$ be morphisms of $\mathsf{dAfnd}$ and let $f: X \to Y$ be a morphism over $Z$. Due to Lemma \ref{lem:groupoidhomotopyKan}, the following are equivalent:
    \begin{enumerate}
        \item $\operatorname{Inf}(X/Z) \to \operatorname{Inf}(Y/Z)$ is a homotopy Kan fibration,
        \item The morphism $(X \subseteq X \times_Z X)^\dagger \to (X \subseteq X \times_Z Y)^\dagger$ induced by $(\operatorname{id},f)$, is an effective epimorphism in $\mathsf{PStk}$. Here the latter morphism is the graph of $f$. 
    \end{enumerate}
    \end{enumerate}
\end{rmk}
\begin{example}\label{ex:Infexample}
\begin{enumerate}[(i)]
    \item Let $X \in \mathsf{dAfnd}$ and let $X \to \operatorname{dSp}(K)$ be the structure morphism. Then, by Remark \ref{rmk:Infremark}(iii), the morphism $\operatorname{Inf}(X) \to \operatorname{Inf}(\operatorname{dSp}(K)) \simeq \operatorname{dSp}(K)$, is a homotopy Kan fibration. 
    \item Let $X \in \mathsf{dAfnd}$ and let $f: U \to X$ be the inclusion of an affinoid open subspace. Then $(U \subseteq U \times U)^\dagger \to (U \subseteq U \times X)^\dagger$ is an equivalence in $\mathsf{dAfnd}$ because $U \times U$ is a neighbourhood of $U$ (embedded via the graph) in $U \times X$. Hence, by Remark \ref{rmk:Infremark}(iii), $\operatorname{Inf}(U) \to \operatorname{Inf}(X)$ is a homotopy Kan fibration. 
    \item If $U \to X$ is any morphism in $\mathsf{dAfnd}$ such that $U \to U \times_X U$ is an open immersion, then the canonical morphism $U \to (U/X)_{\mathrm{str}}$ is an equivalence. This is because, for each $n \geqslant 1$, there is an equivalence $U \simeq (U \subseteq U^{n+1/X})^\dagger$, by construction of the germ. In particular, this includes the case of (ii) above. 
\end{enumerate}
\end{example}
\begin{lem}\label{lem:InfhomotopyKan}
Let $f: X \to Y$ and $g: Y^\prime \to Y$ be morphisms in $\mathsf{dAfnd}$. Set $X^\prime := X \times_Y Y^\prime$. Assume that the induced morphism $\operatorname{Inf}(X) \to \operatorname{Inf}(Y)$ is a homotopy Kan fibration. Then:
\begin{enumerate}[(i)]
    \item The canonical morphism $(X/Y)_{\mathrm{str}} \to X_{\mathrm{str}} \times_{Y_{\mathrm{str}}} Y$ is an equivalence.
    \item The canonical morphism $X^\prime_{\mathrm{str}} \to X_{\mathrm{str}} \times_{Y_{\mathrm{str}}} Y^\prime_{\mathrm{str}}$ is an equivalence.
    \item The morphism $\operatorname{Inf}(X^\prime) \to \operatorname{Inf}(Y^\prime)$ is also a homotopy Kan fibration and the canonical morphism $(X^\prime/Y^\prime)_{\mathrm{str}} \to (X/Y)_{\mathrm{str}} \times_Y Y^\prime$ is an equivalence. 
    \item Assume that there are morphisms $h: Y \to Z$ and $k: X \to Z$ in $\mathsf{dAfnd}$ making $f$ into a morphism over $Z$. Then, the induced morphism $\operatorname{Inf}(X/Z) \to \operatorname{Inf}(Y/Z)$ is a homotopy Kan fibration and the natural morphism $(X/Y)_{\mathrm{str}} \to (X/Z)_{\mathrm{str}} \times_{(Y/Z)_{\mathrm{str}}} Y $ is an equivalence. 
\end{enumerate}
\end{lem}
\begin{proof}
(i): We regard $Y$ as a constant simplicial object, equipped with a morphism $Y \to \operatorname{Inf}(Y)$ induced by the diagonal morphisms
\begin{equation}
    Y \to (Y \subseteq Y^{n+1})^\dagger = \operatorname{Inf}(Y)_n.
\end{equation}
In particular, we calculate
\begin{equation}
    Y \times_{\operatorname{Inf}(Y)_n} \operatorname{Inf}(X)_n \simeq (X \subseteq X^{n+1/Y})^\dagger = \operatorname{Inf}(X/Y)_n.
\end{equation}
Hence, by Proposition \ref{prop:homotopyKanfiberproduct}, we see that $(X/Y)_{\mathrm{str}} \xrightarrow[]{\sim} X_{\mathrm{str}} \times_{Y_{\mathrm{str}}} Y$.

(ii): We calculate
\begin{equation}\label{eq:InfXprimecalculation}
    \operatorname{Inf}(X)_n \times_{\operatorname{Inf}(Y)_n} \operatorname{Inf}(Y^\prime)_n \simeq (X \times_Y Y^\prime \subseteq (X \times_Y Y^\prime)^{n+1})^\dagger \simeq \operatorname{Inf}(X^\prime)_n, 
\end{equation}
and so by Proposition \ref{prop:homotopyKanfiberproduct} again we conclude that $X^\prime_{\mathrm{str}} \xrightarrow[]{\sim} X_{\mathrm{str}} \times_{Y_{\mathrm{str}}} Y^\prime_{\mathrm{str}}$. 

(iii): By the calculation \eqref{eq:InfXprimecalculation}, and the fact that homotopy Kan fibrations are stable under base-change, we see that $\operatorname{Inf}(X^\prime) \to \operatorname{Inf}(Y^\prime)$ is a homotopy Kan fibration. We calculate
\begin{equation}
\begin{aligned}
    Y^\prime \times_Y \operatorname{Inf}(X/Y)_n &\simeq Y^\prime \times_Y (X \subseteq X^{n+1/Y})^\dagger \\ &\simeq (X^\prime \subseteq X^{\prime, n+1/Y^\prime})^\dagger \\ &= \operatorname{Inf}(X^\prime/Y^\prime)_n.  
\end{aligned}
\end{equation}
The morphism $\operatorname{Inf}(X/Y) \to Y$ is always a homotopy Kan fibration, by Lemma \ref{lem:groupoidhomotopyKan}. Therefore by Proposition \ref{prop:homotopyKanfiberproduct} we conclude that $(X^\prime/Y^\prime)_{\mathrm{str}} \xrightarrow[]{\sim} (X/Y)_{\mathrm{str}} \times_Y Y^\prime$. 

(iv): 
We note that $\operatorname{Inf}(X/Z) \to \operatorname{Inf}(Y/Z)$ is the pullback of $\operatorname{Inf}(X) \to \operatorname{Inf}(Y)$ along $\operatorname{Inf}(Y/Z) \to \operatorname{Inf}(Y)$, and is therefore a homotopy Kan fibration. For each $n \geqslant 0$ the following square is also Cartesian:
\begin{equation}
\begin{tikzcd}
	{(X \subseteq X^{n+1/Y})^\dagger} & {Y} \\
	{(X \subseteq X^{n+1/Z})^\dagger} & {(Y \subseteq Y^{n+1/Z})^\dagger}
	\arrow[from=1-1, to=1-2]
	\arrow[from=1-1, to=2-1]
	\arrow[from=1-2, to=2-2]
	\arrow[from=2-1, to=2-2]
\end{tikzcd}
\end{equation}
so that $\operatorname{Inf}(X/Y) \xrightarrow[]{\sim} \operatorname{Inf}(X/Z) \times_{\operatorname{Inf}(Y/Z)} Y$. Therefore, by Proposition \ref{prop:homotopyKanfiberproduct} we conclude that $(X/Y)_{\mathrm{str}} \xrightarrow[]{\sim} (X/Z)_{\mathrm{str}} \times_{(Y/Z)_{\mathrm{str}}} Y$. 
\end{proof}
\begin{lem}\label{lem:pX/Yshriekable}
Let $f: X \to Y$ be a morphism in $\mathsf{dAfnd}$. The canonical morphism $X \to (X/Y)_{\mathrm{str}}$ belongs to the class $\mathrm{rep}$ of representable morphisms in $\mathsf{PStk}$. 
\end{lem}
\begin{proof}
Let us take $Z \in \mathsf{dAff} \subseteq \mathsf{PStk}$ with a morphism $Z \to (X/Y)_{\mathrm{str}}$; we need to show that the pullback $Z \times_{(X/Y)_{\mathrm{str}}} X$ is representable. In any category of presheaves, all representable objects are projective. In particular the map $Z \to (X/Y)_{\mathrm{str}}$ has the right lifting property against effective epimorphisms, so there exists a lift $Z \to X$, up to homotopy. Using this lift, and the associativity of fiber products, we deduce that
\begin{equation}
    Z \times_{(X/Y)_{\mathrm{str}}} X \simeq  Z \times_X X \times_{(X/Y)_{\mathrm{str}}} X \simeq Z \times_X (X \subseteq X \times_Y X)^\dagger, 
\end{equation}
where we used that groupoid objects are effective. Since $\mathsf{dAff}$ is closed under fiber products, we deduce that $X \to (X/Y)_{\mathrm{str}}$ belongs to $\mathrm{rep}$. 
\end{proof}
\begin{defn}\label{defn:goodmorphism}
A morphism $f: X \to Y$ in $\mathsf{dAfnd}$ is called \emph{good} if: 
\begin{itemize}
    \item[$\star$] The morphism $\operatorname{Inf}(X) \to \operatorname{Inf}(Y)$ is a homotopy Kan fibration,
    \item[$\star$] The morphism $X \to (X/Y)_{\mathrm{str}}$ is of universal $!$-descent, with respect to the six-functor formalism on $\mathsf{PStk}$ (Theorem \ref{thm:6FPStk}). This condition makes sense by Lemma \ref{lem:pX/Yshriekable}. 
\end{itemize}
\end{defn}
\begin{lem}\label{lem:goodmorphismproperties}
\begin{enumerate}[(i)]
    \item The class of good morphisms in $\mathsf{dAfnd}$ is stable under base-change and composition. The functor $(-)_{\mathrm{str}} : \mathsf{dAfnd} \to \mathsf{PStk}$ preserves finite products, and pullbacks of edges in the class $\mathrm{good}$. 
    \item If $f: X \to Y$ is good then $f_{\mathrm{str}} : X_{\mathrm{str}} \to Y_{\mathrm{str}}$ belongs to the class $E$ of $!$-able morphisms in the six-functor formalism on $\mathsf{PStk}$ (Theorem \ref{thm:6FPStk}). 
\end{enumerate}
\end{lem}
Before proving the Lemma, we make note of an immediate Corollary.
\begin{cor}\label{cor:dRcorrespondence}
The functor $(-)_{\mathrm{str}}$ induces a symmetric-monoidal functor
\begin{equation}
        (-)_{\mathrm{str}}: \operatorname{Corr}(\mathsf{dAfnd}, \mathrm{good}) \to \operatorname{Corr}(\mathsf{PStk}, \widetilde{E}). 
\end{equation}
where $\widetilde{E}$ is the class of edges in $\mathsf{PStk}$ coming from Theorem \ref{thm:6FPStk}. 
\end{cor}
\begin{proof}[Proof of Lemma \ref{lem:goodmorphismproperties}]
(i): All of these properties follow from Lemma \ref{lem:InfhomotopyKan}.

(ii): Since the class $\widetilde{E}$ is \emph{$*$-local on the target}, it suffices to check that $f_{\mathrm{str}} \in \widetilde{E}$ after pullback along a morphism $Z \to Y_{\mathrm{str}}$ from a representable object of $\mathsf{PStk}$. Again, since representable objects are projective, this lifts to a morphism $Z \to Y$, up to homotopy. Using this morphism and the associativity of fiber products we see that 
\begin{equation}
    Z \times_{Y_{\mathrm{str}}} X_{\mathrm{str}} \simeq Z \times_Y Y \times_{Y_{\mathrm{str}}} X_{\mathrm{str}} \simeq  Z \times_Y (X/Y)_{\mathrm{str}},
\end{equation}
where we used Lemma \ref{lem:InfhomotopyKan}(i). Therefore, since the class $\widetilde{E}$ is stable under base-change, it suffices to show that $(X/Y)_{\mathrm{str}} \to Y$ belongs to $\widetilde{E}$. The morphism $X \to Y$ factors as $X \to (X/Y)_{\mathrm{str}} \to Y$, where the first morphism is of universal $!$-descent by assumption. Therefore, since the class $\widetilde{E}$ is \emph{$!$-local on the source}, we deduce that $(X/Y)_{\mathrm{str}} \to Y \in \widetilde{E}$. 
\end{proof}
By Corollary \ref{cor:dRcorrespondence} we now have a symmetric-monoidal functor
\begin{equation}
 (-)_{\mathrm{str}}: \operatorname{Corr}(\mathsf{dAfnd}, \mathrm{good}) \to \operatorname{Corr}(\mathsf{PStk},\widetilde{E}). 
\end{equation}
By post-composing $(-)_{\mathrm{str}}$ with the six-functor formalism $\operatorname{QCoh}$ on $ (\mathsf{PStk},\widetilde{E})$, we obtain a six-functor formalism 
\begin{equation}\label{eq:CrysSixfunctors}
    \operatorname{Strat} := \operatorname{QCoh} \circ (-)_{\mathrm{str}} :  \operatorname{Corr}(\mathsf{dAfnd}, \mathrm{good}) \to \mathsf{CAlg}(\mathsf{Pr}^L_{\mathsf{st}}).
\end{equation}
We would like to extend this six-functor formalism to all objects of $\mathsf{dRig}$ and also to a much larger class than just the $\mathrm{good}$ morphisms. Unfortunately, the class of good morphisms is not closed under the formation of diagonals, so we cannot apply the extension formalism of \S\ref{subsec:sixf2} and we have to proceed in a more ad-hoc manner. 
\begin{defn}
We define $E_{\mathrm{str}}$ to be the class of morphisms $f: X \to Y$ in $\mathsf{dRig}$ which are representable in $\mathrm{good}$. That is, for any morphism $Y^\prime \to Y$ from an object of $\mathsf{dAfnd}$, the pullback $f^\prime : X^\prime \to Y^\prime$ is a morphism between objects of $\mathsf{dAfnd}$ which belongs to the class $\mathrm{good}$. 
\end{defn}
Therefore, by \cite[Proposition A.5.16]{mann_p-adic_2022} again, the six-functor formalism of \eqref{eq:CrysSixfunctors} extends to a six-functor formalism
\begin{equation}
    \operatorname{Strat} : \operatorname{Corr}(\mathsf{dRig}, E_{\mathrm{str}}) \to \mathsf{CAlg}(\mathsf{Pr}^L_{\mathsf{st}}),
\end{equation}
uniquely such that 
\begin{equation}
    \operatorname{Strat}^*(X) \xrightarrow[]{\sim} \underset{Y \in \mathsf{dAfnd}^\mathsf{op}_{/X}}{\operatorname{lim}}\operatorname{Strat}^*(Y),
\end{equation}
for all $X \in \mathsf{dRig}$. In particular $\operatorname{QCoh}^*$ is the left Kan extension of its restriction to $\mathsf{dAfnd}$. One could further iterate the extension principles of \cite[\S A.5]{mann_p-adic_2022}, although we do not do this here for the reason stated above. Let $f: X \to Y$ be a morphism in $\mathsf{dRig}$. We will denote the six operations of the six-functor formalism $\operatorname{Strat}$ by 
\begin{equation}
    (f^*_{\mathrm{str}}, f_{\mathrm{str},*}, f_{\mathrm{str},!}, f_{\mathrm{str}}^!, \widehat{\otimes}_{X_{\mathrm{str}}}, \underline{\operatorname{Hom}}_{X_{\mathrm{str}}}), 
\end{equation}
where, of course, the functors $f_{\mathrm{str},!}$ and $f_{\mathrm{str}}^!$ are only defined when $f \in E_{\mathrm{str}}$. 
\subsection{Descent and Kashiwara's equivalence}\label{subsec:DescentKashiwara}
Let $f: X \to Y$ be any morphism in $\mathsf{dAfnd}$. The equivalence 
\begin{equation}
    \operatorname{QCoh}((X/Y)_{\mathrm{str}}) \xrightarrow[]{\sim} \underset{[n] \in \Delta}{\operatorname{lim}} \operatorname{QCoh}^*((X \subseteq X^{n+1/Y})^\dagger)
\end{equation}
follows automatically from the fact that $\operatorname{QCoh}^*$ is a limit-preserving functor on $\mathsf{PStk}^\mathsf{op}$. Using this equivalence, it is quite easy to show the following.

\begin{lem}\label{lem:Crysdescent}
Let $X \in \mathsf{dAfnd}$. Let $\{U_i \to X\}_{i =1}^n$ be a finite cover of $X$ by affinoid subspaces.
\begin{enumerate}[(i)]
    \item Let $\mathcal{I}$ be the family of finite nonempty subsets of $\{1, \dots, n\}$ and for each $I \in \mathcal{I}$ set $U_I := \bigcap_{i \in I} U_i$. Then, the canonical morphism
    \begin{equation}
        \operatorname{Strat}^*(X) \to \underset{I \in \mathcal{I}}{\operatorname{lim}} \operatorname{Strat}^*(U_I) 
    \end{equation}
    is an equivalence.
    \item Set $Y := \coprod_{i=1}^n U_i  \to X$. Then, the canonical morphism
    \begin{equation}
        \operatorname{Strat}^*(X) \to \underset{[m] \in \Delta}{\operatorname{lim}}\operatorname{Strat}^*(Y^{m+1/X}) 
    \end{equation}
    is an equivalence.
\end{enumerate}
\end{lem}
\begin{proof}
(i): Let $t_I: U_I \to X$ be the inclusions. By using the presentation  
\begin{equation}
    \operatorname{Strat}(X) \xrightarrow[]{\sim} \underset{[m] \in \Delta}{\operatorname{lim}} \operatorname{QCoh}^*((X \subseteq X^{n+1})^\dagger), 
\end{equation}
it is sufficient to show that, for each $m \geqslant 0$, the canonical morphism
\begin{equation}
    \operatorname{QCoh}^*((X \subseteq X^{m+1})^\dagger) \to \underset{I \in \mathcal{I}}{\operatorname{lim}}  \operatorname{QCoh}^*((U_I \subseteq U_I^{m+1})^\dagger)
\end{equation}
is an equivalence. However, this follows from Lemma \ref{lem:pairsdescent}(i), since the diagonally embedded copy of $|X|$ in $X^{m+1}$ is contained in  $\bigcup_{i=1}^n U_i^{\times m+1}$.

(ii): This is quite similar to the proof of (i), using Lemma \ref{lem:pairsdescent}(ii) instead of Lemma \ref{lem:pairsdescent}(i).
\end{proof}
\begin{cor}
The prestack $\operatorname{Strat}^*: \mathsf{dRig} \to \mathsf{CAlg}(\mathsf{Pr}^L_{\mathsf{st}})$ is a sheaf in the analytic topology.
\end{cor}
\begin{proof}
Since $\operatorname{Strat}^*$ is right Kan extended from $\mathsf{dAfnd}^{\mathsf{op}}$, the combination of Lemma \ref{lem:Crysdescent} and \cite[Proposition A.3.11]{mann_p-adic_2022} gives the Corollary. 
\end{proof}
Let $X \in \mathsf{dRig}$ and let $S \subseteq |X|$ be a closed subset of the underlying topological space. Let $j: U \to X$ be the inclusion of the open analytic subspace corresponding to to complement of $S$. We define 
\begin{equation}
    \Gamma_S \operatorname{Strat}(X) \subseteq \operatorname{Strat}(X)
\end{equation}
to be the full subcategory spanned by objects $M$ such that $j^*_{\mathrm{str}}M \simeq 0$.
\begin{prop}\label{prop:Kashiwara1}
Let $i: Z \to X$ be a Zariski-closed immersion in $\mathsf{dAfnd}$ which is induced by a morphism of algebras which is surjective on $\pi_0$. Assume that $i$ admits a retraction $r: X \to Z$. Then: 
\begin{enumerate}[(i)]
    \item There is a canonical equivalence
\begin{equation}
    (Z \subseteq X)^\dagger \simeq Z_{\mathrm{str}} \times_{X_{\mathrm{str}}} X.
\end{equation}
in $\mathsf{PStk}$.
\item In the Cartesian square 
\begin{equation}
\begin{tikzcd}
	{(Z \subseteq X)^\dagger} & X \\
	{Z_{\mathrm{str}}} & {X_{\mathrm{str}}}
	\arrow["{\iota}", from=1-1, to=1-2]
	\arrow["{q}", from=1-1, to=2-1]
	\arrow["\lrcorner"{anchor=center, pos=0.125}, draw=none, from=1-1, to=2-2]
	\arrow["p", from=1-2, to=2-2]
	\arrow["{i_{\mathrm{str}}}", from=2-1, to=2-2]
\end{tikzcd}
\end{equation}
coming from (i), the Beck-Chevalley morphism 
\begin{equation}
    p^*i_{\mathrm{str},*} \to \iota_* q^*
\end{equation}
is an equivalence of functors from $\operatorname{Strat}(Z) = \operatorname{QCoh}(Z_{\mathrm{str}})$ to $\operatorname{QCoh}(X)$. 
\item The pair $(i_{\mathrm{str}}^*, i_{\mathrm{str},*})$ induces an equivalence $\operatorname{Strat}(Z) \simeq \Gamma_{|Z|}\operatorname{Strat}(X)$.
\end{enumerate}
\end{prop}
\begin{proof}
(i): In order to calculate $X \times_{X_{\mathrm{str}}} Z_{\mathrm{str}}$, we use d\'ecalage, c.f. Lemma \ref{lem:decalage}. We note that $(\operatorname{Dec}_0 \operatorname{Inf}(X)) \times_{\operatorname{Inf}(X)} \operatorname{Inf}(Z)$ is given by the simplicial object  $(Z \subseteq X \times Z^{\bullet+1})^\dagger$, where $Z$ is embedded via the morphism $(i, \Delta_{n+1}) : Z \to X \times Z^{\times n+1}$. We recognise this simplicial object as the \v{C}ech nerve of $(X \subseteq X \times Z)^\dagger \to (Z \subseteq X)^\dagger$. This is a split epimorphism: the splitting is induced by $(\operatorname{id},r): X  \to X \times Z$. Therefore, we conclude that the augmented simplicial object $(Z \subseteq X \times Z^{\bullet +1})^\dagger \to (Z \subseteq X)^\dagger$ is split, as it is the nerve of a split epimorphism. By Lemma \ref{lem:decalage}, this proves that $(Z \subseteq X)^\dagger \simeq Z_{\mathrm{str}} \times_{X_{\mathrm{str}}} X$.

(ii): The functor $\operatorname{QCoh}^*$, by its construction, satisfies descent along $X \to X_{\mathrm{str}}$. By (i), the pullback of each covering map $X^{n+1/{X_{\mathrm{str}}}} \simeq (X \subseteq X^{n+1})^\dagger \to X_{\mathrm{str}}$ along $Z_{\mathrm{str}} \to X_{\mathrm{str}}$ is given by $\iota_n: (Z \subseteq X^{n+1})^\dagger \to (X \subseteq X^{n+1})^\dagger$. Each pushforward $\iota_{n,*}$ is compatible with base-change and therefore, for each Cartesian section 
\begin{equation}
    (M_n)_{n \in \Delta} \in \underset{[n] \in \Delta}{\operatorname{lim}} \operatorname{QCoh}((Z\subseteq X^{n+1})^\dagger) \simeq \operatorname{QCoh}(Z_{\mathrm{str}})
\end{equation}
then 
\begin{equation}
    (\iota_{n,*} M_n)_{n \in \Delta} \in \underset{[n] \in \Delta}{\operatorname{lim}} \operatorname{QCoh}((X \subseteq X^{n+1})^\dagger) \simeq \operatorname{QCoh}(X_{\mathrm{str}})
\end{equation}
is also a Cartesian section. By the equivalence of categories implicit in descent, this implies that $p^* i_{\mathrm{str},*} \simeq \iota_* q^*$.

(iii): Let $j: U \to X$ be the inclusion of the analytic subspace corresponding to $|X| \setminus |Z| \subseteq |X|$. Consider the following diagram in $\mathsf{PStk}$, in which both squares are Cartesian (for the left square this is (i) and for the right square this is Example \ref{ex:Infexample}):
\begin{equation}\label{eq:Kashiwararecollement}
\begin{tikzcd}
	{(Z \subseteq X)^\dagger} & X & U \\
	{Z_{\mathrm{str}}} & {X_{\mathrm{str}}} & {U_{\mathrm{str}}}
	\arrow["\iota", from=1-1, to=1-2]
	\arrow["q"', from=1-1, to=2-1]
	\arrow["\lrcorner"{anchor=center, pos=0.125}, draw=none, from=1-1, to=2-2]
	\arrow["p"', from=1-2, to=2-2]
	\arrow["j"', from=1-3, to=1-2]
	\arrow["\lrcorner"{anchor=center, pos=0.125, rotate=-90}, draw=none, from=1-3, to=2-2]
	\arrow["r", from=1-3, to=2-3]
	\arrow["{i_{\mathrm{str}}}", from=2-1, to=2-2]
	\arrow["{j_{\mathrm{str}}}"', from=2-3, to=2-2]
\end{tikzcd}
\end{equation}
We make two claims: (a) that the counit morphism $i^*_{\mathrm{str}} i_{\mathrm{str},*} \to \operatorname{id}$ is an equivalence, and (b) that $j^*_{\mathrm{str}}M \simeq 0$ if and only if the unit morphism $M \to i_{\mathrm{str},*}i_{\mathrm{str}}^* M$ is an equivalence. 

We note that each of the functors $q^*, p^*, r^*$ is conservative (each of the morphisms $p,q,r$ is of $*$-descent, because they are effective epimorphisms). In particular, it suffices to check that (a) is an equivalence after applying $q^*$. By commutativity of \eqref{eq:Kashiwararecollement}, and the base-change of part (ii), one has $q^* i_{\mathrm{str}}^* i_{\mathrm{str},*}  \simeq\iota^* \iota_* q^*$, and $\iota^*\iota_* \to \operatorname{id}$ is an equivalence, because $(Z \subseteq X)^\dagger \to X$ is a homotopy monomorphism.

For (b), we have the following chain of equivalences, for $M \in \operatorname{QCoh}(X_{\mathrm{str}})$:
\begin{equation*}
    \begin{aligned}
        j^*_{\mathrm{str}}M \simeq 0 &\iff  r^*j_{\mathrm{str}}^*M \simeq 0 && \text{by conservativity of }r^* \\
        &\iff j^* p^*M \simeq 0 && \text{by commutativity of \eqref{eq:Kashiwararecollement}}\\
        &\iff  \iota_* \iota^* p^* M \xrightarrow[]{\sim} p^*M && \text{by Corollary \ref{cor:recollementCor}} \\
        &\iff p^* i_{\mathrm{str},*} i_{\mathrm{str}}^* M \xrightarrow[]{\sim} p^*M && \text{by part (ii) and commutativity of \eqref{eq:Kashiwararecollement}} \\
        &\iff i_{\mathrm{str},*} i_{\mathrm{str}}^* M \xrightarrow[]{\sim} M && \text{by conservativity of }p^*.
    \end{aligned}
\end{equation*}
It then follows from the claims (a) and (b) that $(i_{\mathrm{str}}^*,i_{\mathrm{str},*})$ induces an equivalence $\operatorname{Strat}(Z) \simeq  \Gamma_{|Z|}\operatorname{Strat}(X)$. Indeed, for inclusions of coreflective subcategories, the essential image of the fully-faithful right adjoint is precisely those objects for which the unit morphism is an equivalence. 
\end{proof}
\begin{lem}\label{lem:CrysPairsdescent}
Let $X \in \mathsf{dRig}$ with $S \subseteq |X|$ a closed subset of the underlying topological space. Let open subsets $\{U_i\}_{i \in \mathscr{I}}$ of $|X|$ be given such that $S \subseteq \bigcup_{i \in \mathscr{I}}U_i$. Let $\mathcal{I}$ be the family of finite nonempty subsets of $\mathscr{I}$. Then, the natural morphism
\begin{equation}
    \Gamma_{S} \operatorname{Strat}(X) \to \underset{I \in \mathcal{I}}{\operatorname{lim}} \Gamma_{S_I}\operatorname{Strat}(U_I)
\end{equation}
induced by the collection of upper-star functors, is an equivalence.
\end{lem}
\begin{proof}
This follows quite straightforwardly from the fact that $\operatorname{Strat}^*$ is a sheaf in the analytic topology, c.f. Lemma \ref{lem:Crysdescent}. 
\end{proof}
\begin{defn}
Let $X \in \mathsf{dRig}$. A Zariski-closed immersion $Z \to X$ is called \emph{stratifying} if it locally admits a retraction. That is, there exists affinoid subspaces $\{U_i\}_{i \in \mathscr{I}}$ of $X$ such that $|Z| \subseteq \bigcup_{i \in \mathscr{I}} U_i$, and for all $i \in \mathscr{I}$, the morphism $Z \times_X U_i \to U_i$ admits a retraction. 
\end{defn}
\begin{rmk}
\begin{enumerate}[(i)]
    \item I chose the name \emph{stratifying} because a similar kind of closed immersions appears in the definition of the stratifying site in algebraic geometry \cite[\S 4.2]{Grothendieck_Crystals_68}.
    \item If $i: Z \to X$ is a closed immersion between smooth classical affinoid rigid spaces, then $i$ is stratifying. This follows from a result of Kiehl \cite[Theorem 1.19]{kiehl_rham_1967}, see also \cite[Proposition 2.11]{Lutkebohmertmaximum}.
\end{enumerate}
\end{rmk}
\begin{thm}[Kashiwara's equivalence]\label{thm:Kashiwara}
Let $i: Z \to X$ be a stratifying Zariski-closed immersion in $\mathsf{dRig}$. Then, the pair $(i_{\mathrm{str}}^*, i_{\mathrm{str},*})$ induces an equivalence 
\begin{equation}
\operatorname{Strat}(Z) \simeq \Gamma_{|Z|}\operatorname{Strat}(X).
\end{equation}
\end{thm}
\begin{proof}
We proceed in steps.

\textit{Step 1}: If $X \in \mathsf{dRig}$ is an affinoid and $i: Z \to X$ admits a retraction, this follows from Proposition \ref{prop:Kashiwara1}. 

\textit{Step 2}: Now suppose that $X$ is an analytic subspace of an affinoid space $X^\prime$ and $i: Z \to X$ is a Zariski-closed immersion which admits a retraction. Choose a covering $\{U_i \to X\}_{i \in \mathscr{I}}$ of $X$ by affinoid subspaces of $X^\prime$. Let $\mathcal{I}$ be the family of finite nonempty subsets of $\mathscr{I}$ and for each $I \in \mathcal{I}$ set $U_I := \bigcup_{i \in I} U_i$. Each $U_I$ is an affinoid and each $Z_I := Z \times_X U_I \to U_I$ admits a retraction. We have a commutative square 
\begin{equation}
\begin{tikzcd}[cramped]
	{\Gamma_{|Z|} \operatorname{Strat}(X)} & {\underset{I \in \mathcal{I}}{\operatorname{lim}}\Gamma_{|Z_I|}\operatorname{Strat}(U_I)} \\
	{\operatorname{Strat}(Z)} & {\underset{I \in \mathcal{I}}{\operatorname{lim}}\operatorname{Strat}(Z_I)}
	\arrow[from=1-1, to=1-2]
	\arrow[from=1-1, to=2-1]
	\arrow[from=1-2, to=2-2]
	\arrow[from=2-1, to=2-2]
\end{tikzcd}
\end{equation}
in which the horizontal arrows are equivalences by Lemma \ref{lem:Crysdescent} and Lemma \ref{lem:CrysPairsdescent}, and the right vertical arrow is an equivalence by Step 1. Therefore, the left vertical arrow is an equivalence. 

\textit{Step 3}: Now $X$ is arbitrary. Choose affinoid subspaces $\{U_i\}_{i \in \mathcal{I}}$ of $X$ such that $|Z| \subseteq \bigcup_{i \in \mathscr{I}} U_i$ and each $Z_I := Z \times_X U_i \to U_I$ admits a retraction. Each $U_I$ is an analytic subspace of an affinoid space, and each $Z_I \to U_I$ admits a retraction. Therefore, we may argue as in Step 2, using the result of Step 2, to conclude the proof.
\end{proof}
\subsection{The monad of differential operators and the comonad of jets}\label{subsec:DmodformulasnStuff}
\begin{defn}
Let $f: X \to Y$ be a morphism  in $\mathsf{dAfnd}$ and let $p_{X/Y}: X \to (X/Y)_{\mathrm{str}}$ be the canonical morphism. 
\begin{enumerate}[(i)]
    \item We define the \emph{comonad of jets} of $f$ to be the comonad 
    \begin{equation}
        \mathcal{J}^\infty_{X/Y} := p_{X/Y}^* p_{X/Y,*}
    \end{equation}
    acting on $\operatorname{QCoh}(X)$. We abbreviate $\mathcal{J}^\infty_{X/*}$ to $\mathcal{J}^\infty_X$. 
    \item We define the \emph{monad of differential operators} of $f$ to be the monad 
    \begin{equation}
        \mathcal{D}^\infty_{X/Y} := p_{X/Y}^! p_{X/Y,!}
    \end{equation}
    acting on $\operatorname{QCoh}(X)$. We abbreviate $\mathcal{D}^\infty_{X/*}$ to $\mathcal{D}^\infty_X$.
\end{enumerate}
\end{defn}
\begin{lem}\label{lem:p*p_!isomorphism.}
Let $f: X \to Y$ be a morphism in $\mathsf{dAfnd}$. Let $p_{X/Y}: X \to (X/Y)_{\mathrm{str}}$ be the canonical morphism. Then, there is a canonical equivalence $p_{X/Y,!} \xrightarrow[]{\sim} p_{X/Y,*}$. 
\end{lem}
\begin{proof}
This is an immediate consequence of Lemma \ref{lem:pX/Yshriekable}, because the morphism $p_{X/Y}$ is even representable in $\mathsf{PStk}$, and all such morphisms satisfy $p_{X/Y,!} \simeq p_{X/Y,*}$, c.f. Theorem \ref{thm:6FPStk}. 
\end{proof}
\begin{rmk}\label{rmk:JandDendofunctors}
\begin{enumerate}[(i)]
    \item Let $f: X =\operatorname{dSp}(A) \to \operatorname{dSp}(B) = Y $ be a morphism between affinoids. Due to Lemma \ref{lem:p*p_!isomorphism.} and base-change, the underlying endofunctors of $\mathcal{J}^\infty_{X/Y}$ and $\mathcal{D}^\infty_{X/Y}$ can be described as 
    \begin{equation}
    \begin{aligned}
        \mathcal{J}^\infty_{X/Y} &\simeq \tilde{\pi}_{1,*} \tilde{\pi}_2^*  \simeq  (A \widehat{\otimes}_B^\mathbf{L} A)^\dagger_{\Delta} \otimes^\mathbf{L}_A - \\  \mathcal{D}^\infty_{X/Y} &\simeq \tilde{\pi}_{2,!} \tilde{\pi}_1^! \simeq R \underline{\operatorname{Hom}}_A((A \widehat{\otimes}_B^\mathbf{L} A)^\dagger_{\Delta},-).
    \end{aligned}
    \end{equation}
    Here $\tilde{\pi}_1, \tilde{\pi}_2 : (X \subseteq X \times_Y X)^\dagger \to X$ are the two projections.
    \item As a formal consequence of Lemma \ref{lem:p*p_!isomorphism.}, we see that the monad $\mathcal{D}^\infty_{X/Y}$ is right adjoint to the comonad $\mathcal{J}^\infty_{X/Y}$. In particular, their Kleisli categories are equivalent. Morphisms in either Kleisli category can be interpreted as infinite-order differential operators. 
\end{enumerate}
\end{rmk}
\begin{lem}\label{lem:XYderhammoandicity}
Let $f: X \to Y$ be a morphism in $\mathsf{dAfnd}$ and let $p_{X/Y}: X \to (X/Y)_{\mathrm{str}}$ be the canonical morphism. 
\begin{enumerate}[(i)]
    \item The adjunction $p^*_{X/Y} \dashv p_{X/Y,*}$ is comonadic. That is, the comparison functor induced by $p^*_{X/Y}$ gives an equivalence of categories 
    \begin{equation}\label{eq:Jetequivalence}
        \operatorname{QCoh}((X/Y)_{\mathrm{str}}) \simeq \operatorname{Comod}_{\mathcal{J}^\infty_{X/Y}}\operatorname{QCoh}(X).
    \end{equation}
    \item Assume that $X \to (X/Y)_{\mathrm{str}}$ is of $!$-descent. Then the adjunction $p_{X/Y,!} \dashv p_{X/Y}^!$ is monadic. That is, the comparison functor induced by $p^!_{X/Y}$ gives an equivalence of categories
    \begin{equation}\label{eq:Dmoduleequivalence}
        \operatorname{QCoh}((X/Y)_{\mathrm{str}}) \simeq \operatorname{Mod}_{\mathcal{D}^\infty_{X/Y}}\operatorname{QCoh}(X). 
    \end{equation}
    \item In the situation of (ii), the implicit equivalence of categories \begin{equation}
     \operatorname{Mod}_{\mathcal{D}^\infty_{X/Y}} \operatorname{QCoh}(X) \simeq \operatorname{Comod}_{\mathcal{J}^\infty_{X/Y}}\operatorname{QCoh}(X)
    \end{equation}
    can be described as follows: The functor from left to right is given as 
    \begin{equation}
       \underset{[n] \in \Delta^{\mathsf{op}}}{\operatorname{colim}}\mathcal{J}^\infty_{X/Y} (\mathcal{D}^\infty_{X/Y})^n,
    \end{equation} 
    and the functor from right to left is given as \begin{equation}
    \underset{[n] \in \Delta}{\operatorname{lim}} \mathcal{D}^\infty_{X/Y} (\mathcal{J}^\infty_{X/Y})^n.
    \end{equation} 
\end{enumerate}
\end{lem}
\begin{proof}
(i): The functor $\operatorname{QCoh}^*$ satifies descent along $X \to (X/Y)_{\mathrm{str}}$. Therefore, the claim follows from Lemma \ref{lem:descentmonadicity}(i). 

(ii): This follows from Lemma \ref{lem:descentmonadicity}(ii).

(iii): By using the explicit equivalence of categories implicit in the Barr-Beck-Lurie theorem, c.f. Lemma \ref{lem:BarrBeckExplicit}, we see that the functor from left to right is given by 
\begin{equation}
    p_{X/Y}^* \underset{[n]\in \Delta^\mathsf{op}}{\operatorname{colim}} p_{X/Y,!}(p_{X/Y}^! p_{X/Y,!})^n \simeq \underset{[n]\in \Delta^\mathsf{op}}{\operatorname{colim}} \mathcal{J}^\infty_{X/Y} (\mathcal{D}^\infty_{X/Y})^n,
\end{equation}
where we used that $p_{X/Y}^*$ commutes with colimits, and that $p_{X/Y,!} \simeq p_{X/Y,*}$, by Lemma \ref{lem:p*p_!isomorphism.}. Similarly, we see that the functor from right to left is given by 
\begin{equation}
    p_{X/Y}^! \underset{[n] \in \Delta}{\operatorname{lim}} p_{X/Y,*}(p_{X/Y}^*p_{X/Y,*})^n \simeq \underset{[n] \in \Delta}{\operatorname{lim}} \mathcal{D}^\infty_{X/Y} (\mathcal{J}^\infty_{X/Y})^n,
\end{equation}
where we used that $p_{X/Y}^!$ commutes with limits and that $p_{X/Y,!} \simeq p_{X/Y,*}$ again.
\end{proof}
Under certain circumstances\footnote{That is, the functors appear to exist in greater generality than the nice formulas for them do.}, we can use Lemma \ref{lem:XYderhammoandicity} to give formulas for the six-operations, in terms of modules over the monad $\mathcal{D}^\infty_X$ or comodules over the comonad $\mathcal{J}^\infty_X$. 
\begin{thm}[Formulas for the six operations]\label{thm:formulasforsix}
\begin{enumerate}[(I)]
    \item Let $f: X \to Y$ be any morphism in $\mathsf{dAfnd}$. Let $M \in \operatorname{Comod}_{\mathcal{J}^\infty_X}\operatorname{QCoh}(Y)$. The upper-star pullback 
    \begin{equation}
        f^*M 
    \end{equation}
    is naturally an object of $\operatorname{Comod}_{\mathcal{J}^\infty_X}\operatorname{QCoh}(X)$. Under the equivalence of categories \eqref{eq:Jetequivalence}, this operation is identified with $f_{\mathrm{str}}^*: \operatorname{Strat}(Y) \to \operatorname{Strat}(X)$.
    \item Let $f: X \to Y$ be any morphism in $\mathsf{dAfnd}$ such that $Y \to Y_{\mathrm{str}}$ is of $!$-descent, and let $M \in \operatorname{Comod}_{\mathcal{J}^\infty_X}\operatorname{QCoh}(X)$. Then, the object 
    \begin{equation}
        \underset{[n] \in \Delta}{\operatorname{lim}}\mathcal{D}^\infty_Y f_* (\mathcal{J}^\infty_X)^n M
    \end{equation}
    is naturally an object of $\operatorname{Mod}_{\mathcal{D}^\infty_X}\operatorname{QCoh}(X)$. Under the equivalences of categories \eqref{eq:Jetequivalence} and \eqref{eq:Dmoduleequivalence}, this operation is identified with $f_{\mathrm{str},*}: \operatorname{Strat}(X) \to \operatorname{Strat}(Y)$. 
    \item Let $f: X \to Y$ be any morphism between objects of $\mathsf{dAfnd}$ which belongs to the class of good morphisms. Assume that $X \to X_{\mathrm{str}}$ is of $!$-descent, and let $M \in \operatorname{Mod}_{\mathcal{D}^\infty_X} \operatorname{QCoh}(X)$. Then, the object 
    \begin{equation}
        \underset{[n] \in \Delta^\mathsf{op}}{\operatorname{colim}} \mathcal{J}^\infty_Y f_! (\mathcal{D}^\infty_X)^nM 
    \end{equation}
    is naturally an object of $\operatorname{Comod}_{\mathcal{J}^\infty_Y} \operatorname{QCoh}(Y)$. Under the equivalences of categories \eqref{eq:Dmoduleequivalence} and  \eqref{eq:Jetequivalence}, this operation is identified with $f_{\mathrm{str},!}: \operatorname{Strat}(X) \to \operatorname{Strat}(Y)$. 
    \item Let $f: X \to Y$ be any morphism between objects of $\mathsf{dAfnd}$ which belongs to the class of good morphisms. Assume that $X \to X_{\mathrm{str}}$ and $Y \to Y_{\mathrm{str}}$ are of $!$-descent and let $M \in \operatorname{Mod}_{\mathcal{D}^\infty_Y} \operatorname{QCoh}(Y)$. The upper-shriek pullback 
    \begin{equation}
        f^!M
    \end{equation}
    is naturally an object of $\operatorname{Mod}_{\mathcal{D}^\infty_X} \operatorname{QCoh}(X)$. Under the equivalence of categories \eqref{eq:Dmoduleequivalence}, this operation is identified with $f_{\mathrm{str}}^!: \operatorname{Strat}(Y) \to \operatorname{Strat}(X)$. 
    \item Let $X \in \mathsf{dAfnd}$ and let $M, N \in \operatorname{Comod}_{\mathcal{J}^\infty_X} \operatorname{QCoh}(X)$. The tensor product 
    \begin{equation}
        M \widehat{\otimes}_X N
    \end{equation}
    is naturally an object of $\operatorname{Comod}_{\mathcal{J}^\infty_X} \operatorname{QCoh}(X)$. Under the equivalence of categories \eqref{eq:Jetequivalence}, this operation is identified with the tensor product $\widehat{\otimes}_{X_{\mathrm{str}}}$ on $\operatorname{Strat}(X)$. 
    \item Let $X \in \mathsf{dAfnd}$ and assume that $X \to X_{\mathrm{str}}$ is of $!$-descent. Let $M, N \in \operatorname{Mod}_{\mathcal{D}^\infty_X}\operatorname{QCoh}(X)$. Then, the object
    \begin{equation}
        \underset{[n] \in \Delta}{\operatorname{lim}} \mathcal{D}^\infty_X \underline{\operatorname{Hom}}_X((\mathcal{D}^\infty_X)^n M, N)
    \end{equation}
    is naturally an object of $\operatorname{Mod}_{\mathcal{D}^\infty_X}\operatorname{QCoh}(X)$. Under the equivalence of categories \eqref{eq:Dmoduleequivalence}, this operation is identified with the internal Hom bifunctor on $\operatorname{Strat}(X)$. 
\end{enumerate}
\end{thm}
\begin{proof}
All of these statements will be proven using Lemma \ref{lem:BarrBeckExplicit} and the commutative (but not Cartesian) square  
\begin{equation}\label{eq:XDRsquare}
\begin{tikzcd}[cramped]
	X & Y \\
	{X_{\mathrm{str}}} & {Y_{\mathrm{str}}}
	\arrow["f", from=1-1, to=1-2]
	\arrow["{p_X}", from=1-1, to=2-1]
	\arrow["{p_Y}", from=1-2, to=2-2]
	\arrow["{f_{\mathrm{str}}}"', from=2-1, to=2-2]
\end{tikzcd}
\end{equation}
in $\mathsf{PStk}$. We will also freely use the equivalences \eqref{eq:Jetequivalence} and \eqref{eq:Dmoduleequivalence}, as well as the equivalences $p_{X,!} \simeq p_{X,*}$ and $p_{Y,!} \simeq p_{Y,*}$ coming from Lemma \ref{lem:p*p_!isomorphism.}. 

(I): By using Lemma \ref{lem:BarrBeckExplicit} and Lemma \ref{lem:XYderhammoandicity} we know that the required functor from $\operatorname{Comod}_{\mathcal{J}^\infty_Y} \operatorname{QCoh}(Y)$ to $\operatorname{Comod}_{\mathcal{J}^\infty_X} \operatorname{QCoh}(X)$ is given by the formula 
\begin{equation}
    p_X^* f_{\mathrm{str}}^* \underset{[n] \in \Delta}{\operatorname{lim}} p_{Y,*}(p_{Y}^* p_{Y,*})^n 
\end{equation}
By using commutativity of the square \eqref{eq:XDRsquare} and Lemma \ref{lem:BarrBeckExplicit}, we have
\begin{equation}
     p_X^* f_{\mathrm{str}}^* \underset{[n] \in \Delta}{\operatorname{lim}} p_{Y,*}(p_{Y}^* p_{Y,*})^n \simeq f^* p_Y^* \underset{[n] \in \Delta}{\operatorname{lim}} p_{Y,*}(p_{Y}^*p_{Y,*})^n  \simeq f^*. 
\end{equation}

(II): By using Lemma \ref{lem:BarrBeckExplicit} and Lemma \ref{lem:XYderhammoandicity}, we know that the required functor from $\operatorname{Comod}_{\mathcal{J}^\infty_X} \operatorname{QCoh}(X)$ to $\operatorname{Mod}_{\mathcal{D}^\infty_Y}\operatorname{QCoh}(Y)$ is given by the formula
\begin{equation}
    p_Y^! f_{\mathrm{str},*} \underset{[n] \in \Delta}{\operatorname{lim}} p_{X,*}(p_{X}^* p_{X,*})^n.
\end{equation}
Using that $p_Y^!$ and $f_{\mathrm{str},*}$ commute with limits, the commutativity of the diagram \eqref{eq:XDRsquare}, and Lemma \ref{lem:p*p_!isomorphism.} we have 
\begin{equation}
\begin{aligned}
    p_Y^! f_{\mathrm{str},*} \underset{[n] \in \Delta}{\operatorname{lim}} p_{X,*}(p_{X}^* p_{X,*})^n &\simeq \underset{[n] \in \Delta}{\operatorname{lim}}  p_Y^! f_{\mathrm{str},*} p_{X,*}(p_{X}^* p_{X,*})^n  \\
    &\simeq \underset{[n] \in \Delta}{\operatorname{lim}}  p_Y^! p_{Y,*} f_* (p_{X}^* p_{X,*})^n \\
    &\simeq \underset{[n] \in \Delta}{\operatorname{lim}} \mathcal{D}^\infty_Y f_* (\mathcal{J}^\infty_X)^n.
\end{aligned} 
\end{equation}

(III): By using By using Lemma \ref{lem:BarrBeckExplicit} and Lemma \ref{lem:XYderhammoandicity}, we know that the required functor from $\operatorname{Mod}_{\mathcal{D}^\infty_X} \operatorname{QCoh}(X)$ to $\operatorname{Comod}_{\mathcal{J}^\infty_Y} \operatorname{QCoh}(Y)$ is given by the formula
\begin{equation}
    p_Y^* f_{\mathrm{str},!} \underset{[n] \in \Delta^\mathsf{op}}{\operatorname{colim}} p_{X,!}(p_X^! p_{X,!})^n
\end{equation}
Using that $p_Y^*$ and $f_{\mathrm{str},!}$ commute with colimits, the commutativity of the diagram \eqref{eq:XDRsquare}, and Lemma \ref{lem:p*p_!isomorphism.} we have
\begin{equation}
    \begin{aligned}
        p_Y^* f_{\mathrm{str},!} \underset{[n] \in \Delta^\mathsf{op}}{\operatorname{colim}} p_{X,!}(p_X^! p_{X,!})^n &\simeq  \underset{[n] \in \Delta^\mathsf{op}}{\operatorname{colim}} p_Y^* f_{\mathrm{str},!} p_{X,!}(p_X^! p_{X,!})^n \\
        &\simeq  \underset{[n] \in \Delta^\mathsf{op}}{\operatorname{colim}} p_Y^* p_{Y,!} f_! (p_X^! p_{X,!})^n \\
        &\simeq  \underset{[n] \in \Delta^\mathsf{op}}{\operatorname{colim}} \mathcal{J}^\infty_Y f_! (\mathcal{D}^\infty_X)^n.
    \end{aligned}
\end{equation}
(IV): By using Lemma \ref{lem:BarrBeckExplicit} and Lemma \ref{lem:XYderhammoandicity} we know that the required functor from $\operatorname{Mod}_{\mathcal{D}^\infty_Y} \operatorname{QCoh}(Y)$  to $\operatorname{Mod}_{\mathcal{D}^\infty_X} \operatorname{QCoh}(X)$ is given by the formula 
\begin{equation}
    p_X^! f_{\mathrm{str}}^! \underset{[n] \in \Delta^\mathsf{op}}{\operatorname{colim}} p_{Y,!}(p_Y^! p_{Y,!})^n.
\end{equation}
By using the commutativity of the square \eqref{eq:XDRsquare} and Lemma \ref{lem:BarrBeckExplicit}, we have 
\begin{equation}
    p_X^! f_{\mathrm{str}}^! \underset{[n] \in \Delta^\mathsf{op}}{\operatorname{colim}} p_{Y,!} (p_Y^! p_{Y,!})^n \simeq f^! p_Y^!  \underset{[n] \in \Delta^\mathsf{op}}{\operatorname{colim}} p_{Y,!} (p_Y^! p_{Y,!})^n \simeq f^!. 
\end{equation}

(V): Let $M, N \in \operatorname{Comod}_{\mathcal{J}^\infty_X}\operatorname{QCoh}(X)$. By transport of structure using Lemma \ref{lem:BarrBeckExplicit} and Lemma \ref{lem:XYderhammoandicity}, the tensor product on $\operatorname{Strat}(X)$ translates to the bifunctor sending $M, N$ to 
\begin{equation}
    p^*_X\Big(\big(\underset{[n] \in \Delta}{\operatorname{lim}} p_{X,*}(p_X^!p_{X,*})^n M\big) \widehat{\otimes}_{X_{\mathrm{str}}} \big( \underset{[m] \in \Delta}{\operatorname{lim}} p_{X,*}(p_X^!p_{X,*})^m N\big)\Big).
\end{equation}
Using that $p_X^*$ is symmetric-monoidal, and Lemma \ref{lem:BarrBeckExplicit}, we have equivalences 
\begin{equation}
\begin{aligned}
    & p^*_X\Big(\big(\underset{[n] \in \Delta}{\operatorname{lim}} p_{X,*}(p_X^!p_{X,*})^n M\big) \widehat{\otimes}_{X_{\mathrm{str}}} \big( \underset{[m] \in \Delta}{\operatorname{lim}} p_{X,*}(p_X^!p_{X,*})^m N\big)\Big) \\
    &\simeq \big(p^*_X\underset{[n] \in \Delta}{\operatorname{lim}} p_{X,*}(p_X^!p_{X,*})^n M\big) \widehat{\otimes}_{X} \big(p^*_X \underset{[m] \in \Delta}{\operatorname{lim}} p_{X,*}(p_X^!p_{X,*})^m N\big) \\
    &\simeq M \widehat{\otimes}_X N.
\end{aligned}
\end{equation}

(VI): Let $M, N \in \operatorname{Mod}_{\mathcal{D}^\infty_X}\operatorname{QCoh}(X)$. By transport of structure using Lemma \ref{lem:BarrBeckExplicit} and Lemma \ref{lem:XYderhammoandicity}, the internal Hom bifunctor on $\operatorname{Strat}(X)$ translates to the bifunctor sending $M, N$ to 
\begin{equation}
     p^!_X \underline{\operatorname{Hom}}_{X_{\mathrm{str}}}\Big(\underset{[n] \in \Delta^{\mathsf{op}}}{\operatorname{colim}} p_{X,*}(p_{X}^!p_{X,*})^n M , \underset{[m] \in \Delta^{\mathsf{op}}}{\operatorname{colim}} p_{X,*}(p_{X}^!p_{X,*})^m N\Big).
\end{equation}
We have the following chain of equivalences: 
\begin{equation}
     \begin{aligned}
      & p^!_X \underline{\operatorname{Hom}}_{X_{\mathrm{str}}}\Big(\underset{[n] \in \Delta^{\mathsf{op}}}{\operatorname{colim}} p_{X,*}(p_{X}^!p_{X,*})^n M , \underset{[m] \in \Delta^{\mathsf{op}}}{\operatorname{colim}} p_{X,*}(p_{X}^!p_{X,*})^m N\Big) \\
       &\simeq p^!_X \underset{[n] \in \Delta}{\operatorname{lim}} \underline{\operatorname{Hom}}_{X_{\mathrm{str}}}\Big( p_{X,*}(p_{X}^!p_{X,*})^n M , \underset{[m] \in \Delta^{\mathsf{op}}}{\operatorname{colim}} p_{X,*}(p_{X}^!p_{X,*})^m N\Big) \\
       &\simeq p^!_X \underset{[n] \in \Delta}{\operatorname{lim}}p_{X,*}\underline{\operatorname{Hom}}_{X_{\mathrm{str}}}( p_{X,*}(p_{X}^!p_{X,*})^n M , p_{X}^!\underset{[m] \in \Delta^{\mathsf{op}}}{\operatorname{colim}} p_{X,*}(p_{X}^!p_{X,*})^m N) \\
       &\simeq p^!_X \underset{[n] \in \Delta}{\operatorname{lim}}p_{X,*}\underline{\operatorname{Hom}}_{X_{\mathrm{str}}}( p_{X,*}(p_{X}^!p_{X,*})^n M , N)\\
       &\simeq \underset{[n] \in \Delta}{\operatorname{lim}} p^!_X p_{X,*}\underline{\operatorname{Hom}}_{X_{\mathrm{str}}}(p_{X,*}(p_{X}^!p_{X,*})^n M , N) \\
       &\simeq \underset{[n] \in \Delta}{\operatorname{lim}} \mathcal{D}_{X}^\infty \underline{\operatorname{Hom}}_X((\mathcal{D}_{X}^\infty)^n M, N),
    \end{aligned}
\end{equation}
where, in the third line we used the formula \eqref{eq:adjointtoprojectionformula}, in the fourth line we used Lemma \ref{lem:BarrBeckExplicit}, and in the fifth line we used that $p_X^!$ commutes with limits. 
\end{proof}
\subsection{The germ of the zero-section in $TX$}\label{sec:germzero}
\begin{prop}\label{prop:Adiagonal}
Let $X= \operatorname{Sp}(A)$ be a classical affinoid rigid space equipped with an \'etale morphism $X \to \mathbf{D}^r_K$. Then the algebra morphism
\begin{equation}
    \underset{m}{\operatorname{colim}} A \widehat{\otimes} K \langle dx/p^m \rangle \to \underset{ U \supset \Delta X}{\operatorname{colim}} (A \widehat{\otimes}_K A)_U,
\end{equation}
determined by $a \otimes 1 \mapsto a \otimes 1$ and $dx_i \mapsto 1 \otimes x_i - x_i \otimes 1$, is an isomorphism of complete bornological $K$-algebras. Here colimit on the left runs through the system of affinoid open neighbourhoods $U \supset \Delta X$ and $(-)_U$ denotes the corresponding affinoid localization. 
\end{prop}
\begin{proof}
Let\footnote{I would like to thank Finn Wiersig for a helpful discussion about the proof of this Proposition.} us denote the canonical morphism by
\begin{equation}
    \varphi: \underset{m}{\operatorname{colim}} A \widehat{\otimes} K \langle dx/p^m \rangle \to \underset{ U \supset \Delta X}{\operatorname{colim}} (A \widehat{\otimes}_K A)_U,
\end{equation}
in other words, for each $f = \sum_{\alpha \in \mathbf{N}^r}f_\alpha \otimes (dx)^\alpha$, one has 
\begin{equation}
    \varphi(f) = \sum_{\alpha \in \mathbf{N}^r} (f_\alpha \otimes 1)(1 \otimes x_i - x_i \otimes 1)^\alpha.
\end{equation}
It is not so hard to show that $\varphi$ is bounded. In order to prove the Proposition we will construct a bounded inverse $\psi$ to $\varphi$. Let us temporarily write $x_i := x_i \otimes 1$ and $y_i := 1 \otimes x_i$. Let $\partial_{x_i}$, (resp. $\partial_{y_i}$), be the derivations on $A \widehat{\otimes}_K A$ with $\partial_{x_i}(x_j) = \delta_{ij} = \partial_{y_i}(y_j)$ and $\partial_{x_i}(y_j)= 0 = \partial_{y_i}(x_j)$ for all $i,j$. The multiplication map $\mu : A \widehat{\otimes}_K A \to A$ induces a morphism $\mu : \operatorname{colim}_{ U \supset \Delta X} (A \widehat{\otimes}_K A)_U \to A$. Given $g \in \operatorname{colim}_{ U \supset \Delta X}(A \widehat{\otimes}_K A)_U$ we set
\begin{equation}
    \psi(g) := \sum_{\alpha \in \mathbf{N}^r} \frac{1}{\alpha!} \mu(\partial^\alpha_y g) \otimes (dx)^\alpha.  
\end{equation}
Now we will proceed in steps. 

\emph{Step 1: The morphism $\psi$ is well-defined and bounded.} With notations as above, fix an open affinoid $U \supseteq \Delta X$. The derivations $\partial_{y_i}$ restrict to well-defined bounded operators from $(A \otimes_K A)_U$ to itself. In particular, there exists constants $C_i$ such that $\|\partial_{y_i} g\|_U \leqslant C_i \|g\|_U$ for all $g \in (A \otimes_K A)_U$; here $\|\cdot \|_U$ denotes the residue norm on $(A \widehat{\otimes}_K A)_U$. By the well-known fact that $|1/\alpha!| \leqslant c^\alpha$ for some $c >0$, and by boundedness of $\mu$, we deduce that there exists constants $M$ and $K>0$ such that $\big\| \frac{1}{\alpha!} \mu(\partial^\alpha_y) g \| \leqslant M K^\alpha \| g \|_U$; here $\|\cdot \| $ denotes the residue norm on $A$. Hence if $|p^N| < K^{-1}$ then $\psi$ restricts to a bounded map $(A \otimes_K A)_U \to A \widehat{\otimes}_K K \langle dx_i/p^N \rangle_i$. This shows that $\psi$ is well-defined and the restriction of $\psi$ to $(A \widehat{\otimes}_K A)_U$ is bounded. Since $U$ was arbitrary, $\psi$ is bounded.  

\emph{Step 2: The composite $\psi \circ \varphi = \operatorname{id}$.} 
We note that the derivations $\partial_{y_i}$ are bounded and $A$-linear, for the first copy of $A$. Hence, if $g = \varphi(f) = \sum_{\alpha \in \mathbf{N}^r} (f_\alpha \otimes 1)(y-x)^\alpha$, then $\frac{1}{\alpha!} \mu(\partial^\alpha_yg) = f_\alpha$ so that $\psi \circ \varphi = \operatorname{id}$.

\emph{Step 3: $\psi$ is injective}. 
We first note that we can replace the system of all affinoid neighbourhoods $\{U \supseteq \Delta X\}$ of $\Delta X$ in $X \times X$,  with the system of all affinoid neighbourhoods $U^\prime$ of $\Delta X$ with the following property: each connected component of $\Delta X$ is contained in a unique connected component of $U^\prime$. We now consider the system of algebras defining the colimit
\begin{equation}\label{eq:connectedsystem}
    \underset{U^\prime \supseteq \Delta X}{\operatorname{colim}} (A \widehat{\otimes}_K A)_{U^\prime},
\end{equation}
where $U^\prime$ runs over the affinoid neighbourhoods as above. Recalling that $X$ is smooth, one has in particular that each $(A \widehat{\otimes}_K A)_{U^\prime}$ is normal. Therefore, it is a product (indexed by the connected components of $U^\prime$) of integral domains. It is also Noetherian. Also, by normality, all the transition morphisms in the system \eqref{eq:connectedsystem} are injective. We may use these facts implictly in the following.

Let $I \subseteq A \widehat{\otimes}_K A$ be the ideal defining the diagonal. Then Krull's intersection theorem for Noetherian integral domains (applied separately in each connected component of $U^\prime$) implies that the canonical morphism 
\begin{equation}
    (A \widehat{\otimes}_K A)_{U^\prime} \to \underset{k}{\operatorname{lim}}(A \widehat{\otimes}_K A)_{U^\prime}/I^k  \cong  \underset{k}{\operatorname{lim}}(A \widehat{\otimes}_K A)/I^k 
\end{equation}
is injective. Since the transition morphisms in the system \eqref{eq:connectedsystem} are injective we obtain a canonical injective morphism
\begin{equation}
    \underset{U^\prime \supseteq \Delta X}{\operatorname{colim}} (A \widehat{\otimes}_K A)_{U^\prime} \hookrightarrow \underset{k}{\operatorname{lim}}(A \widehat{\otimes}_K A)/I^k .
\end{equation}
By flatness of $A$ with respect to $\widehat{\otimes}_K$, we see that the morphisms in the system 
\begin{equation}
    \{ A \widehat{\otimes}_K  K\langle dx/p^m \rangle \}_{m \geqslant 0}.
\end{equation} 
are injective. Applying Krull's intersection theorem in a similar manner to above, we obtain a canonical injective morphism 
\begin{equation}
    \underset{k}{\operatorname{colim}} A \widehat{\otimes}_K K\langle dx/p^m \rangle \hookrightarrow \underset{k}{\operatorname{lim}}(A \otimes K[dx])/(dx)^k.
\end{equation}
These morphisms fit into a commutative square:
\begin{equation}\label{eqref:square}
\begin{tikzcd}[cramped]
	{\underset{U^\prime \supseteq \Delta X}{\operatorname{colim}} (A \widehat{\otimes}_K A)} & {\underset{m}{\operatorname{colim}} (A \widehat{\otimes}_K K\langle dx/p^m\rangle)} \\
	{\underset{k}{\operatorname{lim}}(A \widehat{\otimes}_K A)/I^k} & { \underset{k}{\operatorname{lim}}A \widehat{\otimes}_K K[ dx]/(dx)^k}
	\arrow["\psi", from=1-1, to=1-2]
	\arrow[hook, from=1-1, to=2-1]
	\arrow[hook, from=1-2, to=2-2]
	\arrow[from=2-1, to=2-2]
\end{tikzcd}
\end{equation}
We claim that
\begin{equation}
    \underset{k}{\operatorname{lim}} \varphi_k : \underset{k}{\operatorname{lim}} A \widehat{\otimes}_K K[ dx]/(dx)^k  \to \underset{k}{\operatorname{lim}} (A \widehat{\otimes}_K A)/I^k
\end{equation}
is an isomorphism, where $\varphi_k: A \widehat{\otimes}_K K[ dx]/(dx)^k  \to (A \widehat{\otimes}_K A)/I^k$ is induced by $\varphi$. That will be enough to prove Step 3 because the bottom arrow in \eqref{eqref:square} is inverse to this morphism.

Because $X$ is smooth, the immersion $\Delta: X \to X \times X$ is regular. We recall also that $
\Omega^1_{X/K} \cong I/I^2$ by sending $dx_i \mapsto y_i - x_i$. Therefore, we obtain for each  $k \geqslant 0$ an isomorphism $\operatorname{Sym}^k(\Omega^1_{X/K}) \to I^k/I^{k+1}$ sending $(dx)^\alpha \mapsto (y-x)^\alpha$. By passing to the graded, this implies that each $\varphi_k$ is an isomorphism. 

\emph{Step 4: Completing the proof}. By Step 2, $\psi$ is strict epimorphism of complete bornological spaces. It is also injective by Step 3. Therefore, $\psi$ is an isomorphism.
\end{proof}
Motivated by this result we define another groupoid object as follows. In the ``algebraic" setting, I found that reading \cite[\S4]{VandenBerghHochschild} was very useful. This discussion is also quite similar to \cite[Example 6.1.8]{Camargo_deRham}. Let $X$ be a classical affinoid rigid space equipped with an \'etale morphism $X \to \mathbf{D}^r_K$. Then we can consider the \emph{germ of the zero section} as an object of $\mathsf{dAff}$:
\begin{equation}
    (X \subseteq TX)^\dagger = \operatorname{dSp}( A\widehat{\otimes}_K K \langle dx/p^\infty \rangle).
\end{equation}
There is a augmentation $\varepsilon: A\widehat{\otimes}_K K \langle dx/p^\infty \rangle \to A$ and two algebra morphisms $\sigma, \tau: A \to A\widehat{\otimes}_K K \langle dx/p^\infty \rangle$. One has 
\begin{equation}
    \sigma:= \operatorname{id} \otimes 1 : A \to A \widehat{\otimes}_KK \langle dx/p^\infty\rangle,
\end{equation}
which gives the \emph{left} $A$-module structure on $A \widehat{\otimes}_KK \langle dx/p^\infty\rangle$. The algebra morphism $\tau$ sends a function to its Taylor series:
\begin{equation}
    \tau(a) := \sum_{\alpha \in \mathbf{N}^r} \frac{1}{\alpha!} \partial^\alpha (a)\otimes (dx)^\alpha.  
\end{equation}
This gives the \emph{right} $A$-module structure on $A \widehat{\otimes}_KK \langle dx/p^\infty\rangle$. There is an algebra morphism
\begin{equation}\label{eq:Psimorphism}
    \psi: A \widehat{\otimes} K \langle dx/p^\infty\rangle \to (A \widehat{\otimes}_K K \langle dx/p^\infty\rangle) \widehat{\otimes}_A (A \widehat{\otimes}_K K \langle dx/p^\infty\rangle)
\end{equation}
determined by $\psi(a\otimes 1) = (a \otimes 1) \otimes (1 \otimes 1)$ and 
\begin{equation}
    \psi(1 \otimes dx_i) = (1\otimes dx_i) \otimes (1 \otimes 1) + (1 \otimes 1) \otimes (1 \otimes dx_i). 
\end{equation}
It is important to remember that the tensor product \eqref{eq:Psimorphism} is taken with respect to the right $A$-module structure (via $\tau$) on the first factor, and the left  $A$-module structure on the second factor. Using the morphisms $\varepsilon, \sigma, \tau, \psi$ we obtain a groupoid object
\begin{equation}
\begin{tikzcd}
	\cdots & {(X \subseteq TX)^\dagger \times_X (X \subseteq TX)^\dagger} & {(X \subseteq TX)^\dagger} & X
	\arrow[shift right, from=1-1, to=1-2]
	\arrow[shift left, from=1-1, to=1-2]
	\arrow[shift right=3, from=1-1, to=1-2]
	\arrow[shift left=3, from=1-1, to=1-2]
	\arrow[from=1-2, to=1-3]
	\arrow[shift left=2, from=1-2, to=1-3]
	\arrow[shift right=2, from=1-2, to=1-3]
	\arrow[shift right, from=1-3, to=1-4]
	\arrow[shift left, from=1-3, to=1-4]
\end{tikzcd}
\end{equation}
where we suppressed the degeneracy maps and we emphasise that the fiber product is taken with respect to $\sigma$ and $\tau$. Let us call this groupoid object $\operatorname{exp}(\mathcal{T}_X)$. Now Proposition \ref{prop:Adiagonal} can be rephrased in the following way.
\begin{thm}\label{thm:ExpTX}
With notations as above. There is an equivalence 
\begin{equation}
    \operatorname{exp}(\mathcal{T}_X) \simeq \operatorname{Inf}(X)
\end{equation}
of simplicial objects in $\mathsf{dAff}$.
\end{thm}
\begin{rmk}
Using the isomorphism of Proposition \ref{prop:Adiagonal} the morphisms $\epsilon, \sigma, \tau$ and $\psi$ may be described in the following (and more symmetric) way: $\epsilon$ is induced by the multiplication $A \widehat{\otimes}_K A \to A$, $\sigma, \tau$ are induced by $\operatorname{id} \otimes 1$ and $1 \otimes \operatorname{id} : A \to A \widehat{\otimes}_K A$, respectively, and $\psi$ is induced by the morphism $A \widehat{\otimes}_K A \to (A \widehat{\otimes}_K A) \widehat{\otimes}_A (A \widehat{\otimes}_K A)$ which sends $a \otimes a^\prime \mapsto (a \otimes 1) \otimes (1 \otimes a^\prime)$. We see that this is really the same as \cite[\S 16.8]{EGAIV.4}, but we just replaced the \emph{formal neighbourhood of the diagonal} with the \emph{germ of the diagonal}. 
\end{rmk}
Continuing in the above setup, let derivations $\partial_i$ be dual to $dx_i$. We recall that Ardakov--Wadsley's ring $\wideparen{\mathcal{D}}_X(X)$ can be explicitly written as ``rapidly decreasing series in the variable $\partial$":
\begin{equation}
    \wideparen{\mathcal{D}}_X(X) = \left\{\sum_{\alpha \in \mathbf{N}^r} f_\alpha \partial^\alpha : \|f_\alpha \|r^{\alpha} \xrightarrow[]{\alpha \to \infty} 0  \text{ for all } r>0\right\},
\end{equation}
with the expected multiplication for differential operators. For the sake of brevity let us write $J := A \widehat{\otimes}_K K \langle dx/p^\infty\rangle$ and $U := \wideparen{\mathcal{D}}_X(X)$. We may identify $J$ with the left $A$-linear (bornological) dual of $U$ via the pairing $(\partial^\beta, (dx)^\alpha) := \alpha!\delta_{\alpha \beta}$. Using this identification one may define two\footnote{The action ${}^2\nabla$ is not used in this article. I only mention it to emphasise that $J$ has lots of extra structure. The actions ${}^1\nabla, {}^2\nabla$ can be explained in the following way. We regard $J$ as the left $A$-linear dual of $U := \wideparen{\mathcal{D}}_X(X)$: $$
        J = \underline{\operatorname{Hom}}_A(U,A).$$
    Now $U$ is of course a $U$-$U$-bimodule. We can view ${}^2\nabla$ as the na\"ive action by derivations on $J$ which comes from the right $U$-module structure on $U$. On the other hand, ${}^1\nabla$ is the action by derivations on $J$ which comes from the left $U$-module structures on $U$ and $A$ and Oda's rule \cite[Proposition 1.2.9(iii)]{hotta_d-modules_2008}.} commuting actions of $T:= \mathcal{T}_X(X)$ on $J$ by derivations:
\begin{equation}
\begin{aligned}
    {}^1\nabla_\theta(j)(D) := \theta(j(D)) - j(\theta D), && {}^2\nabla_\theta(j)(D) := j(D\theta).
\end{aligned}
\end{equation}
for $j \in J, D \in U, \theta \in T$. Using the action ${}^1\nabla$ of $T$ by derivations we obtain the \emph{de Rham complex} of $J$ which is augmented via $\tau: A \to J$:
\begin{equation}\label{eq:AugmenteddeRham}
    A \to \Omega^\bullet_{A/K} \widehat{\otimes}_{A, \sigma} J.
\end{equation}
\begin{prop}\label{prop:DualSpencer}
The augmented de Rham complex \eqref{eq:AugmenteddeRham} is strictly exact.
\end{prop}
\begin{proof}
The exactness follows from the ``algebraic" version of this statement (see for instance \cite[Proposition 4.2.4]{VandenBerghHochschild}), using that $A \otimes_K K[dx] \to J$ is flat. The strictness can then be obtained by applying an appropriate version of the closed-graph theorem \cite[\S 2]{WaelbroeckBounded}. 
\end{proof}
\begin{rmk}
One may recognise the morphism $J  \to \Omega^1_{A/K} \widehat{\otimes}_{A, \sigma} J$ in the complex \eqref{eq:AugmenteddeRham} as being given by Euler--Lagrange operators. 
\end{rmk}
Now using this finite resolution we obtain the following. 
\begin{thm}\label{thm:StratMonadicity}
Let $X = \operatorname{Sp}(A)$ be a classical smooth affinoid rigid space equipped with an \'etale morphism $X \to \mathbf{D}^r_K$. Let $p: X \to X_{\mathrm{str}}$ be the canonical morphism. Then:
\begin{enumerate}[(i)]
    \item $p_* 1_X \in \operatorname{QCoh}(X_{\mathrm{str}}) = \operatorname{Strat}(X)$ is descendable in the sense of Mathew \cite[\S 3.3]{MathewGalois}.
    \item The morphism $p: X \to X_{\mathrm{str}}$ is of universal $!$-descent. In particular, there is an equivalence of categories $\operatorname{Strat}(X) \simeq \operatorname{Mod}_{\mathcal{D}^\infty_X}(\operatorname{QCoh}(X))$, where the latter is the category of modules over the monad $\mathcal{D}^\infty_X$. 
\end{enumerate}
\end{thm}
\begin{proof}
Due to Theorem \ref{thm:ExpTX} we may consider $X/\operatorname{exp}(\mathcal{T}_X) := \operatorname{colim}_{[n] \in \Delta^\mathsf{op}} \operatorname{exp}(\mathcal{T}_X)_n$ with its canonical morphism $q: X \to X/\operatorname{exp}(\mathcal{T}_X)$ instead of $p: X \to X_{\mathrm{str}}$.

(i): Since the morphism $q: X \to X/\operatorname{exp}(\mathcal{T}_X)$ is of universal $*$-descent (as it is an effective epimorphism), there is an equivalence $\operatorname{QCoh}(X/\operatorname{exp}(\mathcal{T}_X)) \simeq \operatorname{Comod}_{q^*q_*} (\operatorname{QCoh}(X))$ induced by $q^*$. Therefore, to show that $1_{X/\operatorname{exp}(\mathcal{T})
}$ belongs to the thick subcategory generated by $q_*1_X$, is the same as showing that $1_X$ belongs to the thick subcategory of $\operatorname{Comod}_{q^*q_*}(\operatorname{QCoh}(X))$ generated by $q^*q_*1_X = J$. By the Dold-Kan correspondence it is sufficient to give a bounded finite-free resolution of $1_X$ as a $J$-comodule\footnote{Here the $A$-$A$ bimodule object $J$ is viewed as a coalgebra under \emph{convolution}.}. But that is exactly Proposition \ref{prop:DualSpencer}. 

(ii): By Lemma \ref{lem:p*p_!isomorphism.}, we know that $q_! \simeq q_*$.  By (i) above, $q_*1_X \in \operatorname{QCoh}(X/\operatorname{exp}(\mathcal{T}_X))$ is descendable. Therefore one may argue in an essentially identical way to \cite[Proposition 6.19]{ScholzeSixFunctors} to deduce that $q$ is of universal $!$-descent. 
\end{proof}
\subsection{A relation to work of Ardakov--Wadsley}\label{subsec:AWrelation}
In this section we will compute the action of the monad $\mathcal{D}^\infty_X $ on the unit object $1_X \in \operatorname{QCoh}(X)$ when $X= \operatorname{dSp}(A)$ is a smooth classical affinoid equipped with an \'etale morphism $X \to \mathbf{D}^r_K$, and show that it agrees with the infinite-order differential operators defined by Ardakov--Wadsley \cite{Dcap1}. The purpose of this subsection is more-or-less to convince the reader that ``we got the right answer". A more substantial investigation, including a possible relation to coadmissible $\wideparen{\mathcal{D}}$-modules, appears in our subsequent work \cite{soor_relation_2025}. 
\begin{prop}
Let $X = \operatorname{dSp}(A)$ be a smooth classical affinoid rigid space equipped with an \'etale morphism $X \to \mathbf{D}^r_K$. Then, there is a isomorphism $\mathcal{D}^\infty_X 1_X \simeq \wideparen{\mathcal{D}}_X(X)$, where the latter is defined as in \cite{Dcap1} and viewed as an object of $\operatorname{QCoh}(X)$ concentrated in degree $0$. 
\end{prop}
\begin{proof}
In this proof, we will no longer suppress the fact that various functors were really derived. That is to say, we will reintroduce $\mathbf{L}$'s and $R$'s to our notation\footnote{In particular, the absence of $\mathbf{L}$'s and $R$'s also becomes meaningful.}. Let $q: X \to \operatorname{X}/\operatorname{exp}(\mathcal{T}_X)$ be the canonical morphism. Thanks to Theorem \ref{thm:ExpTX} we know that there is an equivalence $\mathcal{D}^\infty_X1_X \simeq q^!q_!1_X$. Now by base-change one has 
\begin{equation}
    q^!q_! 1_X \simeq R\underline{\operatorname{Hom}}_A(A \widehat{\otimes}_KK \langle dx/p^\infty \rangle , A). 
\end{equation}
Now one has 
\begin{equation}
    R\underline{\operatorname{Hom}}_A(A \widehat{\otimes}_KK \langle dx/p^\infty \rangle , A) \simeq R\underset{m}{\operatorname{lim}} \underline{\operatorname{Hom}}_A(A \widehat{\otimes}_KK \langle dx/p^m \rangle , A).
\end{equation}
Using the pairing $(dx^\alpha, \partial^\beta) = \alpha! \delta_{\alpha \beta}$ and cofinality this is further equivalent to 
\begin{equation}
    R \underset{m}{\operatorname{lim}} A \widehat{\otimes}_K K \langle p^m \partial \rangle \simeq \underset{m}{\operatorname{lim}} A \widehat{\otimes}_K K \langle p^m \partial \rangle,
\end{equation}
where we used the Mittag--Leffler result of \cite[Theorem 5.26]{bode_six_2021}. Hence the result follows from the explicit description of \cite[Lemma 3.4]{ardakov_bounded_2018}, which says that $\operatorname{lim}_m A \widehat{\otimes}_K K \langle p^m \partial \rangle \simeq \wideparen{\mathcal{D}}_X(X)$. 
\end{proof}

\appendix
\section{Monads and descent}\label{sec:monads}

I assume that the reader is familiar with the notion of a monad in ordinary category theory, and a \emph{module} over it\footnote{In this article we prefer to use the terminology of \emph{modules over monads} rather than \emph{algebras over monads}.}. In short, given a an ordinary category $\mathscr{C}$, the category $\operatorname{End}(\mathscr{C})$ of endofunctors is naturally a monoidal category via composition of functors. Monoids in $\operatorname{End}(\mathscr{C})$ are known as \emph{monads}. To give such an object is the same as giving an endofunctor $T$ together with \emph{multiplication} and \emph{unit} transformations $\mu : T^2 \to T$ and $\eta : \operatorname{id} \to T$ satisfying diagrams expressing strict associativity and unitality.  Evaluation on objects gives an action $\operatorname{End}(\mathscr{C}) \times \mathscr{C} \to \mathscr{C}$ of this monoidal category on $\mathscr{C}$. The notion of a module over an algebra object makes sense in this generality.\footnote{That is, whenever one has a monoidal category $\mathscr{D}$ acting on a category $\mathscr{C}$, one can consider, for any algebra object $T \in \mathscr{D}$, the category $\operatorname{Mod}_T\mathscr{C}$ of module objects in $\mathscr{C}$ over $T$. In particular, it is not necessary for the ``module" and the ``algebra" to belong to the same category.} In particular, given any monad $T \in \mathscr{C}$, we obtain the category $\operatorname{Mod}_T \mathscr{C}$ of modules over the monad. To give such an object is the same as giving an object $M \in \mathscr{C}$ together with an action morphism $TM \to M$ satisfying a diagram expressing the module axiom. The definition of a comonad, and a comodule over a comonad, is obtained by formally dualizing these definitions.

An abundant source of (co)monads in ordinary category theory, is from adjunctions. Given an adjunction $F: \mathscr{D} \leftrightarrows \mathscr{C}: G$, the endofunctor $FG$ on $\mathscr{C}$ naturally acquires the structure of a monad with unit $\eta: \operatorname{id} \to FG$ acquired from the unit transformation and multiplication $\mu: FGFG \to FG$ acquired from the counit transformation $\varepsilon: GF \to \operatorname{id}$. The functor $G$ canonically factors through the forgetful functor as $\mathscr{C} \xrightarrow{K} \operatorname{Mod}_{FG} \mathscr{D} \to \mathscr{D}$. The functor $K$ is called the \emph{comparison functor}, and the classical theorem of Barr-Beck gives neccesary and sufficient conditions for $K$ to be an equivalence, in which case we say $F \dashv G$ is \emph{monadic}. That is, $G$ should be conservative and preserve certain kinds of colimits, called $G$-split coequalizers.\footnote{Of course, the Barr-Beck theorem also has a formal dual for comonadic adjunctions.}

The comonadicity of adjunctions is related to descent theory via the classical B\'enabou-Roubaud theorem, which roughly says that, under a hypothesis related to base-change, (which is sometimes called the Beck-Chevalley condition), descent is ``the same" as the comonadicity of the push-pull adjunction.

In this subsection, I will attempt, in analogy with this story, to describe how modules over a \emph{homotopy-coherent monad} are constructed in higher category theory. I will also record some results about the relation between monads and descent theory. 
\subsection{What are modules over a homotopy-coherent monad?}\label{subsec:whataremodules}
I am including this section for completeness, but I cannot pretend to give a better introduction than the notes of Lukas Brantner \cite{BrantnerDeformationTCC}, which were enormously helpful for me when learning this. We recall that $\infty$-categories are built out of the category $s\mathsf{Set}$ simplicial sets. More precisely, they are the fibrant objects in the Joyal model structure on $s\mathsf{Set}$ \cite[\S 2.2.5]{HigherToposTheory}. By using the Cartesian-closed structure\footnote{That is, given $X,Y \in s\mathsf{Set}$ one may define $Y^X \in s\mathsf{Set}$ by $(Y^X)_n:= \operatorname{Map}_{s \mathsf{Set}} ( \Delta^n \times X, Y)$, and there is an adjunction $(-) \times X \dashv (-)^X$.} on $s\mathsf{Set}$, one can, to each $\infty$-category $\mathscr{C}$, define another $\infty$-category $
\operatorname{End}(\mathscr{C}) := \mathscr{C}^\mathscr{C}.$
By again using the Cartesian-closed structure on $s\mathsf{Set}$ one obtains a diagram $N(\Delta^{\mathsf{op}}) \to s\mathsf{Set}$ sending $[n] \mapsto \operatorname{End}(\mathscr{C})^n$,\footnote{So that $[0] \mapsto [0] \in s\mathsf{Set}$.}, c.f. \cite[Definition 2.46]{BrantnerDeformationTCC}. By the straightening/unstraightening equivalence\footnote{\emph{Straightening/unstraightening} is the common name given to the \emph{Grothendieck construction} for $(\infty,1)$ categories \cite[\S 3.2]{HigherToposTheory}. It asserts that there is a Quillen equivalence between the category of simplicial functors, and coCartesian fibrations.} this can be equivalently viewed as a coCartesian fibration\footnote{A coCartesian fibration is the higher categorical counterpart of a Grothendieck op-fibration, see \cite[Definition 2.4.2.1]{HigherToposTheory}.} $\operatorname{End}(\mathscr{C})^{\otimes} \to N(\Delta^{\mathsf{op}})$. Here $\operatorname{End}(\mathscr{C})^{\otimes}$ is defined by unstraightening applied to the previous construction. This fibration turns out to be a \emph{monoidal $\infty$-category}, i.e., it is an inner fibration\footnote{A morphism of simplicial sets is called an \emph{inner fibration} if it has the left lifting property against inclusions of \emph{inner} horns.} of simplicial sets satisfying the Segal property.\footnote{For details we direct the reader to \cite[Ch. 2]{HigherAlgebra} and \cite[Ch. 1]{LurieNonCommutativeAlgebra}.} When referring to this monoidal $\infty$-category we may suppress the implicit fibration or even also the $\otimes$.
\begin{defn}\cite[Definition 1.1.14, Definition 3.1.3]{LurieNonCommutativeAlgebra}
A \emph{homotopy-coherent monad}, which we will often just call a \emph{monad}, is defined to be an algebra object in the monoidal $\infty$-category $\operatorname{End}(\mathscr{C})^{\otimes}$. This means that it is a section of the fibration $\operatorname{End}(\mathscr{C})^{\otimes} \to N(\Delta^{\mathsf{op}})$ sending the \emph{inert morphisms}\footnote{See \cite[Lecture 2, Definition 2.48]{BrantnerDeformationTCC}.} in $N(\Delta^{\mathsf{op}})$ to coCartesian edges.   
\end{defn}
Informally, a homotopy-coherent monad is an endofunctor $T \in \operatorname{End}(\mathscr{C})$ equipped with morphisms $\mu: T^2 \to T$ and $\eta : \operatorname{id} \to T$ which are associative and unital up to coherent homotopy.

In higher-category theory, the action of the monoidal $\infty$-category $\operatorname{End}(\mathscr{C})^\otimes$ on $\mathscr{C}$ is captured via the definition of a \emph{left-tensored $\infty$-category} \cite[Definition 2.1.1]{LurieNonCommutativeAlgebra}. In this situation this means the following, c.f. \cite[Proposition 3.1.2]{LurieNonCommutativeAlgebra}. We will construct an $\infty$-category $\mathscr{C}^\otimes$ equipped with a coCartesian fibration $\mathscr{C}^\otimes \to N(\Delta^{\mathsf{op}})$ and a
fibration\footnote{In the Joyal model structure.} $\mathscr{C}^\otimes \to \operatorname{End}(\mathscr{C})^\otimes$ over $N(\Delta^{\mathsf{op}})$ such that, for every $n \geqslant 0$ the inclusion $\{n\} \to [n]$ induces an equivalence $\mathscr{C}_{[n]}^\otimes \xrightarrow[]{\sim} \operatorname{End}(\mathscr{C})^\otimes_{[n]} \times \mathscr{C}$; here $(-)_{[n]}$ denotes the fiber over $[n] \in \Delta^\mathsf{op}$. In order to define $\mathscr{C}^\otimes$ we note that the Cartesian-closed structure on $s\mathsf{Set}$ gives a diagram $N(\Delta^\mathsf{op}) \to s\mathsf{Set}$ sending $[n] \mapsto \operatorname{End}(\mathscr{C})^n \times \mathscr{C}$. By using the morphism $\mathscr{C} \to [0]$ to the terminal object and unstraightening  we obtain $\infty$-category $\mathscr{C}^\otimes$ equipped with the desired fibration $\mathscr{C}^\otimes \to \operatorname{End}(\mathscr{C})^\otimes$ over $N(\Delta^\mathsf{op})$. 
\begin{defn}\cite[Definition 2.1.4, Definition 3.1.3]{LurieNonCommutativeAlgebra}
Fix a homotopy-coherent monad $T$, viewed as a section $T: N(\Delta^\mathsf{op}) \to \operatorname{End}(\mathscr{C})^\otimes$. A \emph{module over $T$} is a section $M: N(\Delta^\mathsf{op}) \to \mathscr{C}^\otimes$ such that:
\begin{itemize}
    \item[$\star$] The composite $N(\Delta^\mathsf{op}) \to \mathscr{C}^\otimes \to \operatorname{End}(\mathscr{C})^\otimes $ is equivalent to $T$, 
    \item[$\star$] $M$ sends edges corresponding to \emph{convex morphisms}\footnote{See \cite[Definition 1.1.7]{LurieNonCommutativeAlgebra}, and also \cite[Lecture 2, Definition 2.54]{BrantnerDeformationTCC}.} $\alpha: [m] \to [n]$ in $\Delta$ such that $\alpha(m) = n$, to coCartesian edges for the fibration $\mathscr{C}^\otimes \to N(\Delta^\mathsf{op})$. 
\end{itemize}
\end{defn}
Informally \cite[Remark 3.1.4]{LurieNonCommutativeAlgebra} a module $M$ over a homotopy-coherent monad is an object $M \in \mathscr{C}$ equipped with a morphism $TM \to M$ which satisfies a version of the module axiom up to coherent homotopy. 

We would like to produce homotopy-coherent monads from adjunctions. The usual definition of an \emph{adjunction} in higher category theory uses not much data. That is, an adjunction consists of a pair of functors $F: \mathscr{D} \leftrightarrows \mathscr{C} : G$  of $\infty$-categories together with a unit $\eta: \operatorname{id} \to FG$, a counit $\varepsilon: FG \to \operatorname{id}$ and a 2-simplex expressing the zig-zag identity for the composite $F \to FGF \to F$, c.f. \cite[\href{https://kerodon.net/tag/02EJ}{Tag 02EJ}]{kerodon}. \emph{A priori}, it is not clear how to produce enough coherences in order to give the endofunctor $FG$ the structure of a homotopy-coherent monad. This motivates the definition of the $\infty$-category of \emph{adjunction data} \cite[\S 3.2]{LurieNonCommutativeAlgebra}. The $\infty$-category $\operatorname{ADat}(\mathscr{D},\mathscr{C})$ consists of certain sections of a certain fibration over a version of $\Delta^\mathsf{op}$ labelled by two colours; we direct the reader to \cite[Definition 3.2.6]{LurieNonCommutativeAlgebra} and \cite[Lecture 3]{BrantnerDeformationTCC} for details. This category keeps track of all the possible coherences implicit in an adjunction, so that every object of $\operatorname{ADat}(\mathscr{D},\mathscr{C})$ gives rise to a homotopy-coherent monad on $\mathscr{C}$, c.f. \cite[Remark 3.2.7]{LurieNonCommutativeAlgebra}. There is a functor $\operatorname{ADat}(\mathscr{D},\mathscr{C}) \to \operatorname{Fun}^L(\mathscr{D},\mathscr{C})$ from the $\infty$-category of adjunction data to the category of left-adjoint functors, in the previous sense. By Lurie's Theorem on the existence of adjunction data \cite[Theorem 3.2.10]{LurieNonCommutativeAlgebra}, this is a \emph{trivial Kan fibration}. That is, given any left-adjoint functor $F: \mathscr{D} \to \mathscr{C}$ one can, (up to contractible choice), choose a right adjoint $G$ such that the endofunctor $FG$ acquires the structure of a homotopy-coherent monad. Such a functor $G$ can then be factored canonically through the forgetful functor as $\mathscr{C} \xrightarrow[]{K} \operatorname{Mod}_{FG}\mathscr{D} \to \mathscr{D}$, c.f. \cite[\S 4.7]{HigherAlgebra}. Let us call the functor $K$ the \emph{comparison functor}.
\begin{thm}[Barr--Beck--Lurie] \cite[\S 4.7]{HigherAlgebra}
With notations as above. Assume that $G$ is conservative and that $G$ \emph{preserves geometric realizations of $G$-split simplicial objects}\footnote{See \cite[\S 4.7]{HigherAlgebra}.}. Then the comparison functor $K$ is an equivalence.
\end{thm}
The equivalence of categories in the Barr--Beck--Lurie theorem can be made more explicit in the following way:
\begin{lem}\label{lem:BarrBeckExplicit}
\begin{enumerate}[(i)]
    \item Let $F: \mathscr{C} \leftrightarrows \mathscr{D}: G$ be a comonadic adjunction of $\infty$-categories. The quasi-inverse to the comparison equivalence $\mathscr{C} \xrightarrow[]{\sim} \operatorname{Comod}_{FG}\mathscr{D}$, is given on objects by the formula 
    \begin{equation}
        M \mapsto \underset{[n] \in \Delta}{\operatorname{lim}} G (FG)^n M.
    \end{equation}
    \item Let $G: \mathscr{D} \leftrightarrows \mathscr{C} : F $ be a monadic adjunction of $\infty$-categories. The quasi-inverse to the comparison equivalence $\mathscr{C} \simeq \operatorname{Mod}_{FG}\mathscr{D}$, is given on objects by the formula 
    \begin{equation}
        M \mapsto \underset{[n]\in \Delta^\mathsf{op}}{\operatorname{colim}} G (FG)^n M.
    \end{equation}
\end{enumerate}
\end{lem}
\begin{proof}
    We will prove (ii), the proof of (i) is similar. We examine the equivalence of categories implicit in the proof of the Barr-Beck-Lurie theorem. Let us temporarily write $T$ for the monad $FG$. The Barr-Beck-Lurie equivalence arises from an adjunction
    \begin{equation}
        R : \mathscr{C} \leftrightarrows \operatorname{Mod}_T (\mathscr{D}) : L 
    \end{equation}
    in which $R$ satisfies 
    \begin{equation}\label{eq:forgetR}
       \operatorname{Forget}_T \circ R  \simeq F.  
    \end{equation}
    We would like to calculate the value of $L$ on objects. Every $M \in \operatorname{Mod}_T(\mathscr{D})$ is the colimit of simplicial object
    \begin{equation}
    M \simeq \underset{[n]\in \Delta^\mathsf{op}}{\operatorname{colim}} T^{n+1} M    
    \end{equation} 
    in $\operatorname{Mod}_T(\mathscr{D})$. Since $L$ is a left adjoint, it commutes with colimits, so it suffices to know the value of $L$ on free $T$-modules. However, by passing to left adjoints in \eqref{eq:forgetR}, we know that 
    \begin{equation}
        L \circ \operatorname{Free}_T \simeq G, 
    \end{equation}
    so $LM \simeq  \operatorname{colim}_{[n] \in \Delta^{\mathsf{op}}} G T^{n} M = \operatorname{colim}_{[n] \in \Delta^{\mathsf{op}}} G (FG)^{n} M$. This proves (ii). 
\end{proof}

\bibliography{references}
\bibliographystyle{alpha}
\Addresses

\end{document}